\pdfoutput=1

\documentclass[leqno,9pt]{amsart}

\usepackage{amsmath}

\usepackage{amssymb}
\usepackage{graphicx}

\newtheorem{thm}{Theorem}

\newtheorem{prop}[thm]{Proposition}
\newtheorem{cor}[thm]{Corollary}

\theoremstyle{definition}
\newtheorem{defn}{Definition}
\newtheorem{ex}{Example}

\theoremstyle{remark}
\newtheorem{remark}{Remark}

\numberwithin{equation}{section}

\def\R{\mathbb{R}}

\def\Z{\mathbb{Z}}

\def\O{\mathrm{O}}

\begin{document}

\title[]{Conformal geometry of timelike curves \\in the $(1+2)$-Einstein universe}

\author{Akhtam Dzhalilov}
\address{(A. Dzhalilov) Department of Natural and Mathematical Sciences , Turin Polytechnic
University in Tashkent, Niyazov Str. 17, 100095 Tashkent, Uzbekistan}
\email{a.dzhalilov@yahoo.com}

\author{Emilio Musso}
\address{(E. Musso) Dipartimento di Scienze Matematiche, Politecnico di Torino,
Corso Duca degli Abruz\-zi 24, I-10129 Torino, Italy}
\email{emilio.musso@polito.it}

\author{Lorenzo Nicolodi}
\address{(L. Nicolodi) Di\-par\-ti\-men\-to di Ma\-te\-ma\-ti\-ca e Informatica,
Uni\-ver\-si\-t\`a degli Studi di Parma, Parco Area delle Scienze 53/A, I-43124 Parma, Italy}
\email{lorenzo.nicolodi@unipr.it}

\thanks{Authors partially supported by
PRIN 2010-2011 ``Variet\`a reali e complesse: geometria, to\-po\-lo\-gia e analisi ar\-mo\-ni\-ca'';
FIRB 2008 ``Geometria Differenziale Complessa e Dinamica Olomorfa'';
the GNSAGA of INDAM; and by TWAS grant 14-121-RG/MATHS/AS-G UNESCOFR: 324028604}





\begin{abstract}
We study the conformal geometry of timelike curves in the $(1+2)$-Einstein universe,
the conformal compactification of Minkowski 3-space defined as the quotient
of the null cone of $\R^{2,3}$ by the action
by positive scalar multiplications.
The purpose is to describe local and global conformal invariants of timelike curves and
to address the question of existence and properties
of closed trajectories for the conformal strain functional.
Some relations between the conformal geometry of
timelike curves and the geometry of knots and links
in the 3-sphere are discussed.
\end{abstract}

\maketitle

\section{Introduction}\label{s:intro}

Conformal Lorentzian geometry has played an important role in general relativity since the work
of H. Weyl \cite{W}. In the 1980s, it has been at the basis of the development of twistor approach
to gravity by Penrose and Rindler \cite{PR} and it is one of the main ingredients in the recently
proposed cyclic cosmological models in general relativity \cite{Pn1,Pn3, Tod}. It also plays a role
in the regularization of the Kepler problem \cite{GS3, Mar, Mi, Mos}, in conformal field theory
\cite{Scho}, and in Lie sphere geometry \cite{Blaschke, Ce}.
%
%
%
For what concerns in particular the geometry of curves,
while the subject of conformal Lorentzian invariants of null curves has received some
attention \cite{Blaschke, BCDG, Ur}, that of timelike curves
seems to have been little studied before.
%

\vskip0.1cm

In this paper we investigate the geometry of timelike curves in $\mathcal{E}^{1,2}$, the conformal compactification of Minkowski 3-space defined as the space of oriented null lines of $\R^{2,3}$ through the origin. This study is intended as a preliminary step to understand the 4-dimensional case, which is that of physical interest. Despite some formal similarities, there are substantial differences between the conformal Riemannian and Lorentzian case: the Lorentzian space $\mathcal{E}^{1,2}$ has the topology of $S^1\times S^2$, which is not simply connected; the global Lorentzian metrics on $\mathcal{E}^{1,2}$ are never maximally symmetric; the universal covering of the conformal group of $\mathcal{E}^{1,2}$ does not admit finite dimensional representations and has a center which is discrete, but not finite \cite{Ba, Raw}. Following \cite{BCDG,Fr}, we call $\mathcal{E}^{1,2}$ the $(1+2)$-Einstein universe.\footnote{Actually, $\mathcal{E}^{1,2}$ is the double covering of the space that in \cite{BCDG,Fr} is called Einstein universe.} A motivation for this terminology is that the universal covering of $\mathcal{E}^{1,2}$, $\R\times S^2$, endowed with the product metric $-dt^2 + g_{S^2}$, provides a static solution of Einstein's equation with a positive cosmological constant. This solution was proposed by Einstein himself as a model of a closed static universe in \cite{Einstein}. In addition, the pseudo-Riemannian geometries of the standard Friedmann--Lemaitre--Robertson--Walker cosmological models can be realized as subgeometries of $\mathcal{E}^{1,2}$ \cite{HE}.

\vskip0.1cm

The purposes of this paper are threefold. The first is to describe local and global conformal differential
invariants of a timelike curve. The second purpose is to address the question of existence and properties
of closed trajectories for the variational problem defined by the conformal strain functional,
the Lorentzian analogue of the conformal arclength functional in M\"obius geometry \cite{MMR, M1, M2, MN}.
The Lagrangian of the strain functional depends on third-order jets and shares many similarities with
the relativistic models for massless or massive particles based on higher-order action functionals, a topic
which has been much studied over the past twenty years \cite{FGL, GM, KP, MN1, MN2, NFS2, NMMK, P}.
The last purpose is to establish a connection between the conformal geometry of timelike curves
in the Einstein universe and the geometry of transversal knots in the unit 3-sphere.

From a physical point of view, the relevant objects are the lifts of timelike curves to the universal
covering of $\mathcal{E}^{1,2}$ and their global conformal invariants.
The compactified model has the advantage of having a matrix group as its restricted conformal group,
which simplifies the use of the geometric methods based on the transformation group and eases
the computational aspects.


\vskip0.1cm
The paper is organized as follows.
In Section \ref{s:pre}, we collect some background material about
conformal Lorentzian geometry. For the geometry of the Einstein universe, we mainly follow \cite{BCDG}.

In Section \ref{s:TLcurves}, we study the conformal geometry of timelike curves in the Einstein universe.
We define the infinitesimal conformal strain, which is the Lorentzian analogue of the conformal arc
element of a curve in $S^3$ \cite{LO2010, Mon, M1, M2, MN}, and the notion of a conformal vertex.
An explicit description of curves all of whose points are vertices is given in Proposition 1.
Next, we define the concept of osculating conformal cycle and give a geometric characterization
of conformal vertices in terms of the analytic contact between the curve and its osculating cycle
(Proposition 2). We then prove the existence of a canonical conformal frame field along a generic
timelike curve (i.e., a timelike curve without vertices) and define the two conformal curvatures,
which are the main local conformal invariants of a generic curve (Theorem 3). As a byproduct,
some elementary consequences are derived (Propositions 4, 5 and 6 and Corollary 7).

In Section \ref{s:constant curvatures}, the canonical conformal frame is used to investigate generic
timelike curves with constant conformal curvatures. We exhibit explicit parameterizations of such
curves in terms of elementary functions (Theorems 8 and 9) and discuss their main geometric properties.
These are the Lorentzian counterparts of analogous results for curves with constant conformal
curvatures in $S^3$ \cite{SS}.

In Section \ref{s:strainfunctional}, we use the canonical frame to compute the Euler-Lagrange equations
of the conformal strain functional (Theorem \ref{t:EulerLagrange}). Consequently we show that the
conformal curvatures of the extrema can be expressed in terms of Jacobi's elliptic functions.
As a byproduct, we show that the conformal equivalence classes of critical curves depend on two real
constants and prove that there exist countably many distinct conformal equivalence classes of closed
trajectories for the strain functional (Theorem 11).

In Section \ref{s:curves-knots},
we establish a connection between the conformal geometry of timelike curves and the
geometry of transversal knots in the $S^3$
via the directrices of a generic timelike curve. These
are immersed curves in $S^3$, everywhere transverse to the canonical contact distribution, which
%
%
are built using the symplectic lift of the canonical conformal frame. If the directrices of a
generic closed timelike curve
are simple, then their linking and Bennequin numbers \cite{FT} provide three global conformal invariants,
different in general from the Maslov index of the curve.
These invariants are computed for the directrices
of a special class of closed timelike curves with constant conformal curvatures
(Proposition 12).
It is still an open question how the local symplectic invariants \cite{AD,CW} of the directrices can be
related to the strain and the conformal curvatures.
%
%
A more difficult problem is to understand
how the classical and non-classical invariants of transversal knots of $S^3$ \cite{Et, Et2, EtHo, FT}
are related to the conformal geometry of closed timelike curves.

\section{Conformal Lorentzian geometry}\label{s:pre}

\subsection{The automorphism group $\mathrm{A}^{\uparrow}_+(2,3)$}

Consider $\R^5$ with a nondegenerate scalar product $\langle \cdot,\cdot\rangle$ of signature $(2,3)$ and
a volume form $dV\in \Lambda^5(\R^5)$. Let $\mathcal{N}^2$ denote the cone of all isotropic
bivectors, i.e., the non-zero decomposable bivectors $X\wedge Y\in \Lambda_2(\R^5)$, such that
$\langle X,X\rangle = \langle Y,Y\rangle = \langle X,Y\rangle = 0$.
Choose an oriented spacelike
3-dimensional vector subspace $\mathbb{V}^3_+\subset \R^5$ and a positive-oriented orthogonal
basis $V_1,V_2,V_3$ of $\mathbb{V}^3_+$. We call an isotropic bivector $X\wedge Y$
{\it future directed} if $dV(X,Y,V_1,V_2,V_3)>0$ and denote by $\mathcal{N}^2_+$ the half-cone of all
future-directed isotropic bivectors. Let $\R^{2,3}$ denote $\R^5$
with the scalar product $\langle \cdot,\cdot\rangle$, the orientation induced by $dV$, and the
time-orientation determined by the half-cone $\mathcal{N}^2_+$.
The 10-dimensional Lie group $\mathrm{A}^{\uparrow}_+(2,3)$ of linear isometries of $\R^{2,3}$
preserving the given orientation and
time-orientation is referred to as the {\it automorphism group} of $\R^{2,3}$.
Its Lie algebra is the vector space
\[
 \mathfrak{a}(2,3)=\{f\in \mathrm{End}(\R^{2,3})\,:\, \langle f(X),Y\rangle
  + \langle X,f(Y)\rangle = 0, \, \forall \,X,Y\in \R^{2,3}\},
  \]
equipped with the commutator as a Lie bracket. Given a basis $\mathcal{B}=(B_0,B_1,B_2,B_3,B_4)$ of $\R^{2,3}$
and an endomorphism $f\in \mathrm{End}(\R^{2,3})$,
let $M_{\mathcal{B}}(f)$ be the $5\times 5$ matrix representing $f$ with respect to $\mathcal{B}$. Similarly, let $G_{\mathcal{B}}$ be the symmetric matrix representing the scalar product $\langle \cdot,\cdot \rangle$ with respect to $\mathcal{B}$. For every choice of $\mathcal{B}$, the map $\chi_{\mathcal{B}}:F\in \mathrm{A}^{\uparrow}_+(2,3)\mapsto M_{\mathcal{B}}(F)\in \mathrm{GL}(5,\R)$ is a faithful matrix representation of $\mathrm{A}^{\uparrow}_+(2,3)$. We say that:
\begin{itemize}

\item $\mathcal{B}$ is a {\it M\"obius basis} if $\mathcal{B}$ is positive-oriented,
$G_{\mathcal{B}}$ $=$
$-(E^0_4+E^4_0)-E^1_1+E^2_2+E^3_3 =: \mathtt m$, $\mathtt m =(\mathtt m_{ij})$, and the isotropic bivector $B_0\wedge (B_1+B_2)$
is future-oriented;\footnote{$E^a_b$, $0\leq a,b\leq 4$, denotes the elementary matrix with 1
in the $(a,b)$ place and $0$ elsewhere.}

\item $\mathcal{B}$ is a {\it Poincar\'e basis} if $\mathcal{B}$ is positive-oriented,
$G_{\mathcal{B}}=-E^0_0-E^1_1+E^2_2+E^3_3+E^4_4$, and the isotropic bivector $(B_0-B_4)\wedge(B_1+B_2)$ is future-oriented;

\item $\mathcal{B}$ is a {\it Lie basis} if $\mathcal{B}$ is positive-oriented,
$G_{\mathcal{B}}=-(E^0_4+E^4_0)-(E^1_3+E^3_1)+E^2_2$, and the isotropic bivector $B_0\wedge B_1$ is future-oriented.
\end{itemize}

We choose and fix a reference M\"obius basis $\mathbf{M}^{o}=(M^{o}_0,\dots,M^{o}_4)$ of $\R^{2,3}$,
referred to as the {\it standard M\"obius basis}. The image of $\mathrm{A}^{\uparrow}_+(2,3)$ under the faithful representation determined by $\mathbf{M}^{o}$ is a connected closed subgroup of $\mathrm{SL}(5,\R)$, denoted by $M^{\uparrow}_+(2,3)$. Let
\begin{equation}\label{itertwining}
 \begin{cases}
  T_\mathrm{mp}=\frac{1}{\sqrt{2}}(E^0_0+E^4_0)+E^1_1+E^2_2+E^3_3+\frac{1}{\sqrt{2}}(E^4_4-E^0_4),\\
  T_\mathrm{m \ell}=E^0_0-\frac{1}{\sqrt{2}}(E^1_1 +E^1_3)+E^2_2+\frac{1}{\sqrt{2}}(E^3_1-E^3_3)+E^4_4.
   \end{cases}
\end{equation}
Then, $\mathcal{B}$ is a M\"obius basis if and only if $\mathcal{B}\cdot T_\mathrm{mp}$ is a Poincar\'e basis
if and only if $\mathcal{B}\cdot T_\mathrm{m\ell}$ is a Lie basis; in particular,
$\mathbf{P}^{o}= (P^o_0,\dots,P^o_4)= \mathbf{M}^{o}\cdot T_\mathrm{mp}$
is referred to as the {\it standard Poincar\'e basis} and
$\mathbf{L}^{o}=(L^o_0,\dots,L^o_4) = \mathbf{M}^{o}\cdot T_\mathrm{m\ell}$
the {\it standard Lie basis} of $\R^{2,3}$. Differentiating the $\R^{2,3}$-valued maps
 $\mathcal{M}_j:F\in \mathrm{A}^{\uparrow}_+\mapsto F (M^o_j)\in \R^{2,3}$, $j=0,\dots,4$, yields $d\mathcal{M}_j=\sum_i\mu_j^i \mathcal{M}_i$, where $\mu^i_j$ are left-invariant 1-forms.
The conditions $\langle \mathcal{M}_j,\mathcal{M}_i\rangle= \mathtt{m}_{ji}$ imply that $\mu =(\mu^i_j)$
takes values in the Lie algebra
$\mathfrak{m}(2,3)$ of $M^{\uparrow}_+(2,3)$.
%
%
Choosing
\[
 \begin{array}{lllll}
 M^0_0=E^0_0-E^4_4,\! & M^0_1=E^0_1-E^1_4,\! & M^0_2=E^0_2+E^2_4,\! &  M^0_3=E^0_3+E^3_4,\! &  M^1_2=E^1_2+E^2_1,\\
  M^1_3=E^1_3+E^3_1,\! &M^2_3=E^2_3-E^3_2,\! &  M^4_1=E^4_1-E^1_0,\! &  M^4_2=E^4_2+E^2_0,\! & M^4_3=E^4_3+E^3_0\\
   \end{array}
   \]
as a basis of $\mathfrak{m}(2,3)$, we can write
\[
\begin{split}
 \mu&=\mu^0_0M^0_0+\mu^1_0M^0_1+\mu^2_0M^0_2+\mu^3_0M^0_3+\mu^2_1M^1_2\\
 &\quad +\mu^3_1M^1_3+\mu^3_2M^2_3+
   \mu^1_4M^4_1+\mu^2_4M^4_2+\mu^3_4M^4_3,
   \end{split}
   \]
where the left-invariant 1-forms $\mu^0_0$, $\mu^1_0$, $\mu^2_0$, $\mu^3_0$, $\mu^2_1$, $\mu^3_1$, $\mu^3_2$, $\mu^1_4$, $\mu^2_4$, $\mu^3_4$
are linearly independent and span $\mathfrak{a}(2,3)^*$, the dual of the Lie algebra $\mathfrak{a}(2,3)$.
If $F\in \mathrm{A}^{\uparrow}_+(2,3)$, the vectors $\mathcal{M}_0(F),\dots,\mathcal{M}_4(F)$ constitute
a M\"obius basis of $\R^{2,3}$, whose dual basis is denoted by  $\mathcal{M}^0(F),\dots,\mathcal{M}^4(F)$.
The Maurer-Cartan form of $\mathrm{A}^{\uparrow}_+(2,3)$ can be written as $\mu = \sum_{i,j}\mu^i_j\mathcal{M}^j\otimes \mathcal{M}_i$.
It satisfies the Maurer-Cartan equations $d\mu = -\frac{1}{2}[\mu,\mu]$.

\begin{figure}[ht]
\begin{center}
\includegraphics[height=6.2cm,width=6.2cm]{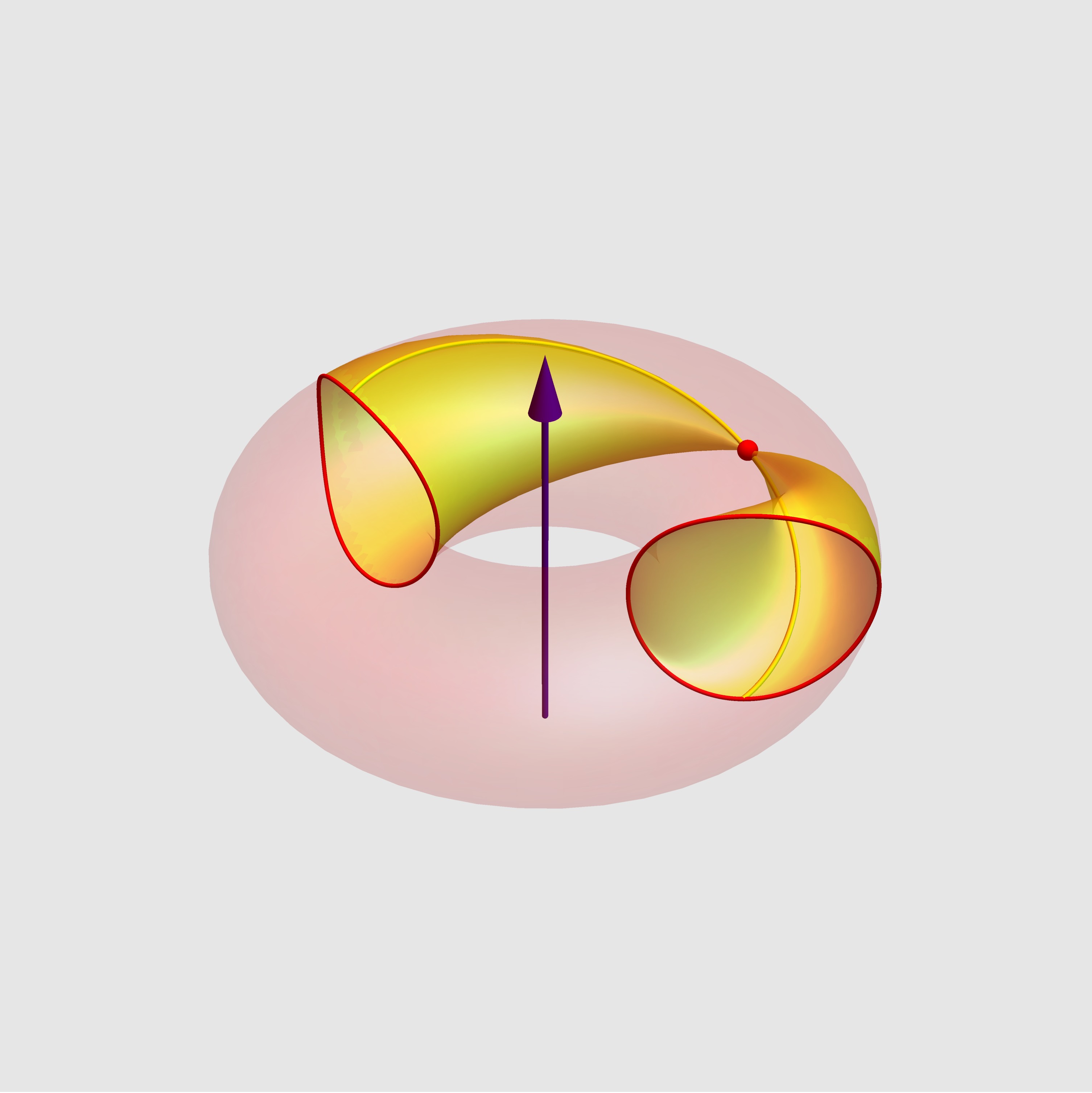}
\includegraphics[height=6.2cm,width=6.2cm]{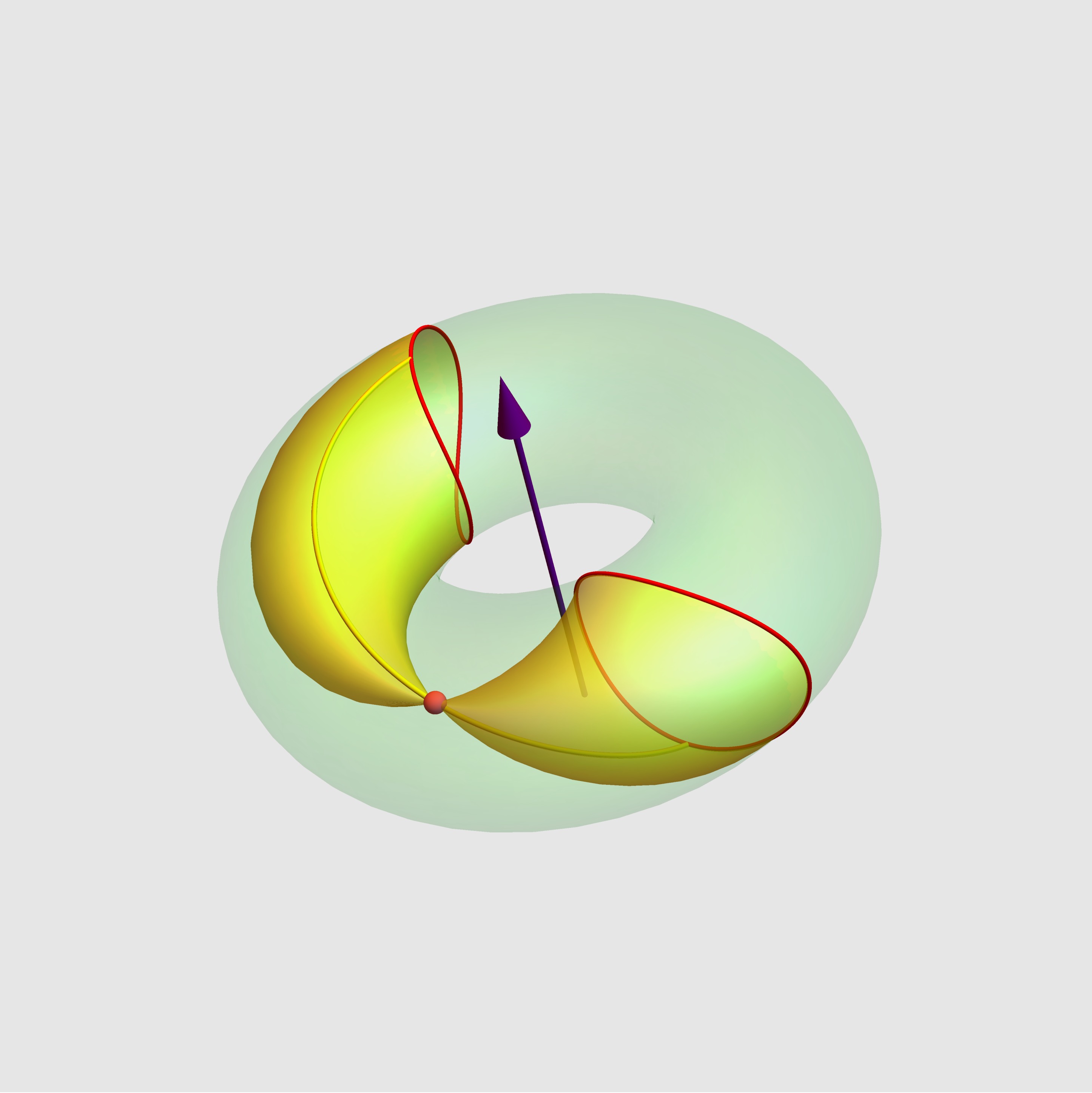}
\caption{The two adS chambers and a null light-cone. The Einstein universe is obtained by identifying the two boundaries. The wall, i.e., the surface given by the identification of the two boundaries, is a totally
umbilical torus of signature $(1,1)$.}\label{FIG1}
\end{center}
\end{figure}

\subsection{The $(1+2)$-Einstein universe and its conformal group}

Consider the orthogonal direct sum decomposition $\R^{2,3}=\mathbb{V}^2_-\oplus \mathbb{V}_+^3$
into the oriented, negative definite 2-dimensional subspace
$\mathbb{V}_-^2 =[P^o_0\wedge P^o_1]$ and the oriented 3-dimensional spacelike subspace
$\mathbb{V}_+^3=[P^o_2\wedge P^o_3\wedge P^o_4]$.
Let $(x_0,\dots,x_4)$ be the Cartesian coordinates of the standard Poincar\'e basis and denote
by $\mathcal{E}^{1,2}$ the 3-dimensional submanifold of $\R^{2,3}$ defined by the equations $x_0^2+x_1^2=1$ and $x_2^2+x_3^2+x_4^2=1$. As a manifold, $\mathcal{E}^{1,2}$ is the Cartesian product $S^1_-\times S^2_+$ of the unit circle of $\mathbb{V}^2_-$ with the 2-dimensional unit sphere of $\mathbb{V}^3_+$. The scalar product on $\R^{2,3}$ induces a Lorentzian pseudo-metric $g_{\textsl{e}}$ on $\mathcal{E}^{1,2}$. The normal bundle of $\mathcal{E}^{1,2}$ is spanned by the restrictions of the vector fields $\mathbf{n}_1=x_0\partial_{x_0}+x_1\partial_{x_1}$ and $\mathbf{n}_2=x_2\partial_{x_2}+x_3\partial_{x_3}+x_4\partial_{x_4}$. Thus, contracting $dV$ with $\mathbf{n}_1$ and $\mathbf{n}_2$, we get a volume form on $\mathcal{E}^{1,2}$ which in turn defines an orientation on  $\mathcal{E}^{1,2}$.
The vector field $-x_1\partial_{x_0}+x_0\partial_{x_1}$ is tangent to $\mathcal{E}^{1,2}$ and induces a nowhere vanishing timelike vector field $\mathbf{t}$ on $\mathcal{E}^{1,2}$.
%
%
We time-orient $\mathcal{E}^{1,2}$ by requiring that $\mathbf{t}$ is future-oriented.

\begin{defn}
The Lorentzian manifold $(\mathcal{E}^{1,2},g_{\textsl{e}})$, with the above specified orientation
and time-orientation,
is called the {\it $(1+2)$-Einstein universe}. The Einstein universe is a homogeneous Lorentzian manifold and its restricted isometry group is a 4-dimensional maximal compact subgroup of $\mathrm{A}^{\uparrow}_+(2,3)$, isomorphic to $SO(2)\times SO(3)$.
\end{defn}

For each non-zero vector $X\in \R^{2,3}$, we denote by $[X]$ the oriented line spanned by $X$ (i.e., the ray of $X$).
The set of all null rays, denoted by $\mathcal{M}^{1,2}$ is a manifold and the map
\begin{equation}\label{EUI1}
 \Phi :  \mathcal{E}^{1,2} \ni X \mapsto [X]\in \mathcal{M}^{1,2}
  \end{equation}
is a diffeomorphism. This allows us to identify $\mathcal{E}^{1,2}$ with $\mathcal{M}^{1,2}$ and
to transfer to $\mathcal{M}^{1,2}$ the oriented, time-oriented conformal Lorentzian structure
of $\mathcal{E}^{1,2}$.
%
%
We will make no distinction between the two models and
the context will make clear which of them is being used.
Using the above identification, the automorphism group $\mathrm{A}^{\uparrow}_+(2,3)$ acts effectively and transitively on the left of $\mathcal{E}^{1,2}\cong \mathcal{M}^{1,2}$ by $F\cdot X =[F(X)]$, for each $F\in \mathrm{A}^{\uparrow}_+(2,3)$ and $[X]\in \mathcal{M}^{1,2}$. The action preserves the oriented, time-oriented conformal Lorentzian structure of $\mathcal{E}^{1,2}$.
It is a classical result that every restricted conformal transformation of the Einstein universe is induced by a unique element of $\mathrm{A}^{\uparrow}_+(2,3)$ \cite{DNF,Fr}. For this reason, we call $\mathrm{A}^{\uparrow}_+(2,3)$ the (restricted) {\it conformal group of the Einstein universe}. The map
\[
 \tau : \mathcal{E}^{1,2}\ni X \mapsto (x_0 x_2-x_1(x_3+2),x_1x_2+x_0(x_3+2),x_4)\in \R^3
  \]
is a 2:1 branched covering onto the toroid $\mathrm{T}\subset \R^3$ swept out by the rotation around the $z$-axis of the unit disk in the $xz$-plane centered at $(2,0,0)$. Thus, the Einstein universe can be identified with the quotient space $\mathrm{T}\sqcup \mathrm{T}/\sim$ of the disjoint union
$\mathrm{T}\sqcup \mathrm{T}=\mathrm{T}\times \{-1,1\}$ of two copies of the toroid $\mathrm{T}$ modulo
the equivalence relation defined by $[(P,\epsilon)]_{\sim}=\{(P,\epsilon)\}$, if $P\in \mathrm{Int}(\mathrm{T})$, and $[(P,\epsilon)]_{\sim}=\{(P,\pm \epsilon)\}$, if $P\in \partial \mathrm{T}$.
In what follows we will use the ``toroidal'' projection $\tau$ to visualize and clarify the geometrical
content of the results.

\subsection{Conformal embeddings of Lorentzian space forms}

As a model for {\em anti-de Sitter 3-space}, we
consider the hyperquadric of $\R^4$
\[
 \mathbb{M}^{(1,2)}_{-1}=\{(x,y)\in \R^{4} \,:\, -(x^1)^2-(x^2)^2+(y^1)^2+(y^2)^2=-1\}
  \]
equipped with the Lorentzian structure induced by the neutral scalar product
\[
  (\mathbf{x},\widetilde{\mathbf{x}})_{(2,2)}=-x^1 \widetilde{x}^1-x^1 \widetilde{x}^1+y^1 \widetilde{y}^1+y^2 \widetilde{y}^2.
  \]
Then, $\mathbb{M}^{(1,2)}_{-1}$ can be embedded in the Einstein universe by the conformal map
\[
 \mathbf{j}_{adS}:(x,y)
 \mapsto \frac{1}{\sqrt{(x^1)^2+(x^2)^2}}\left(x^1P^o_0+x^2P^o_1+P^o_2+ y^1P^o_3+y^2P^o_4\right).
  \]
The image is the open subset $\mathcal{T}_{+}=\{X\in \mathcal{E}^{1,2} \,:\, \langle X,P^o_2 \rangle  >0\}$,
the {\it positive adS-chamber}. The boundary of $\mathcal{T}_{+}$ is the {\it adS-wall}, i.e.,
the timelike embedded torus
\[
   \mathcal{T}_{\infty}=\{X\in \mathcal{E}^{1,2} \,:\, \langle X,P^o_2 \rangle =0\}\subset \mathcal{E}^{1,2}.
   \]
The complement of $\mathcal{T}_{+}\cup \mathcal{T}_{\infty}$ is the {\it negative adS-chamber} $\mathcal{T}_{-}$, another copy of the anti-de Sitter space inside $\mathcal{E}^{1,2}$. The restriction of the branched covering $\tau$ to each of the two adS-chambers is a smooth diffeomorhism onto the interior of $\mathrm{T}$, while the restriction of $\tau$ to the adS-wall is a diffeomorphism onto $\partial \mathrm{T}$ (see Figure \ref{FIG1}).

\begin{figure}[ht]
\begin{center}
\includegraphics[height=6.2cm,width=6.2cm]{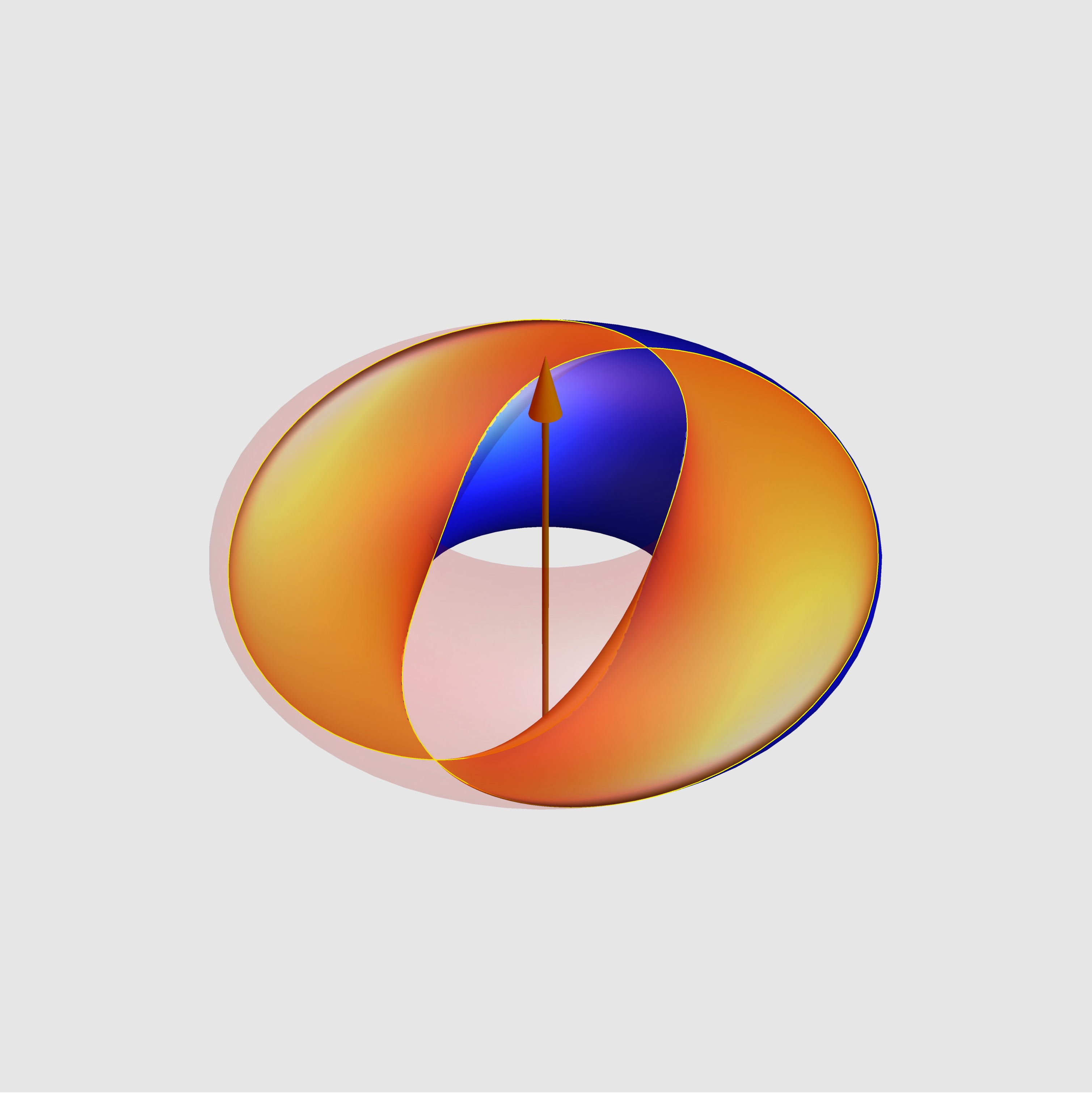}
\includegraphics[height=6.2cm,width=6.2cm]{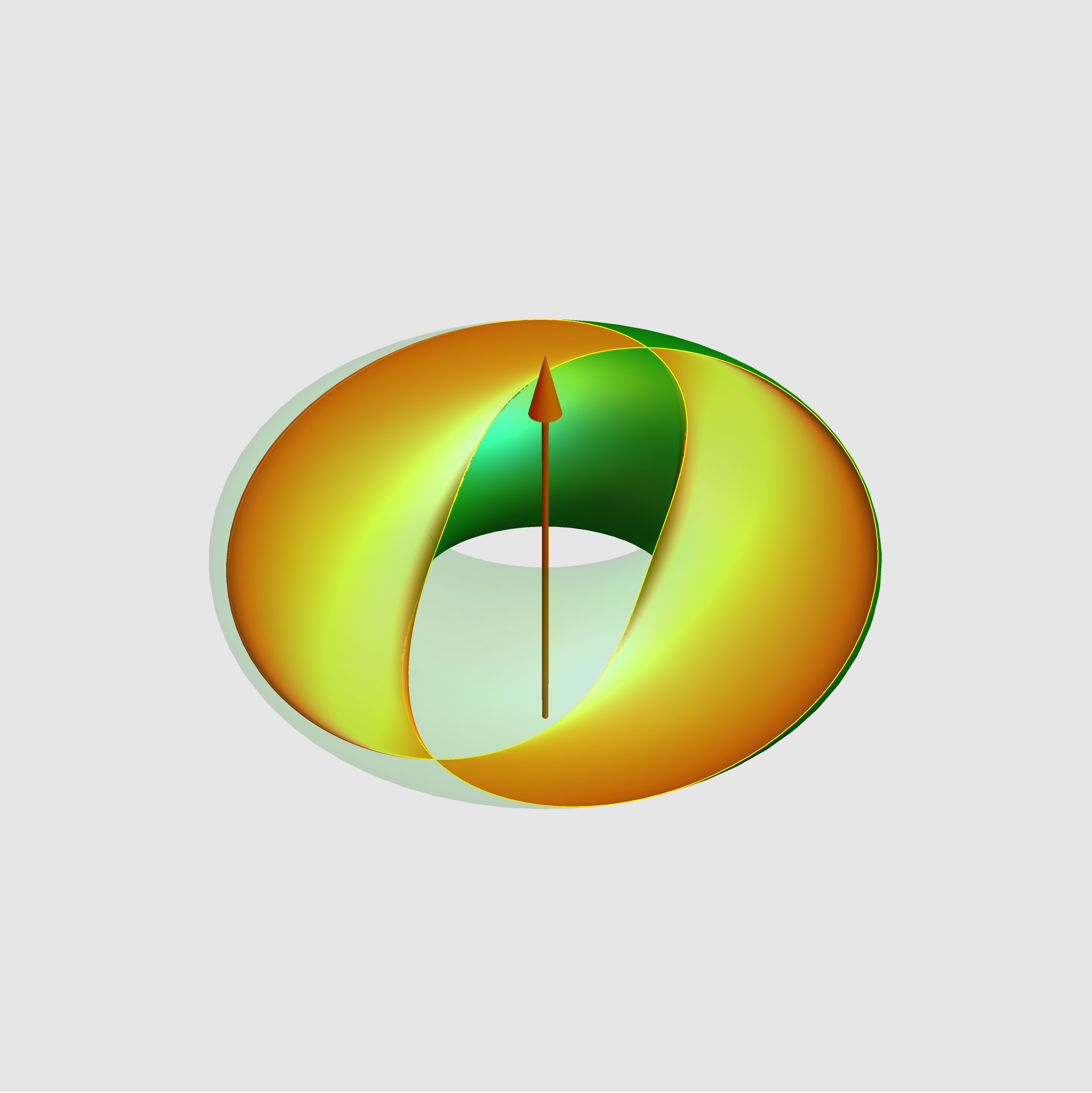}
\caption{The Minkowski chamber. Its boundary is a totally umbilical null cone of the Einstein universe. The two surfaces colored in blue and green are identified.}\label{FIG2}
\end{center}
\end{figure}

\noindent Let $\mathbb{M}^{(1,2)}_0$ be {\em Minkowski 3-space}, i.e., the affine space $\R^3$ with the Lorentzian scalar product
\[
  (\mathbf{x},\widetilde{\mathbf{x}})_{(1,2)}=-x^1\widetilde{x}^1+x^2\widetilde{x}^2+x^3\widetilde{x}^3.
   \]
For each point $\mathbf{x} = {}^t\!(x_1,x_2,x_3)\in \mathbb{M}^{(1,2)}_0$, let
$^*\mathbf{x} = (-x_1,x_2,x_3)$. The map
\[
\mathbf{j}_m(\mathbf{x})= \Phi ^{-1}\left(\left[M^o_0+x_1M^o_1+x_2 M^o_2+x_3 M^o_3 + \frac{{}^*\!\mathbf{x} \mathbf{x}}{2}M^o_4\right]\right)
\]
is a conformal embedding whose image, the {\it positive Minkowski-chamber}, is the open subset $\mathcal{M}_{+}=\{X\in \mathcal{E}^{1,2} \,:\, -\langle X,P^o_0 \rangle  >\langle X,P^o_4 \rangle \}$. Its boundary, the {\it Minkowski-wall}, is the light cone $\mathcal{M}_{\infty}=\{X\in \mathcal{E}^{1,2} \,:\, -\langle X,P^o_0 \rangle =\langle X,P^o_4 \rangle\}$. The complement of $\mathcal{M}_{+}\cup \mathcal{M}_{\infty}$ is the {\it negative Minkowski-chamber}, another copy of Minkowski space inside $\mathcal{E}^{1,2}$ (see Figure \ref{FIG2}).

\noindent Next, consider {\em de Sitter 3-space}, that is, the quadric
\[
 \mathbb{M}^{(1,2)}_1=\left\{ {}^t\!(w^1,w^2,w^3,w^4)\in \R^{1,3}\,:\,-(w^1)^2+(w^2)^2+(w^3)^2+(w^4)^2=1\right \}\subset \R^{4}
  \]
equipped with the Lorentzian structure induced by the scalar product
\[
 (\mathbf{w},\widetilde{\mathbf{w}})_{(1,3)} = -w^1\widetilde{w}^1+w^2\widetilde{w}^2+w^3\widetilde{w}^3+w^4\widetilde{w}^4.
  \]
Then, $\mathbb{M}^{(1,2)}_1$ is mapped into $\mathcal{M}^{1,2}$ by the conformal embedding
\[
  \mathbf{j}_{dS}:{}^t\!(w^1,w^2,w^3,w^4)
  \mapsto \frac{1}{\sqrt{1+(w^1)^2}}( P^o_0+w^1P^o_1+w^2P^o_2+w^3P^o_3+w^4P^o_4).
   \]
The image of $\mathbf{j}_{dS}$ is the {\it positive dS-chamber}
$\mathcal{S}_{+}=\{X= \sum_{j=0}^4 x^jP^o_j\in \mathcal{E}^{1,2} \,:\, x^0>0\}$. Its boundary is the disjoint union of two spacelike surfaces $\mathcal{S}^{\pm}_{\infty}=\{X=\sum_{j=0}^4 x^jP^o_j\in \mathcal{E}^{1,2} \,:\, x^0=0,\, x^1=\pm 1\}$,
the {\it dS-walls} of $\mathcal{E}^{1,2}$. Each wall is a totally umbilical 2-dimensional spacelike sphere embedded in $\mathcal{E}^{1,2}$. The complement of $\mathcal{S}_{+}\cup \mathcal{S}^+_{\infty}\cup \mathcal{S}^-_{\infty}$ is the {\it negative dS-chamber} $\mathcal{S}_{-}=\{X= \sum_{j=0}^4 x^jP^o_j\in \mathcal{E}^{1,2}\, :\, x^0<0\}$, which is another copy of de Sitter space inside the Einstein universe (see Figure \ref{FIG3}).

\begin{figure}[ht]
\begin{center}
\includegraphics[height=6.2cm,width=6.2cm]{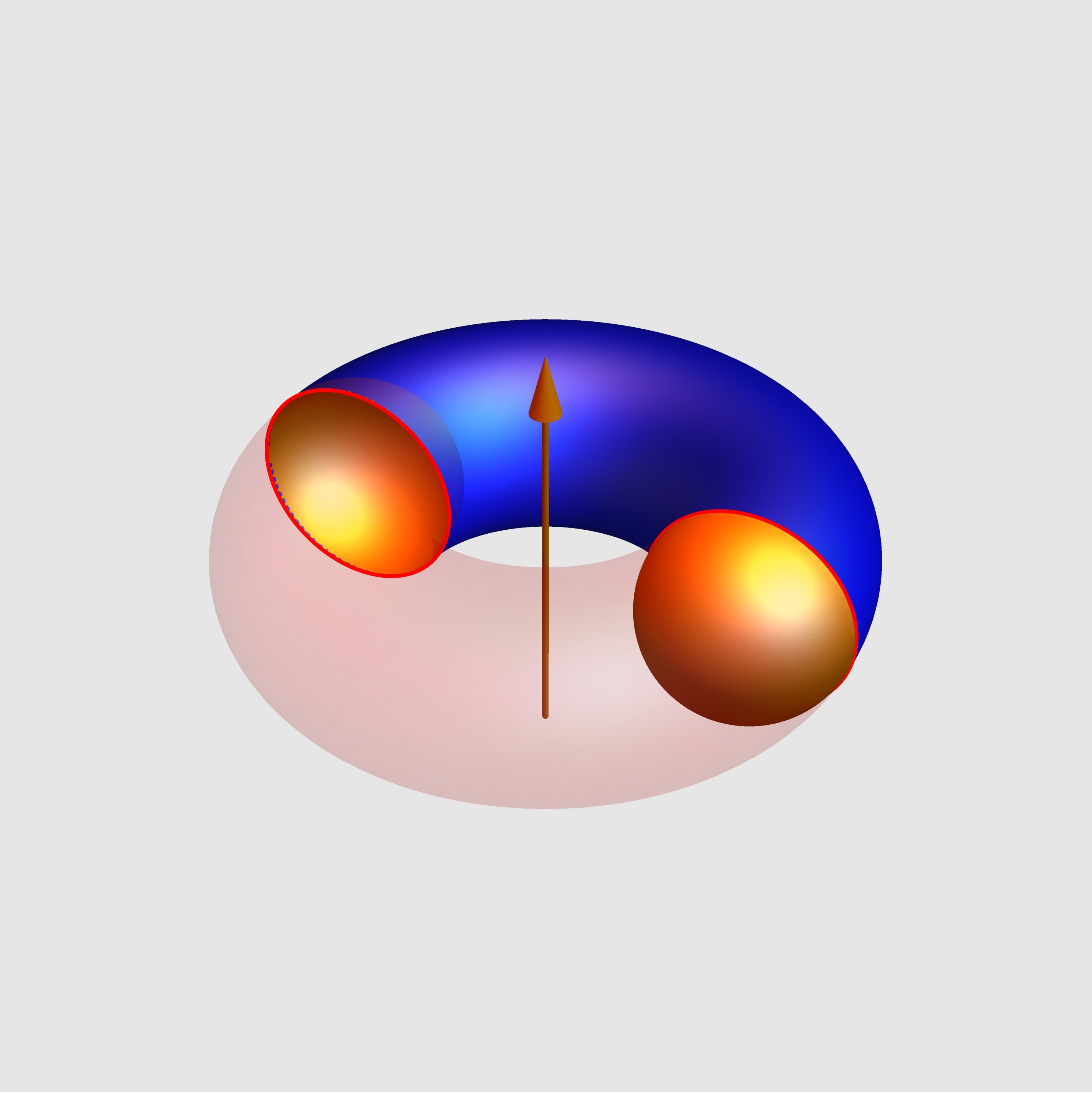}
\includegraphics[height=6.2cm,width=6.2cm]{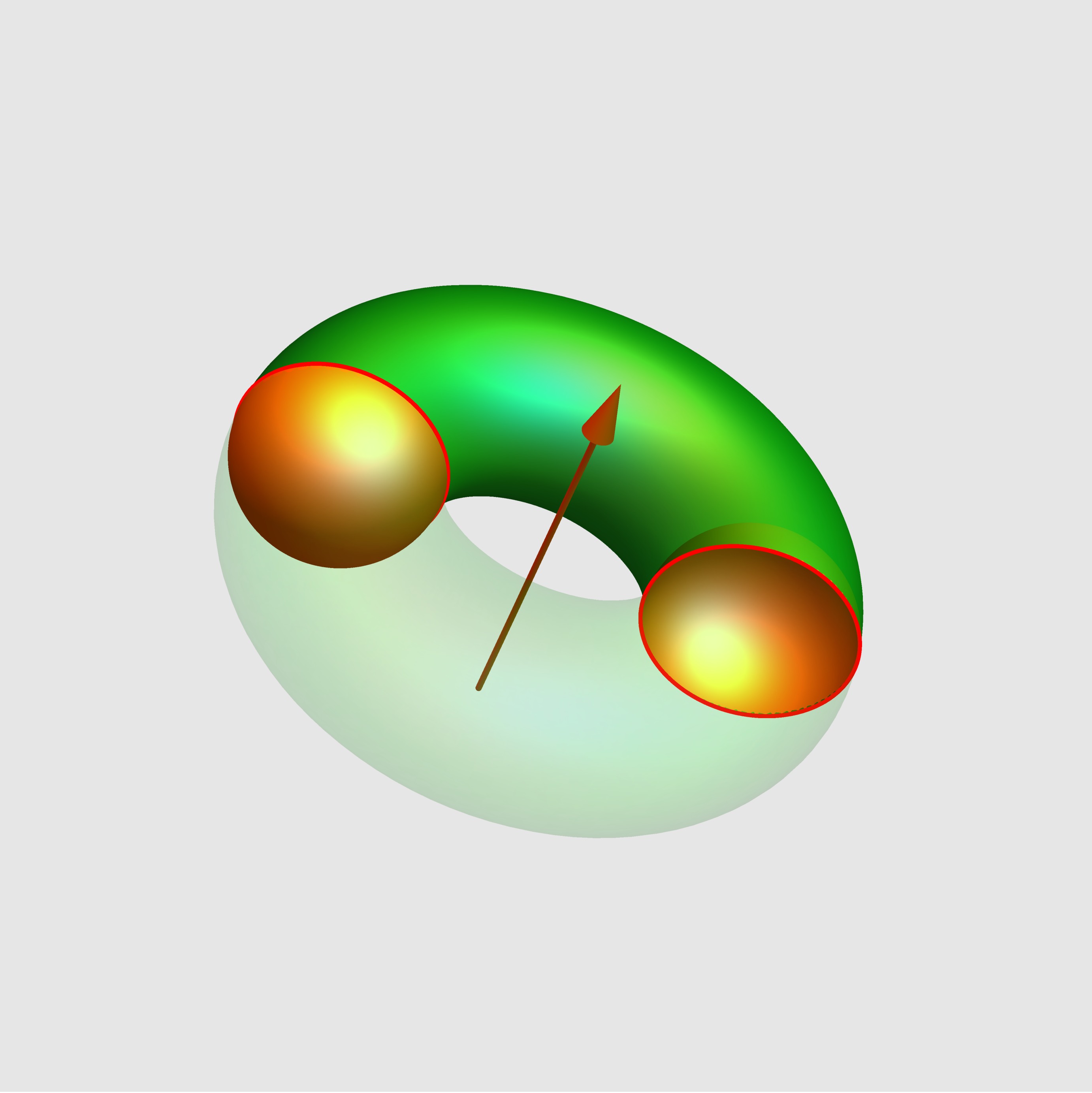}
\caption{The de Sitter chamber. The boundary consists of two disjoint spacelike totally umbilical spheres of the Einstein universe. The surfaces colored in blue and green are identified.}\label{FIG3}
\end{center}
\end{figure}

\section{Conformal geometry of timelike curves}\label{s:TLcurves}

Let $\gamma : {I}\subset \R\to \mathcal{E}^{1,2}$ be a parametrization of a smooth curve of
the Einstein universe. We write $\gamma=\eta+\beta$, where
$\eta : {I}\to S^1_-\subset \mathbb{V}^2_- = [P^o_0\wedge P^o_1]$ and $\beta : {I}\to S^2_+\subset \mathbb{V}^3_+=[P^o_2\wedge P^o_3 \wedge P^o_4]$ are smooth curves, referred to as the {\it time-} and
{\it space-component} of $\gamma$, respectively. In particular, $\gamma$ is timelike if and only if $- \langle \eta',\eta'\rangle > \langle \beta',\beta'\rangle \ge 0$ and is future-directed if and only if $\eta'=\varrho_{\eta}j\eta$, where $\varrho_{\eta}$ is a strictly positive function and $j:\mathbb{V}_-\to \mathbb{V}_-$ is the counterclockwise rotation of an angle $\pi/2$ in the oriented, negative definite Euclidean plane $\mathbb{V}^2_-$. Henceforth, we only consider future-directed timelike curves.

\begin{defn}\label{def:maslov}
The {\it Maslov index} of a closed timelike curve is the degree is its time-component.
\end{defn}

\begin{defn}
If $\gamma:I\to \mathcal{E}^{1,2}$ is a timelike curve, then $\gamma|_t,\gamma'|_t,\gamma''|_t$ span an oriented 3-dimensional vector subspace of $\R^{2,3}$ of signature $(-,-,+)$. Such a subspace, denoted by $\mathcal{T}|_t$, is called the {\it conformal osculating space} of $\gamma$ at $\gamma(t)$. Its orthogonal complement $\mathcal{N}|_t$ is the {\it normal conformal space} of $\gamma$ at $\gamma(t)$. The orthogonal projection onto the normal spaces is denoted by $\pi_{\mathcal{N}}|_t$. The continuous function
\begin{equation}\label{strain-density}
 \upsilon_{\gamma}=\sqrt[4]{\frac{\langle \pi_{\mathcal{N}}(\gamma'''),
  \pi_{\mathcal{N}}(\gamma''')\rangle}{|\langle \gamma',\gamma'\rangle|}}
  \end{equation}
is called the {\it strain density} of $\gamma$.
The exterior differential form $\sigma_\gamma:=\upsilon_{\gamma}dt$ is the {\it infinitesimal strain}.
In the proof of Theorem \ref{thm:canonical-frame} below, we will show that $\sigma_\gamma$
is invariant under the action of the conformal group and changes of parameter.
The {\it strain} is a function $u_{\gamma}:I\to \R$, such that $du_{\gamma}=\sigma_{\gamma}$. By construction, the strain is a non-decreasing function of class $C^1$, uniquely defined up to an additive constant. A point $\gamma(t_0)$ is called a {\it conformal vertex} if the infinitesimal strain vanishes at $t_0$. A timelike curve without conformal vertices is said {\it generic}. If $\upsilon_{\gamma}=0$, $\gamma$ is said a {\it conformal cycle}.
\end{defn}

\begin{remark}
The strain and the infinitesimal strain are dimensionless quantities. The infinitesimal strain is the Lorentzian analogue of the conformal arc element of a curve in the 3-dimensional round sphere \cite{LO2010,M1,MN}.
\end{remark}

\begin{defn}
Two timelike curves $\gamma:I\to \mathcal{E}^{1,2}$ and $\widetilde{\gamma}:\widetilde{I}\to \mathcal{E}^{1,2}$ are said to be {\it conformally equivalent} to each other if there exist a change of parameters $f:I\to \widetilde{I}$ and a restricted conformal transformation $F\in \mathrm{A}^{\uparrow}_+(2,3)$ such that $\gamma = F\cdot \widetilde{\gamma}\circ f$.
\end{defn}

\begin{prop}
A conformal cycle is equivalent to the centerline of the positive adS-chamber.
\end{prop}

\begin{proof}
It is quite easy to see that the normal and osculating spaces of a conformal cycle are constants.
By possibly applying a restricted conformal transformation, we may assume that the osculating space coincides with $[M^o_0,M^o_1,M^o_4]$. Then, using the identification of $\mathcal{E}^{1,2}$ with $\mathcal{M}^{1,2}$, we can write $\gamma=[x_0M^o_0+x_1M^o_1+x_4M^o_4]$, where $x_0$, $x_1$ and $x_4$ are smooth functions, such that $2x_0x_4-x_1^2=0$. Hence, there exists a smooth function $\phi:I\to \R$, such that
$x_0(t)=(\cos\phi(t)+1)/\sqrt{2}$, $x_1(t)=\sin \phi(t)$ and $x_4(t)=(\cos \phi(t)-1)/\sqrt{2}$.
Since $\gamma$ is a rank-one map, the derivative of $\phi$ is nowhere vanishing.
Taking $\phi$ as a new parameter, we obtain
$\gamma(\phi)=\cos(\phi)P^o_1+\sin(\phi)P^0_1$. This proves the result.
\end{proof}

\begin{remark}
A conformal cycle consists of the null rays belonging to the intersection of the null cone of $\R^{2,3}$ with a 3-dimensional linear subspace of type $(2,1)$. Therefore, the totality of conformal cycles can be identified with $G_{2,1}(\R^{2,3})$, the 6-dimensional Grassmannian of 3-dimensional timelike subspaces of type $(2,1)$. This also shows that any conformal cycle is invariant under a 3-dimensional group of restricted conformal transformations isomorphic to $\O^{\uparrow}_+(2,1)$.
\end{remark}

Given a timelike curve $\gamma :I\to \mathcal{E}^{1,2}$, the totality of null rays belonging
to $\mathcal{T}|_{t_0}$ is the {\it osculating cycle} of $\gamma$ at $\gamma(t_0)$. We give two characterizations of conformal vertices using osculating cycles. The proof of the results relies on simple computations and is omitted.

\begin{prop}Let $\gamma$ be a timelike curve. Then:
\begin{itemize}
\item The osculating cycle of $\gamma$ at $\gamma(t_0)$ has second order analytic contact with $\gamma$ at $\gamma(t_0)$. In addition, $\gamma(t_0)$ is a conformal vertex if and only if the order of analytic contact is strictly bigger than 2.

\item $\gamma(t_0)$ is a conformal vertex if and only if $t_0$ is a stationary point of the curve
\[
  I \ni t \mapsto [M_0|_{t_0},M_1|_{t_0},M_4|_{t_0}]\in G_{2,1}(\R^{2,3}).
  \]
\end{itemize}
\end{prop}

\begin{remark}{The conformal strain density of a curve at a point $\gamma(t)$ measures the infinitesimal distorsion of the curve from its osculating cycle.}\end{remark}

\begin{figure}[ht]
\begin{center}
\includegraphics[height=6.2cm,width=6.2cm]{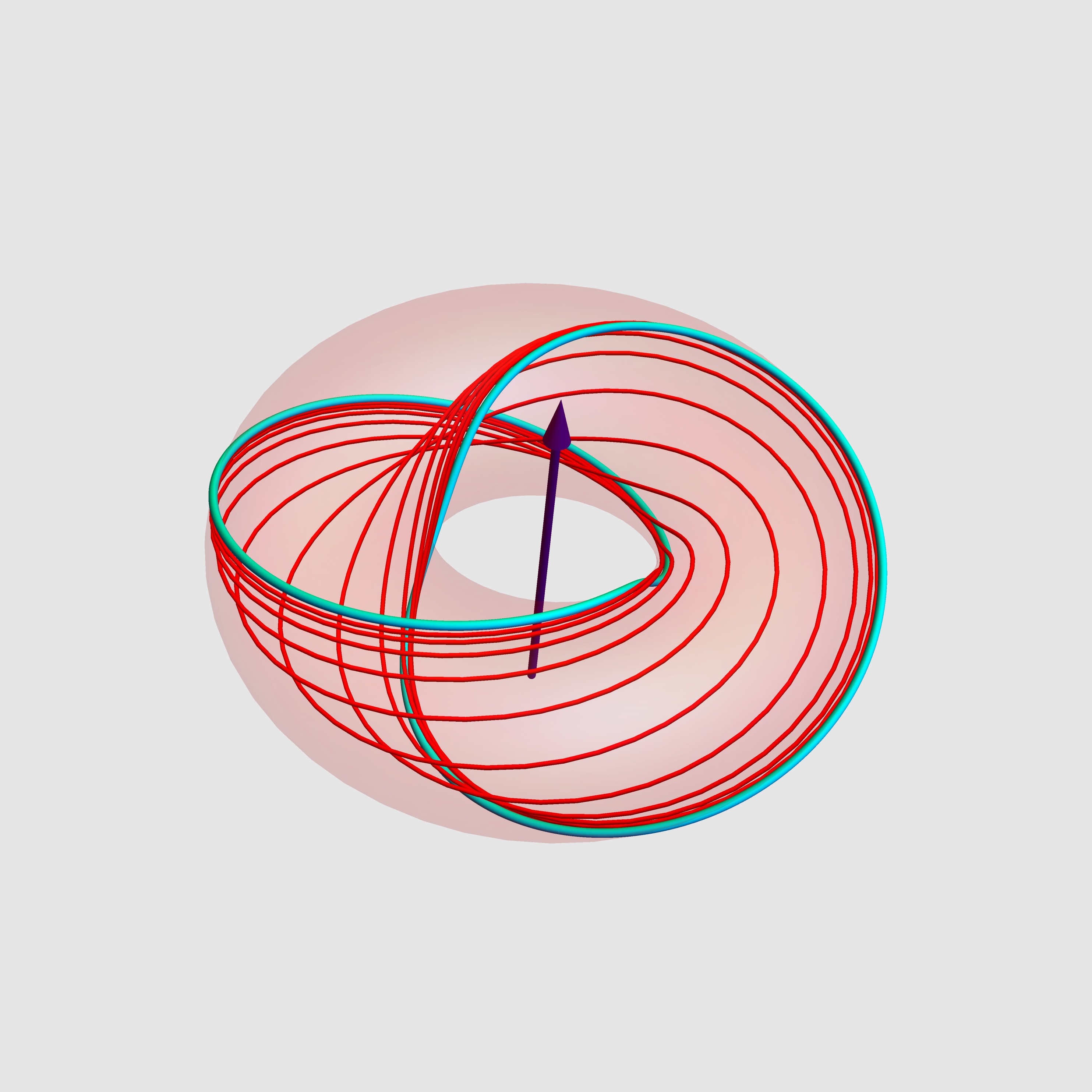}
\includegraphics[height=6.2cm,width=6.2cm]{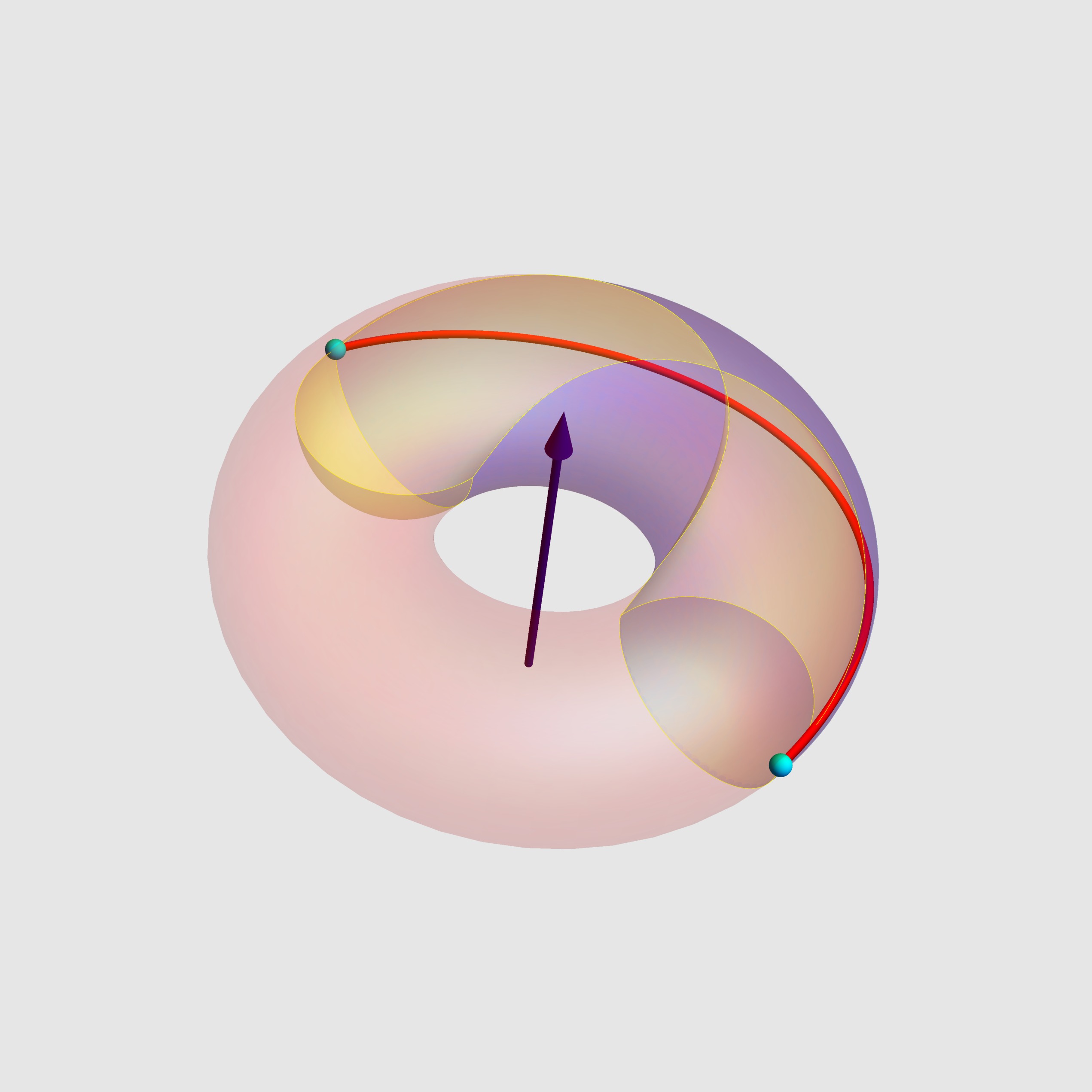}
\caption{Curves with constant curvatures of the Classes $C_1$ (on the left) and $C_3$ (on the right), with parameters $a=0.5$, $b=0.2$ and $a=1$, $b=1.42009$, respectively.}\label{FIG4C1C3}
\end{center}
\end{figure}

\subsection{The canonical conformal frame of a generic timelike curve}\label{ss:conf-fr}

Let $\mathcal{M}$ be the manifold of M\"obius frames of $\R^{2,3}$.
The map $\mathrm{A}^{\uparrow}_+(2,3)\to \mathcal{M}$, $F \mapsto (F({M}^o_0),\dots,F(M^o_4))$,
gives an explicit identification of the conformal group with the frame manifold $\mathcal{M}$.
Consider a timelike curve $\gamma:I\to \mathcal{M}^{1,2}$. A {\it conformal frame} along $\gamma$ is a lift $\mathbf{M}:I\to \mathcal{M}$ of $\gamma$ to $\mathcal{M}$, i.e., a smooth map from the open interval $I$ to $\mathcal{M}$, such that $M_0|_t=r(t)\cdot \gamma(t)$, where $r(t)>0$, for each $t\in I$.
Given a frame along $\gamma$, we put
$\mathbf{M}^*(\mu)=\mathfrak{m}dt$, where $\mathfrak{m}=(m^i_j)$
is a smooth map with values in the Lie algebra $\mathfrak{m}(2,3)$.
We can write
\[
\mathfrak{m}= \left(\begin{matrix}
m^0_0 & m^0_1 & m^0_2 & m^0_3 & 0\\
m^1_0 & 0 & m^1_2& m^1_3 & -m^0_1\\
 m^2_0 & m^1_2 & 0 & -m^3_2 & m^0_2\\
m^3_0 & m^1_3 & m^3_2 & 0 & m^0_3\\
0 & -m^1_0 & m^2_0 & m^3_0 & -m^0_0\\
 \end{matrix}\right).
 \]
%
%
Slightly abusing notation, we omit $\mathbf{M}^*$
and write $\mu^i_j=m^i_jdt$ for the pull-back of forms.

\begin{thm}\label{thm:canonical-frame}
Let $\gamma:I\to \mathcal{E}^{1,2}$ be a generic, future-oriented, timelike curve. Then there exists a
unique conformal frame $\mathbf{M}:I\to \mathcal{M}$ along $\gamma$,
the {\em canonical conformal frame}, such that
%
%
\[
  (\mu^i_j) =
  \left(\begin{matrix}
  0 & -h  &  1 & 0 & 0 \\
  1 &  0  &  0 & 0 & h   \\
  0 & 0  & 0 & -k  & 1 \\
  0 & 0 & k  & 0 & 0 \\
  0  & -1 & 0 & 0 & 0\\
   \end{matrix}\right)\sigma_{\gamma},
   \]
where $\sigma_{\gamma}$ is the infinitesimal strain of $\gamma$ and $h, k : I \to\R$ are smooth functions,
called the {\em conformal curvatures} of $\gamma$.

\end{thm}

\begin{proof}
If $\mathbf{M}$ and $\widehat{\mathbf{M}}$ are two conformal frames along $\gamma$, then $\widehat{\mathbf{M}}=\mathbf{M} \mathbf{X}$, where $\mathbf{X}$ is a smooth map taking values
in the subgroup
$M^{\uparrow}_+(2,3)_0=\{\mathbf{X}\in M^{\uparrow}_+(2,3) \,:\, X^1_0=X^2_0=X^3_0=X^4_0=0 \}$ of $M^{\uparrow}_+(2,3)$.
The elements of $M^{\uparrow}_+(2,3)_0$ can be written as
\[
\mathbf{X}(r,A,u)=\left(
                      \begin{array}{ccc}
                        r & {}^*\!u A & \frac{{}^*\!u u}{2r} \\
                        0 & A & u/r \\
                        0 & 0 & 1/r \\
                      \end{array}
                    \right),
                    \]
where $A\in O^{\uparrow}_+(1,2)$, $r>0$, $u= {}^t\!(u^1,u^2,u^3)$ and ${}^*\!u=(-u^1,u^2,u^3)$.
The two maps $\mathfrak{m}$ and $\widehat{\mathfrak{m}}$ are related by
\begin{equation}\label{CG}
 \widehat{\mathfrak{m}} = \mathbf{X}^{-1}  \mathbf{X}'+ \mathbf{X}^{-1}  \mathfrak{m}  \mathbf{X}.
  \end{equation}
A conformal frame along $\gamma$ is of {\it first order} if $M_0\wedge M_0'$ is a positive multiple of $M_0\wedge M_1$. It is easily seen that first order frames do exist along any timelike curve. If $\mathbf{M}:I\to \mathcal{M}$ is a first order conformal frame along a future-directed timelike curve, then $m^1_0>0$ and $\mu^2_0=\mu^3_0=0$.
If $\mathbf{M}$ is a first order conformal frame, then any other is given by $\widetilde{\mathbf{M}}=\mathbf{M} \mathbf{X}$, where $\mathbf{X}$ is a smooth map with values in the subgroup
\begin{equation}\label{M1}M^{\uparrow}_+(2,3)_1=\{X\in M^{\uparrow}_+(2,3)_0 \,:\, X^1_1=1,X^2_1=X^3_1=X^4_1=0\}\end{equation}
of $M^{\uparrow}_+(2,3)_0$. Thus, we may write
\[
 \mathbf{X}=\mathbf{X}(r,\theta,u,\hat{u}) = \left(
                             \begin{array}{cccc}
                               r & -u & {}^t\!\hat{u} R(\theta) & \frac{-u^2+{}^t\!\hat{u} \hat{u}}{2r} \\
                               0 & 1 & 0 & u/r \\
                               0 & 0 & R(\theta) & \hat{u}/r\\
                               0 & 0 & 0 & 1/r \\
                             \end{array}
                           \right),
                           \]
where $r:I\to \R^+$, $\theta$, $u:I\to \R$, and $\hat{u}={}^t\!(u^2,u^3):I\to \R^2$ are smooth maps and
\[
 R(\theta)=\left(
              \begin{array}{cc}
                \cos\theta & -\sin\theta \\
                \sin\theta & \cos\theta \\
              \end{array}
            \right).
            \]
Using (\ref{CG}), we find $\tilde{\mu}^1_0=r\mu^1_0$ and $(\tilde{\mu}^2_1,\tilde{\mu}^3_1)=(\mu^2_1-u^2\mu^1_0,\mu^3_1-u^3\mu^1_0) R(\theta)$.
From this we see that any timelike curve admits a {\it second order} conformal frame, that is, a first order conformal frame such that $\mu^2_1=\mu^3_1=0$. In addition, if $\mathbf{M}$ and $\widetilde{\mathbf{M}}$ are second order conformal frames along $\gamma$, then $\widetilde{\mathbf{M}}=\mathbf{M} \mathbf{Y}$, where $\mathbf{Y}$ is a smooth map with values in the subgroup $M^{\uparrow}_+(2,3)_2=\{X\in M^{\uparrow}_+(2,3)_1 \,:\, X^2_4=1,\,X^3_4=0\}$ of $M^{\uparrow}_+(2,3)_1$.
Then, $\mathbf{Y}$ can be written as $\mathbf{Y}(r,\theta,u)=\mathbf{X}(r,\theta,u,0)$, where $r>0$, and $\theta$,
$u$ are smooth real-valued functions. From this we infer that
\begin{equation}\label{Gauge1}
 (\tilde{\mu}^2_4,\tilde{\mu}^3_4)=r^{-1}(\mu^2_4,\mu^3_4) R(\theta).
  \end{equation}
Therefore, the differential form
\begin{equation}\label{Q}
  \mathcal{Q}_{\gamma}=\left((\mu^2_4)^2+(\mu^3_4)^2\right)(\mu^1_0)^2=\left((m^2_4)^2+(m^3_4)^2\right)(m^1_0)^2 dt^4
   \end{equation}
is smooth and well defined, i.e., it does not depend on the choice of the second order frame. By construction, $\mathcal{Q}_{\gamma}$ is positive semidefinite, in the sense that $\mathcal{Q}_{\gamma} =Q_{\gamma}dt^4$, where $Q_{\gamma}$ is a smooth non-negative function. We now prove that $\sqrt[4]{{Q}_{\gamma}}$ coincides with the strain density \eqref{strain-density} of the curve. To this purpose, we choose a second order frame field along $\gamma$ such that $M_0=\gamma$ and $\gamma'=\mathrm{v}M_1$, where $\mathrm{v}=\sqrt{|\langle \gamma',\gamma'\rangle|}=m^1_0$ is the speed of $\gamma$. From the Frenet equations $\mathbf M' = \mathbf M\, \mathfrak m$ for $\mathbf{M}$, we obtain $\gamma''=\mathrm{v}'M_1+\mathrm{v}M_1'=\mathrm{v}'M_1+\mathrm{v}(m^0_1M_0- \mathrm{v}M_4)$.
Differentiating this equation and using again the Frenet equations yield
$\gamma'''\equiv -\mathrm{v}^2(m^2_4M_2+m^3_4M_3)$, modulo $\{M_0,M_1,M_4\}$. This implies
\begin{equation}\label{NC}
 \pi_{\mathcal{N}}(\gamma''')=-\mathrm{v}^2(m^2_4M_2+m^3_4M_3),\quad \langle \pi_{\mathcal{N}}(\gamma'''),\pi_{\mathcal{N}}(\gamma''')\rangle=\mathrm{v}^4((m^2_4)^2+(m^3_4)^2),
    \end{equation}
which combined with (\ref{Q})
proves the claim. Now, using (\ref{Gauge1})
and the fact that $(m^2_4)^2+(m^3_4)^2>0$, it follows that there exist second order frames along $\gamma$,
such that $\mu^2_4=\mu^1_0$, $\mu^3_4=0$. These frames are said of {\it third order}.
If $\mathbf{M}$ and $\widetilde{\mathbf{M}}$ are third order frames along $\gamma$, then $\widetilde{\mathbf{M}}=\mathbf{M} \mathbf{Y}(u)$, where
$\mathbf{Y}(u)=\mathbf{X}(1,0,u,0)$ and $u:I\to \R$ is a smooth function. Then, from (\ref{CG}) is follows that $\tilde{\mu}^0_0=\mu^0_0+u\mu^1_0$. Therefore, we may single out a unique third order frame such that $\mu^0_0=0$,
which satisfies the required properties.
\end{proof}

\begin{figure}[ht]
\begin{center}
\includegraphics[height=6.2cm,width=6.2cm]{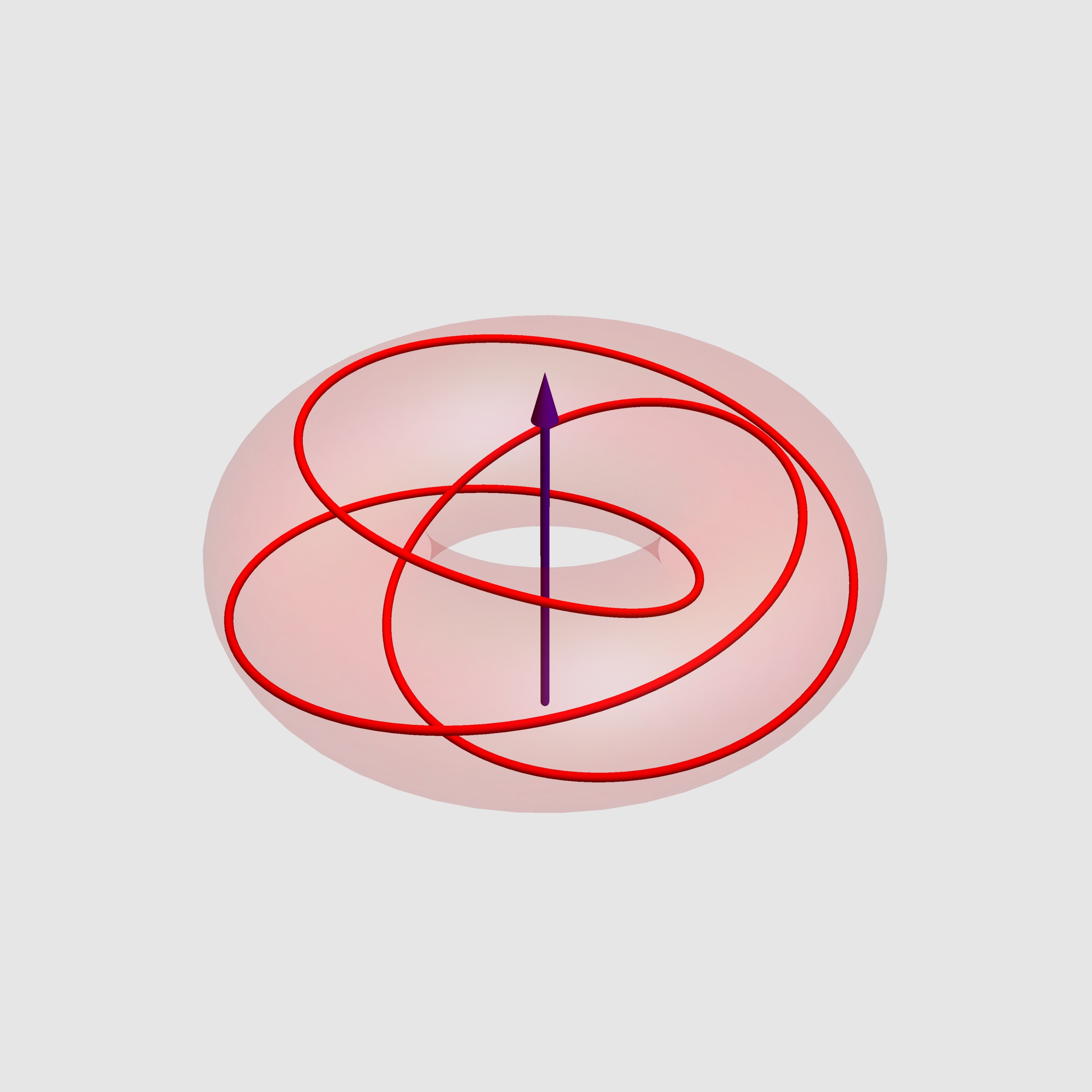}
\includegraphics[height=6.2cm,width=6.2cm]{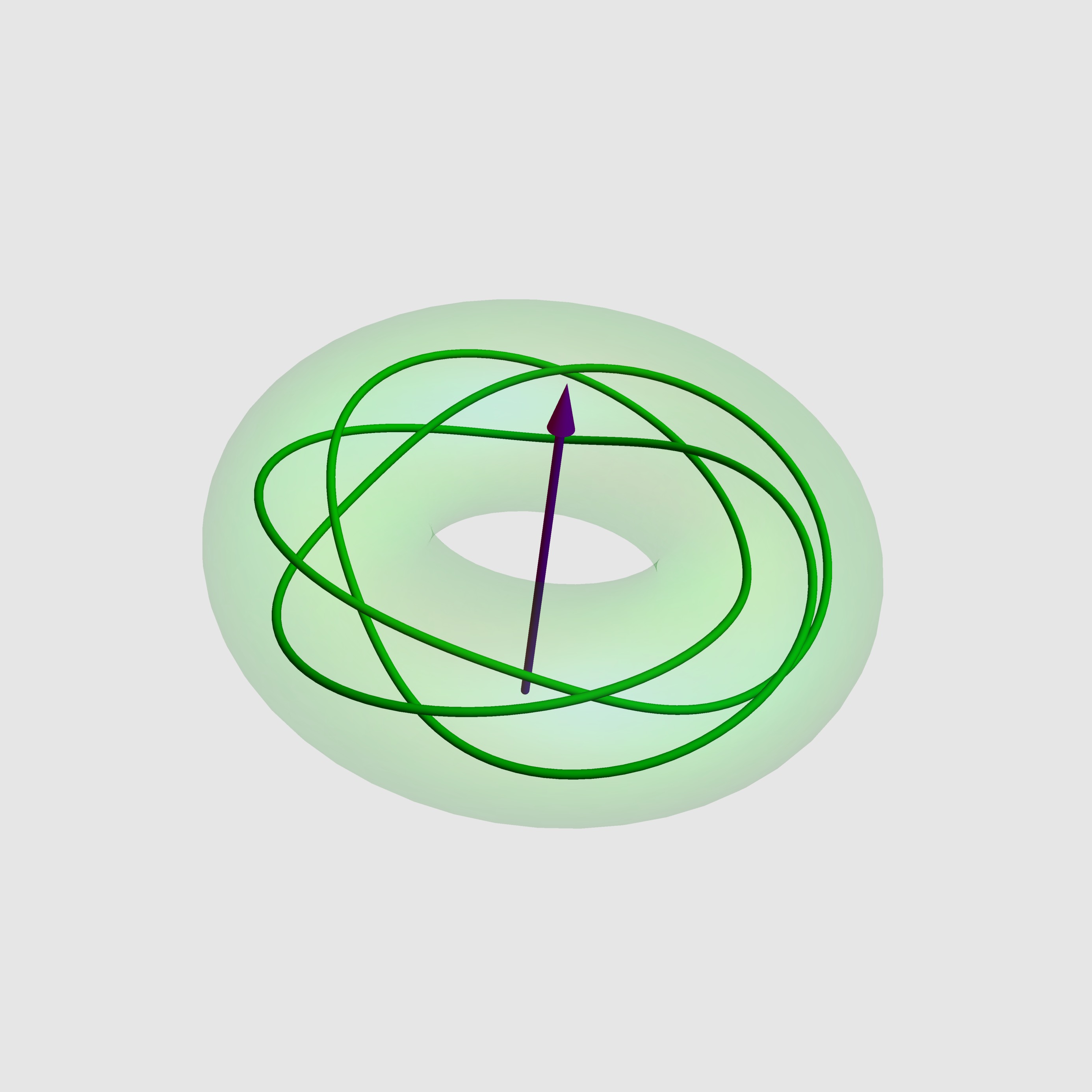}
\caption{Curves with constant curvatures of the Classes $C_{2i}$ (on the left) and $C_{2ii}$ (on the right), with parameters $a=0.5$, $b=2/3$ and $a=0.9$, $b=5/3$, respectively.}\label{FIG4C2iC2ii}
\end{center}
\end{figure}

The following are standard but relevant consequences of the existence of a canonical frame.

\begin{prop}
A generic timelike curve can be parametrized is such a way that $\sigma_{\gamma}$ coincides with the differential of the independent variable. In this case, we say that the curve is {\em parametrized by conformal parameter}, which is usually denoted by $u$. The conformal parameter is uniquely defined up to a constant, $u\to u+c$.
\end{prop}

\begin{remark}
Let $\gamma$ and $\widetilde{\gamma}$ be two generic timelike and future-directed curves, parametrized by conformal parameter. If they are equivalent to each other, the change of parameter is a shift of the independent variable, so that $\gamma(u) = F\cdot \widetilde{\gamma}(u+c)$, where $c$ is a real constant and $F\in \mathrm{A}^{\uparrow}_+(2,3)$ is a conformal transformation. The conformal curvatures are related by $k(u)=\widetilde{k}(u+c)$ and $h(u)=\widetilde{h}(u+c)$, for each $u\in I$.
\end{remark}

\begin{remark}
From now on, when considering a generic timelike curve parametrized by conformal parameter, we suppose that
its interval of definition $I$ be maximal. In this case, $I$ is called the {\em proper interval}
of $\gamma$.
\end{remark}

\begin{prop}
Two timelike curves $\gamma$, $\widetilde{\gamma}$ parametrized by conformal parameter are equivalent to each other if and only if $I=\widetilde{I}+c$, $k(u)=\widetilde{k}(u+c)$, and $h(u)=\widetilde{h}(u+c)$, for some constant $c$.
\end{prop}

\begin{defn}
Given two smooth functions $k,h:I\to \R$, let $\mathcal{K}(h,k)$ denote the $\mathfrak{m}(2,3)$-valued map
defined by
$\mathcal{K}(h,k):= M^0_2-M^4_1- {k}M^2_3 - {h}M^0_1$.
The curvatures and the canonical frame can be used to build the {\em curvature operator} of $\gamma$, i.e., the map
$\mathfrak{K}_{\gamma}:I\to \mathfrak{a}(2,3)$ defined by $\mathbf{M}\, \mathcal{K}(h,k)\, \mathbf{M}^*$, where $k,h$ are the conformal curvatures, $\mathbf{M}$ is the canonical conformal frame and $\mathbf{M}^*|_t= {}^t\!(M^0|_t,\dots,M^4|_t)$ is the dual basis of $\mathbf{M}|_t$, for each $t\in I$.
\end{defn}

\begin{prop}
If $k,h:I\to \R$ are two smooth functions, then there is a generic timelike curve $\gamma$, parametrized by conformal parameter, whose canonical conformal frame $\mathbf{M}$ satisfies the {\em conformal Frenet equations}
$\mathbf{M}'=\mathbf{M}\,\mathcal{K}(k,h)$. By construction, $k$ and $h$ are the conformal curvatures of $\gamma$, which is unique, up to a restricted conformal transformation.
\end{prop}

As a consequence of the previous results, we have the following.

\begin{cor}
The first conformal curvature $k$ vanishes identically if and only if there exists a restricted conformal transformation $F\in \mathrm{A}^{\uparrow}_+(2,3)$, such that the trajectory of $F\circ \gamma$ belongs to
the adS-wall of the Einstein universe.
\end{cor}
\begin{proof}
The conformal curvature $k$ is identically zero if and only if $M_3$, the third vector of the canonical moving frame, is constant. This is equivalent to the existence of a restricted conformal transformation $F\in \mathrm{A}^{\uparrow}_+(2,2)$, such that the null ray $F\circ \gamma(u)$ is orthogonal to the third vector of the standard Poincar\'e basis of $\R^{2,3}$, i.e., if and only if $F\circ \gamma(u)\in \mathcal{T}_{\infty}$.
\end{proof}

\begin{figure}[ht]
\begin{center}
\includegraphics[height=6.2cm,width=6.2cm]{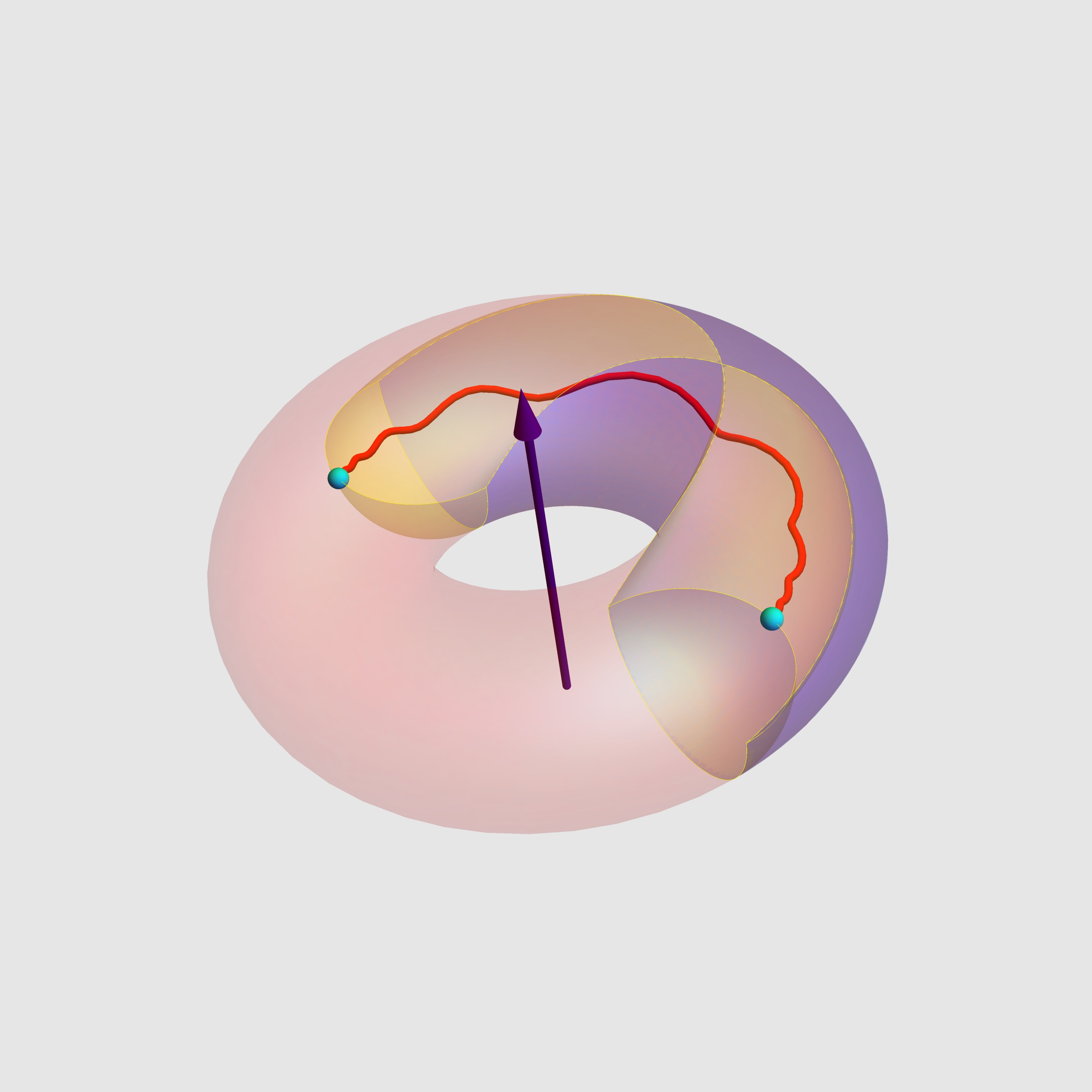}
\includegraphics[height=6.2cm,width=6.2cm]{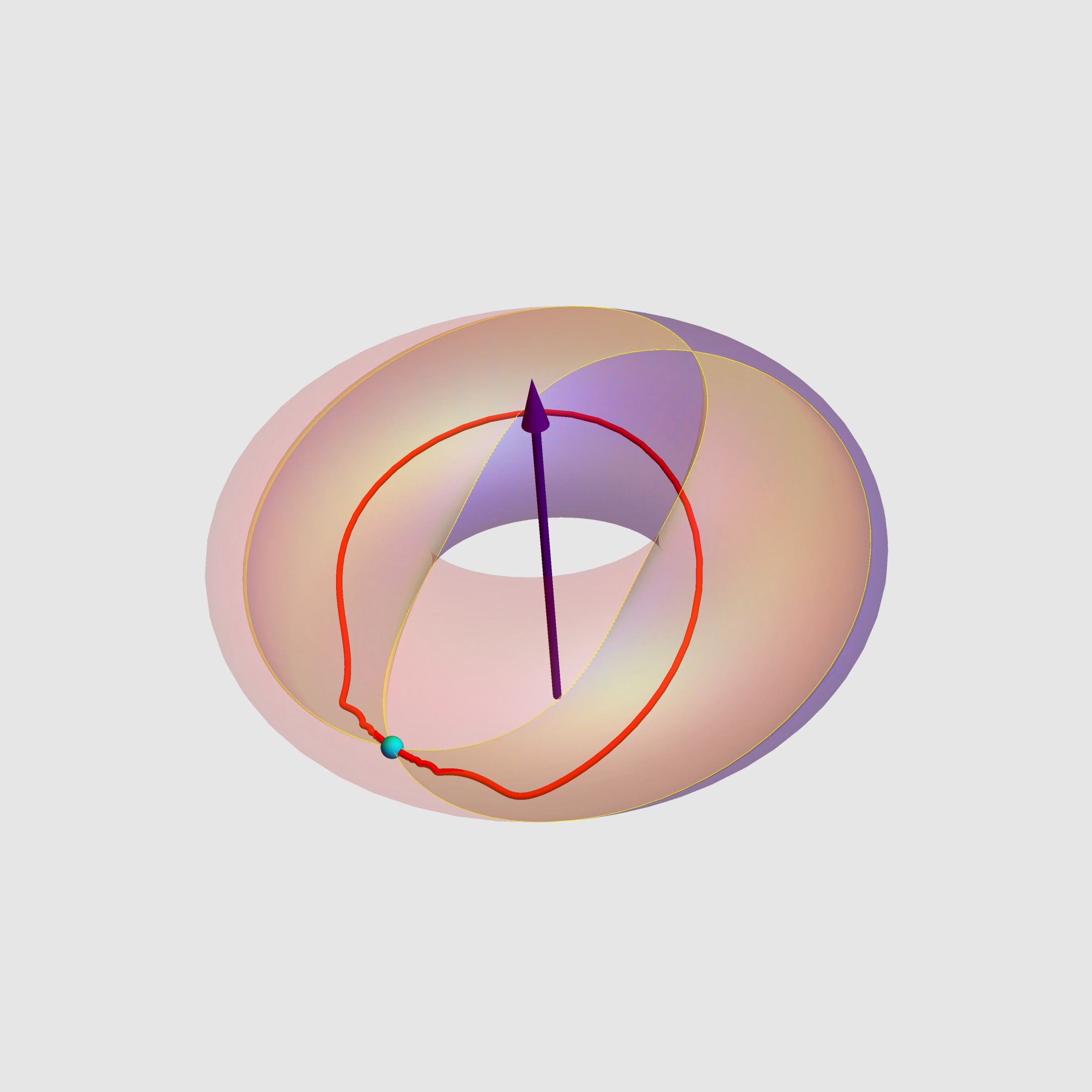}
\caption{Curves with constant curvatures of the Classes $C_{4}$ (on the left) and $C_{5}$ (on the right), with parameters $a=0.99$, $b=7.01722$ and $b=1.001$, respectively.}\label{FIG4C4C5}
\end{center}
\end{figure}

\section{Timelike curves with constant conformal curvatures}\label{s:constant curvatures}

If we act on a generic timelike curve with a time-preserving and orientation-reversing conformal
transformation, the first curvature
changes sign. Therefore, without loss of generality, we may assume $k\geq 0$. Consequently,
the conformal equivalence classes of generic timelike curves with constant conformal curvatures
can be parametrized in terms of two real parameters $(k,h)$ belonging to the half plane $\mathbb{R}^2_+=\{(k,h)\in \R^2 \,:\, k\ge 0\}$.
In principle, since any timelike curve with constant curvatures is an orbit of a 1-parameter subgroup
of $A^{\uparrow}_+(2,3)$ generated by the curvature operator $\mathfrak K_\gamma$,
the explicit determination of timelike curves with constant curvatures can be reduced
to compute $\mathrm{Exp} (t\mathfrak K_\gamma)$. However, the parametrizations obtained in this way
lack in general a direct geometric interpretation suitable for the description of curve trajectories.
Some additional work is required. We classify the generic homogeneous curves with $k>0$ in terms of the stratification
of $\mathbb{R}^2_+$ determined by the orbit-type of $\mathcal{K}(k,h)$. There are nine classes, namely:
\begin{equation}\label{strata}
\begin{cases}
\mathrm{C}_1=\{(k,h)\in \mathbb{R}^2_+ \,:\, -2<k^2-2h<2, k>0\},\\
\mathrm{C}_{2.i}=\{(k,h)\in \mathbb{R}^2_+ \,:\, k^2-2h<-2, k>0\},\\
\mathrm{C}_{2.ii}=\{(k,h)\in \mathbb{R}^2_+ \,:\, 2<k^2-2h, h>-\frac{1}{2k^2}, k>1\},\\
\mathrm{C}_{3}=\{(k,h)\in \mathbb{R}^2_+ \,:\, 0<k<1, -\frac{1}{2k^2}<h<\frac{k^2-2}{2}\},\\
\mathrm{C}_{4}=\{(k,h)\in \mathbb{R}^2_+ \,:\, h<-\frac{1}{2k^2}, k>0\}
\end{cases}
\end{equation}
and
\begin{equation}\label{strata2}
\begin{cases}
\mathrm{C}_{5}=\{(k,h)\in \mathbb{R}^2_+ \,:\, k>1,h=-\frac{1}{2k^2}\},\\
\mathrm{C}_{6}=\{(k,h)\in \mathbb{R}^2_+ \,:\, 0<k<1,h=-\frac{1}{2k^2}\},\\
 \mathrm{C}_{7.i}=\{(k,h)\in \mathbb{R}^2_+ \,:\, k^2-2h=-2, k>0\},\\
\mathrm{C}_{7.ii}=\{(k,h)\in \mathbb{R}^2_+ \,:\, k^2-2h=2, k>0\},\\
\mathrm{C}_{8}=\{(k,h)\in \mathbb{R}^2_+ \,:\, k^2-2h=2,0<k<1\},\\
\mathrm{C}_{9}=\{(1,-1/2)\}.
\end{cases}
\end{equation}
The curvature operators of homogeneous curves of the first four classes are regular elements of the Lie algebra $\mathfrak{a}(2,3)$. In the other cases, the curvature operators are exceptional elements of $\mathfrak{a}(2,3)$. For this reason, the homogeneous curves of the first four classes are said {\em regular}, while those belonging to the last five classes are said {\em exceptional}. We now describe explicit parametrizations of homogeneous timelike curves in terms of elementary functions.
For the classes of regular curves, we have the following.

\begin{figure}[ht]
\begin{center}
\includegraphics[height=6.2cm,width=6.2cm]{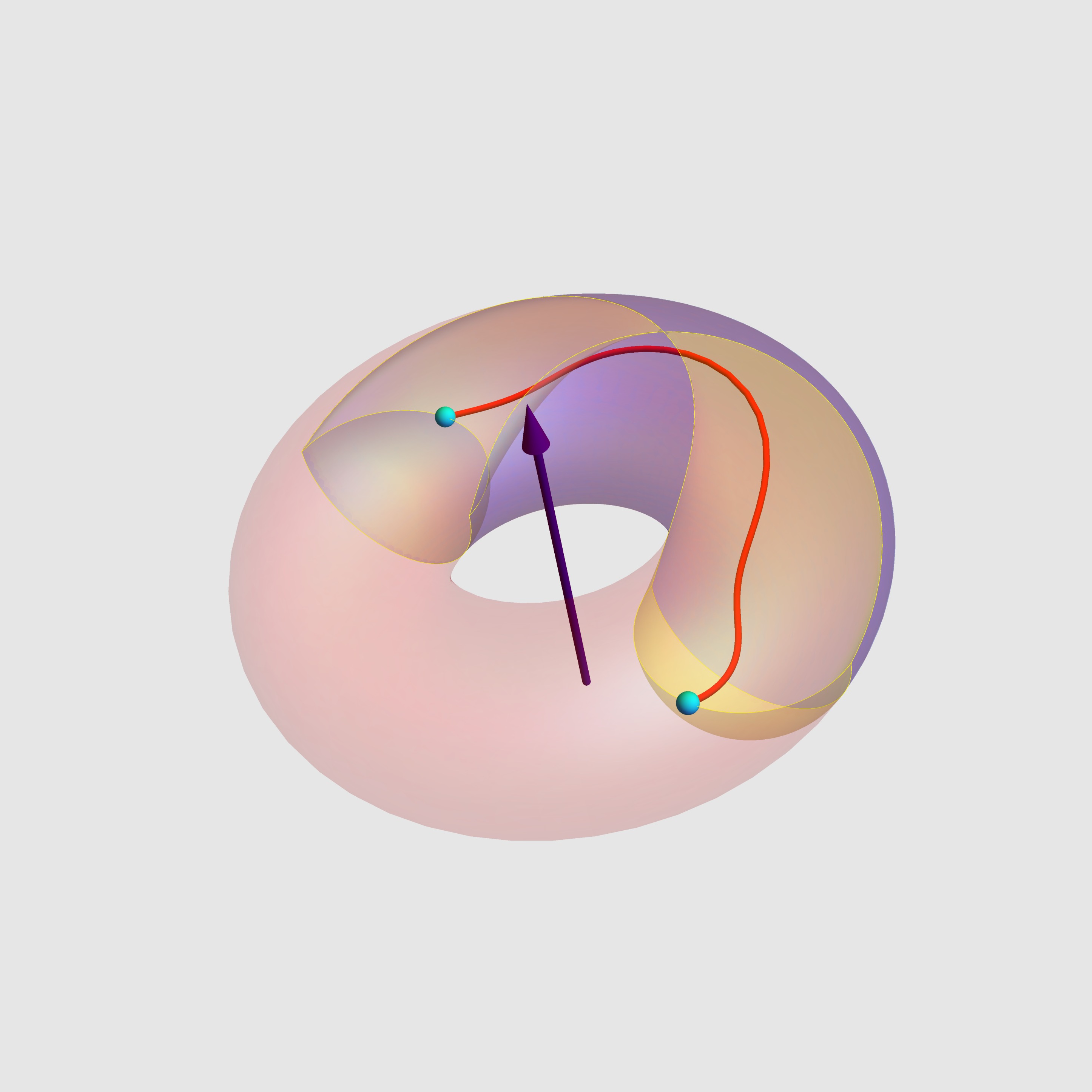}
\includegraphics[height=6.2cm,width=6.2cm]{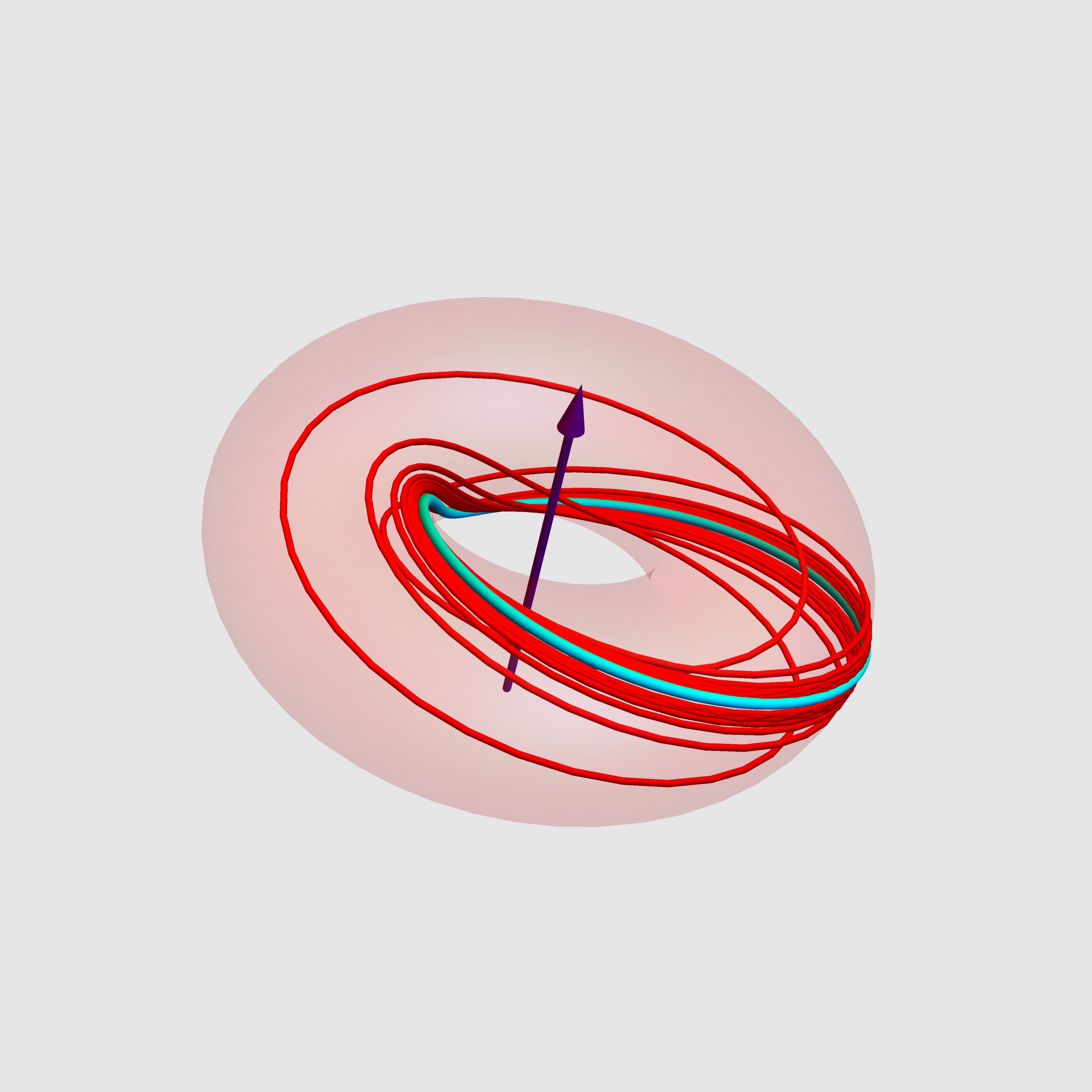}
\caption{Curves with constant curvatures of the Classes $C_{6}$ (on the left) and $C_{7i}$ (on the right), with parameters $b=0.999$ and $b=1.5$, respectively.}\label{FIG4C6C7i}
\end{center}
\end{figure}

\begin{thm}\label{thm:reg-hom-cur}
Let $\gamma$ be a regular homogeneous curve. Then the following hold true:

\begin{itemize}
\item If $\gamma\in \mathrm{C}_1$, there exists a unique Lie basis $(L_0,\dots,L_4)$ and a unique $(a,b)$ belonging to the open domain $\Omega_1 = \{(a,b)\in (-1,1)\times \R \,:\, b>0,\, (1+b+a(b-1))(b-1+a(1+b))<0\}$
such that, after a change of the independent variable, the curve is parametrized by
$\R\ni t\mapsto [\gamma_{(a,b)}(t)]$, where
\begin{equation}\label{C1}
 \gamma_{(a,b)}(t) =x_0(t)L_0 + x_1(t)L_1-
  \sqrt{2(1-a^2)}L_2+x_3(t)L_3+x_4(t)L_4
  \end{equation}
and
\[
\begin{array}{cc}
x_0(t)=e^{b(t-\frac{\pi}{4})}(\cos(t)+a\sin(t)), & x_1(t)=- e^{-b(t-\frac{\pi}{4})}(\sin(t)+a\cos(t)),\\
x_3(t)=e^{b(t-\frac{\pi}{4})}(a\cos(t)-\sin(t)), &x_4(t) = e^{-b(t-\frac{\pi}{4})}(\cos(t)-a\sin(t)).
\end{array}
\]

\item If $\gamma\in \mathrm{C}_{2.i}$, there exists a unique Poincar\'e basis $(P_0,\dots,P_4)$ and a unique element $(a,b)$ of $\Omega_{2.i}= (0,1)\times (0,1)$ such that, after a change of variable, the curve is parametrized by $\R\ni t\mapsto [\gamma_{(a,b)}(t)]$, where
\begin{equation}\label{C2i}
 \gamma_{(a,b)}(t) =\cos(t) P_0+\sin(t)P_1 +aP_2+\sqrt{1-a^2}\cos(bt)P_3 - \sqrt{1-a^2}\sin(bt)P_4.
   \end{equation}

\item If $\gamma\in \mathrm{C}_{2.ii}$, there exists a unique Poincar\'e basis $(P_0,\dots,P_4)$ and a unique $(a,b)$ belonging to open domain $\Omega_{2.ii}=\{(a,b)\in \R^2\,:\, a\in (0,1),\,b>1,\,(1-a^2)b^2<1\}$ such that, after a change of variable, the curve is parametrized by $\R\ni t\mapsto [\gamma_{(a,b)}(t)]$, where
\begin{equation}\label{C2ii}
\gamma_{(a,b)}(t)=\cos(t)P_0+\sin(t)P_1 -aP_2+\sqrt{1-a^2}\cos(bt)P_3 + \sqrt{1-a^2}\sin(bt)P_4.
\end{equation}

\item If $\gamma$ belongs to $\mathrm{C}_{3}$, there exists a unique Lie basis $(L_0,\dots,L_4)$ and a unique element $(a,b)$ of the domain $\Omega_3=\{(a,b)\in \R^2\,:\,a>1/4,\, b>1,\, 4a(b^2-1)-(b^2+1)>0\}$ such that, after a change of variable, the curve is parametrized by $\R\ni t\mapsto [\gamma_{(a,b)}(t)]$, where
\begin{equation}\label{C3}
\gamma_{(a,b)}(t) =\frac{1+4a}{4}e^{t}L_0+e^{-bt}L_1+L_2+\frac{1-4a}{4}e^{bt}L_3+e^{-t}L_4.
\end{equation}

\item If $\gamma\in \mathrm{C}_4$, there exists a unique Poincar\'e basis $(P_0,\dots,P_4)$ and a unique  $(a,b)$ in the open domain $\Omega_4=\{(a,b)\in \R^2\,:\, a\in (0,1),\, b>0,\, a^2+(a^2-1)b^2>0 \}$ such that, after a change of variable, the curve is parametrized by $\R\ni t\mapsto [\gamma_{(a,b)}(t)]$, where
\begin{equation}\label{C4}
 \gamma_{(a,b)}(t) = P_0+a\sinh(t)P_1-a\cosh(t)P_2+\sqrt{1-a^2}\cos(bt)P_3+\sqrt{1-a^2}\sin(bt)P_4.
  \end{equation}

  \end{itemize}
\end{thm}

\begin{proof}
All the curves in the statement are orbits of a 1-parameter group of conformal transformations,
so that the scalar products $\langle \gamma^{(n)},\gamma^{(m)}\rangle$ are constants. Let $c=\sqrt{-\langle \gamma',\gamma'\rangle}$ and consider the nowhere vanishing vector field along $\gamma$ defined by
\[
  \Gamma_4=-\frac{1}{c^2}\gamma''+\frac{\langle \gamma'',\gamma''\rangle}{2c^4} \gamma.
   \]
Next, consider the constant
\[
 R=\sqrt[4]{\frac{\langle \Gamma_4',\Gamma_4'\rangle}{c^2}+\frac{\langle \Gamma_4',\gamma'\rangle \langle \gamma'',\gamma''\rangle}{c^6}+\frac{\langle \gamma'',\gamma''\rangle^2 \langle \gamma',\gamma'\rangle}{4c^{10}}}
  \]
and put
\[
 M_0=R\gamma,\quad M_1=\frac{1}{c}\gamma',\quad M_2=\frac{1}{R^2c}\Gamma_4'+\frac{1}{2c^4R^2}\gamma',\quad M_4=\frac{1}{R}\Gamma_4.
 \]
It is now a computational matter to check that $\langle M_i,M_j\rangle = \mathtt{m}_{ij}$, $i,j=0,1,2,4$.
Then, if we let $M_3$ be the unique spacelike vector field along $\gamma$, such that $\mathbf{M}|_t$ $=$ $(M_0,M_1,M_2,M_3,M_4)|_t$ is a M\"obius basis, for each $t\in \R$, we obtain a
conformal moving frame satisfying $\mathbf{M}'=\upsilon\cdot \mathbf{M}\mathcal{K}(k,h)$, where
\begin{equation}\label{CR}
\upsilon=Rc,\quad k=\frac{1}{Rc}\sqrt{\langle M_2'-R^2c\gamma, M_2'-R^2c\gamma\rangle},\quad h=-\frac{\langle \gamma'',\gamma''\rangle}{2c^4R^2}.
\end{equation}
This shows that $\mathbf{M}$ is the canonical conformal frame along $\gamma$ and, in addition, that the conformal curvatures of $\gamma$ are the constants $k$ and $h$ in \eqref{CR}.

Let $(a,b)$ be an element of $\Omega_1$ and let $\gamma_{(a,b)}$ be given as in (\ref{C1}).
From (\ref{CR}), we compute the following expressions
for the conformal curvatures of $\gamma_{(a,b)}$,
\[
 k_{(a,b)}=\sqrt{\frac{(1+b^2)(1-a^2)}{2b(1+a^2)}},\quad h_{(a,b)}=\frac{1-6b^2+b^4+8ab(b^2-1)-a^2(b^4-6b^2+1)}{4b(1+a^2)(1+b^2)}.
 \]
The map $(a,b)\mapsto (k_{(a,b)},h_{(a,b)})$ is a diffeomorphism of $\Omega_1$ onto $\mathrm{C}_1$, and hence each homogeneous curve of the first class is parametrized by $\gamma_{(a,b)}$, for a unique $(a,b)\in \Omega_1$.

Let $(a,b)\in \Omega_{2.i}$ and $\gamma_{(a,b)}$ be given as in (\ref{C2i}). Then, using (\ref{CR}), we obtain
\[
 k_{(a,b)}=\frac{a\sqrt{b}}{\sqrt[4]{(1-b^2)^2(1-a^2)}},\quad h_{(a,b)}=\frac{1-b^4(1-a^2)}{2b(1-b^2)\sqrt{1-a^2}}.
 \]
Since the map $(a,b)\mapsto (k_{(a,b)},h_{(a,b)})$ is a diffeomorphism of $\Omega_{2.i}$ onto $\mathrm{C}_{2.i}$,
 we deduce that each homogeneous curve of the first type of the second class is parametrized by $\gamma_{(a,b)}$,
 for a unique $(a,b)\in \Omega_{2.i}$.

Let $(a,b)$ be an element of $\Omega_{2.ii}$ and $\gamma_{(a,b)}$ be given as in (\ref{C2ii}). Then, using (\ref{CR}), we obtain
\begin{equation}\label{CurvC2}
 k_{(a,b)}=\frac{a\sqrt{b}}{\sqrt[4]{(b^2-1)^2(1-a^2)}},\quad h_{(a,b)}=\frac{1-b^4(1-a^2)}{2b(b^2-1)\sqrt{1-a^2}}.
 \end{equation}
The map $(a,b)\mapsto (k_{(a,b)},h_{(a,b)})$ is a diffeomorphism of $\Omega_{2.ii}$ onto $\mathrm{C}_{2.ii}$, so that every homogeneous curve of the second type of the second class is parametrized by $\gamma_{(a,b)}$, for a unique $(a,b)\in \Omega_{2.ii}$.

If $(a,b)\in \Omega_{3}$ the conformal curvatures of $\gamma_{(a,b)}$ are given by
\[
 k_{(a,b)}=\frac{\sqrt{2b}}{\sqrt[4]{(b^2-1)^2(16a^2-1)}},\quad h_{(a,b)}=\frac{(4a-1)b^4-4a-1}{2b(b^2-1)\sqrt{16a^2-1}}.
 \]
The map $(a,b)\mapsto (k_{(a,b)},h_{(a,b)})$ is a diffeomorphism of $\Omega_{3}$ onto $\mathrm{C}_{3}$ and hence each homogeneous curve of the third class is parametrized by $\gamma_{(a,b)}$, for a unique $(a,b)\in \Omega_{3}$.

If $(a,b)\in \Omega_{4}$, the conformal curvatures of $\gamma_{(a,b)}$ are given by
\[
  k_{(a,b)}=\frac{\sqrt{b}}{\sqrt[4]{a^2(1-a^2)(1+b^2)^2}},\quad h_{(a,b)}=\frac{a^2(b^4-1)-b^4}{2ab\sqrt{1-a^2}(1+b^2)}.
   \]
The map $(a,b)\mapsto (k_{(a,b)},h_{(a,b)})$ is a diffeomorphism of $\Omega_{4}$ onto $\mathrm{C}_{4}$, so that each homogeneous curve of the third class is parametrized by $\gamma_{(a,b)}$, for a unique $(a,b)\in \Omega_{4}$.
\end{proof}

\begin{remark}
The curves of the first class can be trapped in an adS-chamber, but not in a Minkowski or a dS-chamber.
They have two distinct asymptotic closed null curves lying in an adS-wall (see Figure \ref{FIG4C1C3}). The curves of the second class may be trapped in an adS-chamber. If $b$ is a rational number $m/n$, the curves are closed torus knots of type $(n,m)$. Otherwise they are ``irrational'' lines of a homogeneous torus. They cannot be trapped in a Minkowski or a dS-chamber. The integer $m$ is the liking number of the toroidal projection with the $z$-axis, while $n$ is the linking number with the centerline of the toroid (see Figure \ref{FIG4C2iC2ii}).
The homogeneous curves of the third class can be trapped in the intersection of three chambers of different types, anti-de Sitter, de Sitter, or Minkowski. They have two limit points lying on the intersection of the walls of the three chambers trapping the curve (see Figure \ref{FIG4C1C3}). Also the homogeneous curves of the fourth class are trapped in the intersection of three chambers of different types. They have two limit points belonging to the adS-chamber. One of them lies in the intersection of the Minkowski-wall with one of the dS-walls, while the other lies in the intersection of the Minkowski-wall with the other dS-wall (see Figure \ref{FIG4C4C5}).
\end{remark}


Arguing as in the proof of Theorem \ref{thm:reg-hom-cur}, we can prove the following result
for the classes of exceptional curves.

\begin{thm}\label{thm:exc-hom-cur}
Let $\gamma$ be an exceptional homogeneous curve. Then the following hold true:

\begin{itemize}

\item If $\gamma\in \mathrm{C}_5$, there is a unique M\"obius basis  $(M_0,\dots,M_4)$ and a unique $b>1$
such that, after a change of variable, the curve is parameterized by $\R\ni t\mapsto[\gamma_{(b)}(t)]$, where
\begin{equation}\label{C5}
  \gamma_{(b)}(t) = \frac{1}{2}(1-b^2t^2)M_0+btM_1+\cos(t)M_2-\sin(t)M_3+M_4.
   \end{equation}

\item If $\gamma\in \mathrm{C}_6$, there exists a unique M\"obius basis $(M_0,\dots,M_4)$ and a unique positive $b<1$
such that, after a change of variable, the curve is parameterized by $\R\ni t\mapsto[\gamma_{(b)}(t)]$, where
\begin{equation}\label{C6}
  \gamma_{(b)}(t) = \frac{1}{2}(1+b^2t^2)M_0+\sinh(t)M_1+\cosh(t)M_2++btM_3++M_.
    \end{equation}

\item If $\gamma\in \mathrm{C}_{7.i}$, there exists a unique Lie basis  $(L_0,\dots,L_4)$ and a unique $b>1$
such that, after a change of variable, the curve is parameterized by $\R\ni t\mapsto[\gamma_{(b)}(t)]$, where
\begin{equation}\label{C71}
 \gamma_{(b)}(t) = x_0(t)L_0-
  x_1(t)L_1+\sqrt{\frac{2(b^2-1)}{b}}L_2-\sin(t)L_3+\cos(t)L_4
  \end{equation}
and
$x_0(t)=\frac{(b^2-1)\cos(t)-t\sin(t)}{b}$, $x_1(t)=\frac{t\cos(t)+(b^2-1)\sin(t)}{b}$.

\item If $\gamma\in \mathrm{C}_{7.ii}$, there exists a unique Lie basis  $(L_0,\dots,L_4)$ and a unique $b>1$
such that, after a change of variable, the curve is parameterized by $\R\ni t\mapsto[\gamma_{(b)}(t)]$, where
\begin{equation}\label{C72}
 \gamma_{(b)}(t) = x_0(t)L_0-
   b\sin(t)L_1-\sqrt{2b(1+b^2)}L_2+x_3(t)L_3+b\cos(t)L_4
     \end{equation}
and
$x_0(t)=(1+b^2)\cos(t)+t\sin(t)$, $x_3(t)=t\cos(t)-(1+b^2)\sin(t)$.

\item If $\gamma\in \mathrm{C}_{8}$, there exists a unique Lie basis  $(L_0,\dots,L_4)$ and a unique $b\in(0,1)$
such that, after a change of variable, the curve is parameterized by $\R\ni t\mapsto[\gamma_{(b)}(t)]$, where
\begin{equation}\label{C8}
 \gamma_{(b)}(t) = (b-t)e^tL_0-
  \sqrt{1-b}e^tL_1+2\sqrt[4]{b^2(1-b)}L_2-(b+t)e^{-t}L_3+\sqrt{1-b}e^{-t}L_4.
  \end{equation}

\item If $\gamma\in \mathrm{C}_{9}$, there exists a unique Lie basis  $(L_0,\dots,L_4)$
such that, after a change of variable, the curve is parameterized by $\R\ni t\mapsto [\widetilde\gamma(t)]$, where
\begin{equation}\label{C9}
 \widetilde{\gamma} (t) = \frac{t^4+6t^2-3}{24}L_0+
  \frac{t(t^2+3)}{6}L_1+\frac{1+t^2}{2}L_2+tL_3-L_4.
   \end{equation}
   \end{itemize}
\end{thm}

\begin{figure}[ht]
\begin{center}
\includegraphics[height=6.2cm,width=6.2cm]{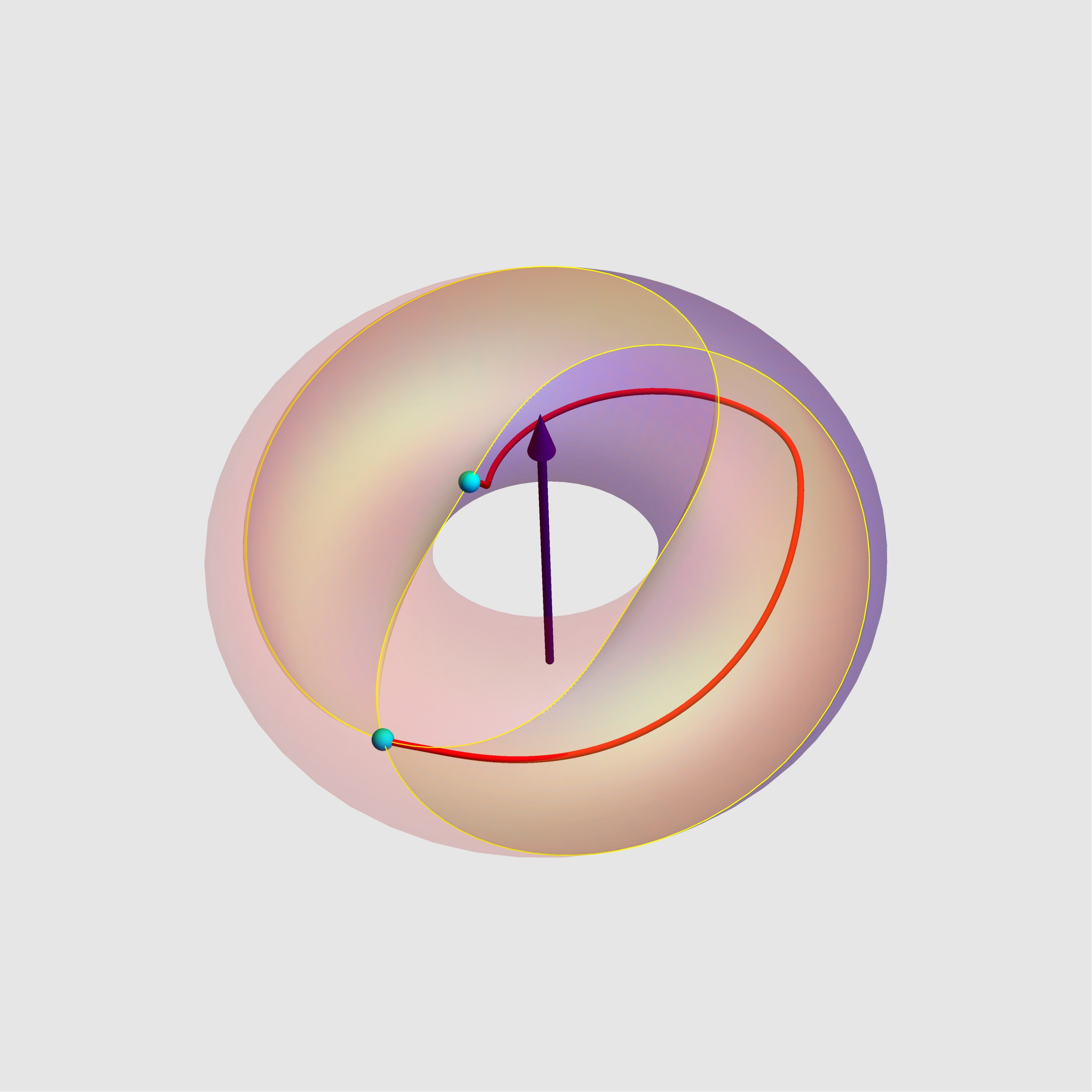}
\includegraphics[height=6.2cm,width=6.2cm]{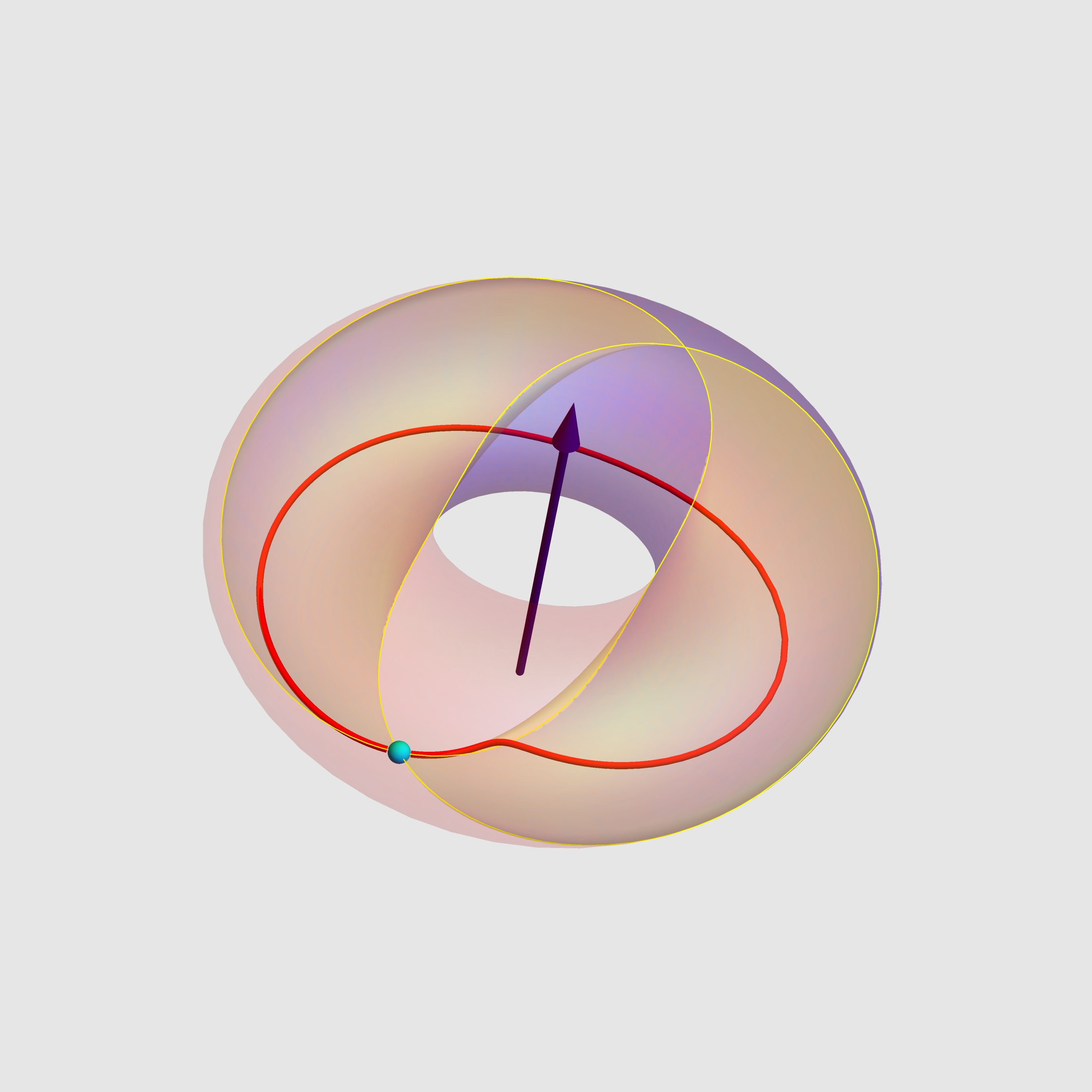}
\caption{Curves with constant curvatures of the Classes $C_{8}$ (on the left) and $C_{9}$ (on the right). }\label{FIG4C8C9}
\end{center}
\end{figure}

\begin{remark}
The curves of the fifth class are equivalent to timelike round helices of Minkowski space (with a timelike axis).
They are trapped in the intersection of a Min\-kow\-ski\--chamber with an adS-chamber. They spiral toward one of the two vertices of the Minkowski-wall and in the interior of the adS-chamber.
Such curves cannot be trapped in any dS-chamber (see Figure \ref{FIG4C4C5}).
The curves of the sixth class are equivalent to timelike ``hyperbolic'' helices of Minkowski space (with a spacelike axis). They are trapped in the intersection of three chamber of different types and have two distinct limit points lying in the interior of the adS-chamber and on the intersections of the Minokowski-wall with one of the dS-walls (see Figure \ref{FIG4C6C7i}). The curves of the seventh class can be trapped in an adS-chamber but not in a Minkowski or dS-chamber. They have one asymptotic closed null curve lying in the adS-wall (see Figure \ref{FIG4C6C7i}). The curves of the eighth class can be trapped in the intersections of an adS-chamber with a Minkowski-chamber but cannot be trapped in a dS-chamber. They have two distinct limit points lying in the adS-wall (see Figure \ref{FIG4C8C9}). One of the limit points is the vertex of the Minkowski wall. The curves of the ninth class are all equivalent to each other; they can be trapped in the intersections of an adS-chamber with a Minkowski-chamber but cannot be trapped in a dS-chamber. They have a unique limit point, one of the two vertices of the Minkowski wall. Such curves close up smoothly at infinity (see Figure \ref{FIG4C8C9}).
\end{remark}

\section{The conformal strain functional}\label{s:strainfunctional}

\subsection{The Euler-Lagrange equations}
A variation of a curve $\gamma:\mathrm{I}\subset \R\to \mathcal{E}^{1,2}$ is a map $\Gamma:\mathrm{I}\times (-\epsilon,\epsilon) \to \mathcal{E}^{1,2}$
such that $\Gamma(u,0)=\gamma(u)$, for every $u\in \mathrm{I}$. Given $t\in (-\epsilon,\epsilon)$, we write $\gamma_{[t]}$ to denote the curve $\mathrm{I}\ni u \mapsto \Gamma(u,t)$. If $\Gamma(u,t)=\gamma(u)$,
for every $u$ outside a closed subinterval $\mathrm{K} \subset \mathrm{I}$
and for every $t\in (-\epsilon,\epsilon)$, we say that the variation is compactly supported.
The smallest of all such closed subintervals is the support of the variation. If $\gamma$ is
generic and timelike and $\Gamma$ is compactly supported then, up to choosing $\epsilon$
sufficiently small, the curves $\gamma_{[t]}$ are generic and timelike, for every $t$.
In the latter case we say that the variation is admissible. Given a closed interval
$\mathrm{K}\subset \mathrm{I}$, the {\it total strain} of the timelike arc $\gamma(\mathrm{K})$
is the integral
$$
\mathcal{S}_{\mathrm{K}}(\gamma)=\int_{\mathrm{K}} \upsilon_{\gamma}(u)du.
$$
We say that $\gamma$ is a critical curve of the conformal strain functional if
$$
\frac{d}{dt}\left(\mathcal{S}_{\mathrm{K}}(\gamma_{[t]}\right)|_{t=0}=0,
$$
for every closed sub-interval $\mathrm{K}\subset \mathrm{I}$ and for every compactly supported variation $\gamma$ with support contained in $\mathrm{K}$.

\begin{thm}\label{t:EulerLagrange}
A generic timelike curve parametrized by conformal parameter is a critical curve of
the conformal strain functional if and only if
\begin{equation}\label{ELE1}
\ddot{k}-k^3+2kh=0,\quad \dot{h}-3k\dot{k}=0.
\end{equation}
\end{thm}

\begin{proof}
First, we prove that any timelike generic curve satisfying (\ref{ELE1}) is a critical point of the conformal strain functional. So, assume that \eqref{ELE1} is satisfied and let $\Gamma$ be an admissible variation. Let $\upsilon_{[t]}$, $k_{[t]}$ and $h_{[t]}$ denote the strain density and the conformal curvatures
of $\gamma_{[t]}$ and set
$\mathfrak{v}(u,t)=\upsilon_{[t]}(u)$, $\mathfrak{k}(u,t)=k_{[t]}(u)$, $\mathfrak{h}(u,t)=h_{[t]}(u)$.
These functions satisfy $\mathfrak{v}(u,t)=1$, $\mathfrak{k}(u,t)=k(u)$ and $\mathfrak{h}(u,t)=h(u)$, for every $(u,t)\notin \mathrm{K}\times (-\epsilon,\epsilon)$, where $\mathrm{K}$ denotes the support of the variation.
Note that since $\gamma$ is parametrized by conformal parameter, then $\mathfrak{v}(u,0)=1$, for every $u\in I$.
Consider the map $\tilde{\mathbf{M}}:I\times (-\epsilon,\epsilon)\to \mathcal{M}$,
such that $\tilde{\mathbf{M}}(u,t)=\mathbf{M}_{[t]}(u)$, where $\mathbf{M}_{[t]}$ is the canonical frame along $\gamma_{[t]}$, and let $\mathcal{U}$, $\mathcal{V}:\mathrm{I}\times (-\epsilon,\epsilon)\to \mathfrak{m}(2,3)$
 denote the maps such that $\partial_u\tilde{\mathbf{M}}=\tilde{\mathbf{M}}  \mathcal{U}$ and  $\partial_t\tilde{\mathbf{M}}=\tilde{\mathbf{M}}   \mathcal{V}$.
From the Maurer--Cartan equations, it follows that
\begin{equation}\label{ELE3}
 \partial_u\mathcal{V}-\partial_t\mathcal{U}+[\mathcal{U},\mathcal{V}]=0.
  \end{equation}
By construction, we have
$\mathcal{U}=\mathfrak{v}(M^0_2-M^4_1-\mathfrak{k}M^2_3-\mathfrak{h}M^0_1)$.
Let
\[
\begin{split}
 \mathcal{V}&=V^0_0M^0_0+V^1_0M^0_1+V^2_0M^0_2+V^3_0M^0_3+V^2_1M^1_2\\
 &\quad +V^3_1M^1_3+V^3_2M^2_3+
   V^1_4M^4_1+V^2_4M^4_2+V^3_4M^4_3.
   \end{split}
   \]
Using (\ref{ELE3}), we compute
\begin{equation}\label{ELE4}
\begin{cases}
V^0_0=\frac{1}{\mathfrak{v}}(\partial_t\mathfrak{v} + \partial_uV^1_4),\\
V^0_0=\frac{1}{\mathfrak{v}}(\partial_uV^2_0 - \partial_t\mathfrak{v})-(\mathfrak{h}V^2_1+\mathfrak{k}V^3_0),\\
V^2_1=\frac{1}{\mathfrak{v}}\partial_uV^2_4 - \mathfrak{k}V^3_4,\\
V^3_1=\frac{1}{\mathfrak{v}}\partial_uV^3_4 + \mathfrak{k} V^2_4,\\
V^3_0=-\frac{1}{\mathfrak{v}}\partial_uV^3_1-\mathfrak{k}V^2_1-\mathfrak{h}V^3_4.
\end{cases}
\end{equation}
%
From the third, fourth and fifth equations of (\ref{ELE4}), it follows that
\begin{equation}\label{ELE5}
V^3_0=-\left( \frac{1}{\mathfrak{v}^2}\partial^2_{uu}V^3_4-\frac{1}{\mathfrak{v}^3}\partial_u\mathfrak{v}\partial_uV^3_4+
\frac{2}{\mathfrak{v}}\mathfrak{k}\partial_uV^2_4+\frac{1}{\mathfrak{v}}\partial_u\mathfrak{k}V^2_4+
(\mathfrak{h}-\mathfrak{k}^2)V^3_4
\right).
\end{equation}
Let $v^i_j:I\to \R$ be the functions defined by $v^i_j(u)=V^i_j(u,0)$.
Evaluating (\ref{ELE4}) and (\ref{ELE5}) at $(u, 0)$ yields
\begin{equation}\label{ELE6}
\begin{cases}
\partial_t\mathfrak{v}|_{(u,0)}=v^0_0 - \dot{v}^1_4,\\
v^0_0=\frac{1}{2}\left( \dot{v}^2_0+\dot{v}^1_4-(hv^2_1+kv^3_0)\right),\\
v^2_1=\dot{v}^2_4-kv^3_4,\\
v^3_0=-(\ddot{v}^3_4+2k\dot{v}^2_4+\dot{k}v^2_4+(h-k^2)v^3_4).
\end{cases}
\end{equation}
%
%
Note that the supports of the functions $v^i_j$ are contained in $\mathrm{K}$.
Using (\ref{ELE6}) and integrating by parts yields
\begin{equation}\label{ELE7}
 \frac{d}{dt}\left.\left(\int_{\mathrm{K}} \mathfrak{v}(u,t)du \right)\right|_{{t=0}}= \int_{\mathrm{K}}
 v_0^0(u)du=-\frac{1}{2}\int_{\mathrm{K}} (hv^2_1 + kv^3_0)du.
    \end{equation}
Using again (\ref{ELE6}), we have
\begin{equation}\label{ELE8}
\begin{split}
\int_{\mathrm{K}}(hv^2_1+kv^3_0)du &= \int_{\mathrm{K}}\left((k^3-2kh)v^3_4+(h-2k^2)\dot{v}^2_4-k\dot{k}v^2_4-k\ddot{v}^3_4\right)du \\
& = - \int_{\mathrm{K}} (\ddot{k}-k^3+2kh)v^3_4du - \int_{\mathrm{K}} (\dot{h}-3k\dot{k})v^2_4 du.
\end{split}
\end{equation}
%
From (\ref{ELE7}) and (\ref{ELE8}), it follows that $\gamma$ is a critical curve if \eqref{ELE1} is satisfied.

\vskip0.1cm

Conversely, it suffices to prove that for each $u_0\in \mathrm{I}$
there exists a closed interval $\mathrm{K}\subset \mathrm{I}$, containing $u_0$, such that,
for every pair of smooth functions $v,w:\mathrm{I}\to \R$ whose supports are contained in $\mathrm{K}$,
there exists a compactly supported variation satisfying
$v^1_4=0$, $v^2_4=v$ and $v^3_4=w$.
Since this property is local and invariant under the action of the conformal group, we may assume
that the trajectory of the curve is contained in $\mathbf{j}_m(\mathbb{M}_0^{(1,2)})$.
We can then write
\[
 \gamma=[M^o_0+\sum_{j=1}^{3}\alpha^jM^o_j+\frac{{}^*\!\alpha  \alpha}{2}M^o_4],
  \]
where $\alpha={}^t\!(\alpha^1,\alpha^2,\alpha^3)$ is a future-directed timelike curve of Minkowski 3-space.
Next, let
\begin{equation}\label{ELE9}
\begin{cases}
T=\frac{1}{\sqrt{(\dot{\alpha}^1)^2-(\dot{\alpha}^2)^2-(\dot{\alpha}^3)^2}}
{}^t\!(\dot{\alpha}^1,\dot{\alpha}^2,\dot{\alpha}^3),\\
N=\frac{1}{\sqrt{(\dot{\alpha}^1)^2-(\dot{\alpha}^2)^2}}{}^t\!(\dot{\alpha}^2,\dot{\alpha}^1,0),\\
B=T_{\alpha}\times N_{\alpha},
\end{cases}
\end{equation}
where ``$\times$'' denotes the vector cross product in $\mathbb{M}_0^{(1,2)}$. By posing
\begin{equation}
\begin{cases}
\mathrm{F}_0=M^o_0+\sum_{j=1}^{3}\alpha^jM^o_j+\frac{{}^*\!\alpha \alpha}{2}M^o_4,,\\
\mathrm{F}_1=\sum_{j=1}^{3}T^jM^o_j+{}^*\!\alpha T M^o_4,\\
\mathrm{F}_2=\sum_{j=1}^{3}N^jM^o_j+{}^*\!\alpha  N M^o_4,\\
\mathrm{F}_3=\sum_{j=1}^{3}B^jM^o_j+ {}^*\!\alpha  B M^o_4,\\
\mathrm{F}_4=M^o_4,
\end{cases}
\end{equation}
the map $\mathbf{F}=(\mathrm{F}_0,\dots, \mathrm{F}_4):I\to \mathcal{M}$ is a first order conformal frame field,
the Minkowski frame along $\gamma$. Thus, there exists a unique smooth map $X:I\to M^{\uparrow}_+(2,3)_{1}$,
such that $\mathbf{M}=\mathbf{F}\cdot X$ is the canonical frame.\footnote{Here, $M^{\uparrow}_+(2,3)_{1}$
is the closed subgroup defined in \eqref{M1}.} Let
\begin{equation}\label{ELE11}X=\left(
      \begin{array}{ccccc}
        r &* & * & * & * \\
        0 & 1 & 0 & 0 & * \\
        0 & 0 & \cos \vartheta  & -\sin \vartheta  & * \\
        0 & 0 & \sin \vartheta  & \cos \vartheta  & * \\
        0 & 0 & 0 & 0 & r^{-1} \\
      \end{array}
    \right),\end{equation}
where $r,\vartheta:I\to \R$ are smooth functions. Next, let
\begin{equation}\label{ELE10Bis}
 \phi_2=\frac{1}{r}(v\cos \vartheta -w \sin \vartheta ),\quad
  \phi_3=\frac{1}{r}(v \sin \vartheta + w \cos \vartheta )
   \end{equation}
and define $\upsilon : \mathrm{I}\to \mathbb{M}_0^{(1,2)}$ by
$\upsilon=\phi_2 N_{\alpha}+\phi_3B_{\alpha}$.
Since $\mathrm{Supp}(\upsilon)\subset \mathrm{K}$, then
\begin{equation}\label{ELE10Tris}
 \tilde{\alpha}: \mathrm{I}\times (-\epsilon,\epsilon)
 \to \mathbb{M}_0^{(1,2)}, (u,t)\mapsto \alpha(u)+t\upsilon(u),
   \end{equation}
is a compactly supported variation of $\alpha$, and hence
\[
  \Gamma : (u,t)\mapsto [M^o_0+\sum_{j=1}^{3}\widetilde{\alpha}^jM^o_j
    +\frac{{}^*\!\widetilde{\alpha} \widetilde{\alpha}}{2}M^o_4]
      \]
is a compactly supported variation of $\gamma$.
Without loss of generality, we may assume that all the curves $\gamma_{[t]}$ are generic and timelike.
For each $t\in (-\epsilon,\epsilon)$, let $\mathbf{F}_{[t]}:I\to \mathcal{M}$ be
the Minkowski first order frame along $\gamma_{[t]}$ and define $\widetilde{\mathbf{F}}:I\times(-\epsilon,  \epsilon)\to \mathcal{M}$ by $\widetilde{\mathbf{F}}(u,t)=\mathbf{F}_{[t]}(u)$. Then, for every $t\in (-\epsilon,\epsilon)$, there exists a unique map $\tilde{X}_{[t]}:I\to M^{\uparrow}_+(2,3)_1$, such that $\mathbf{M}_{[t]}=\mathbf{F}_{[t]}\cdot \tilde{X}_{[t]}$ is the canonical frame along $\gamma_{[t]}$ and that $\tilde{X}_{[0]}=X$.
If we put $\tilde{\mathbf{M}}(u,t)=\mathbf{M}_{[t]}(u)$, then $V(u)$ $=$ $\mathcal{V}(u,0)$ $=$ $(\widetilde{\mathbf{M}}^{-1}\partial_t \widetilde{\mathbf{M}})|_{(u,0)}=U(u)+W(u)$,
where $U(u)=X(u)^{-1}\cdot (\tilde{\mathbf{F}}^{-1}\partial_t \tilde{\mathbf{F}})|_{(u,0)}\cdot X(u)$ and $W(u)=(\tilde{X}^{-1}\partial_t \tilde{X})|_{(u,0)}$.
Let $V_0= {}^t\!(v^0,v^1,v^2,v^3,0)$, $U_0={}^t\!(u^0,u^1,u^2,u^3,0)$, and
$W_0={}^t\!(w^0,w^1,w^2,w^3,0)$ be the first column vectors of $V$, $U$ and $W$, respectively.
Since $\tilde{X}$ take values in the subgroup $M^{\uparrow}_+(2,1)_1$, then $u^1=u^2=u^3=0$ and $v^1=w^1,v^2=w^2,v^3=w^3$.
Taking into account (\ref{ELE11}), we have
\begin{equation}\label{ELE12}
\left(
\begin{array}{c}
v^1 \\
v^2 \\
v^3 \\
\end{array}
\right)=\left(
\begin{array}{c}
w^1 \\
w^2 \\
w^3 \\
\end{array}
\right)=
r\left(
      \begin{array}{ccc}
        1 & 0 & 0 \\
        0 & \cos \vartheta & \sin \vartheta \\
        0 & -\sin \vartheta & \cos \vartheta \\
      \end{array}
    \right)\cdot \left(
                   \begin{array}{c}
                     m^1 \\
                     m^2 \\
                     m^3 \\
                   \end{array}
                 \right),
    \end{equation}
where $m^1$, $m^2$ and $m^3$ are the second, third and fourth components of the first column vector of $(\tilde{\mathbf{F}}^{-1}\partial_t \tilde{\mathbf{F}})|_{(u,0)}$. According to the
definition of $\tilde{\mathbf{F}}$, it follows that $m^1$, $m^2$ and $m^3$ are the components of $\upsilon$ with respect to the frame $(T,N,B)$. Then, in view of (\ref{ELE10Tris}), we get
\begin{equation}\label{ELE13}\left(
    \begin{array}{c}
      m^1 \\
      m^2 \\
      m^3 \\
    \end{array}
  \right)=\left(
            \begin{array}{c}
              0 \\
              \phi^2 \\
              \phi^3 \\
            \end{array}
          \right)=\frac{1}{r}\left(
                               \begin{array}{ccc}
                                 1 & 0 & 0 \\
                                 0 & \cos \vartheta & -\sin \vartheta \\
                                 0 & \sin \vartheta  & \cos \vartheta \\
                               \end{array}
                             \right)\cdot \left(
                                            \begin{array}{c}
                                              0 \\
                                              v \\
                                              w \\
                                            \end{array}
                                          \right).
\end{equation}
From (\ref{ELE12}) and (\ref{ELE13}), it follows that $v^1=m^1=0$, $v^2=v$ and $v^3=w$, which proves the
required result.
\end{proof}

\begin{figure}[ht]
\begin{center}
\includegraphics[height=4.1cm,width=4.1cm]{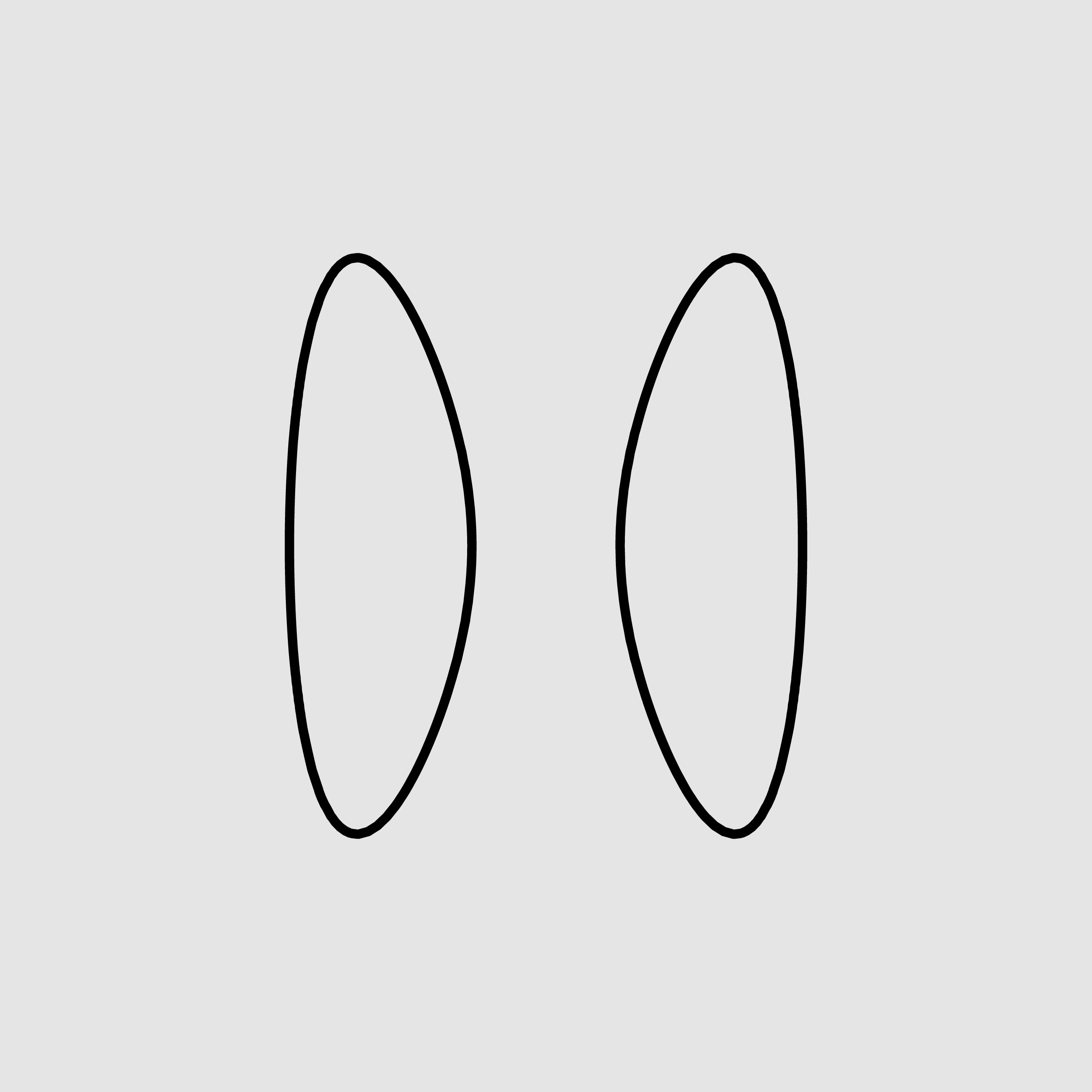}
\includegraphics[height=4.1cm,width=4.1cm]{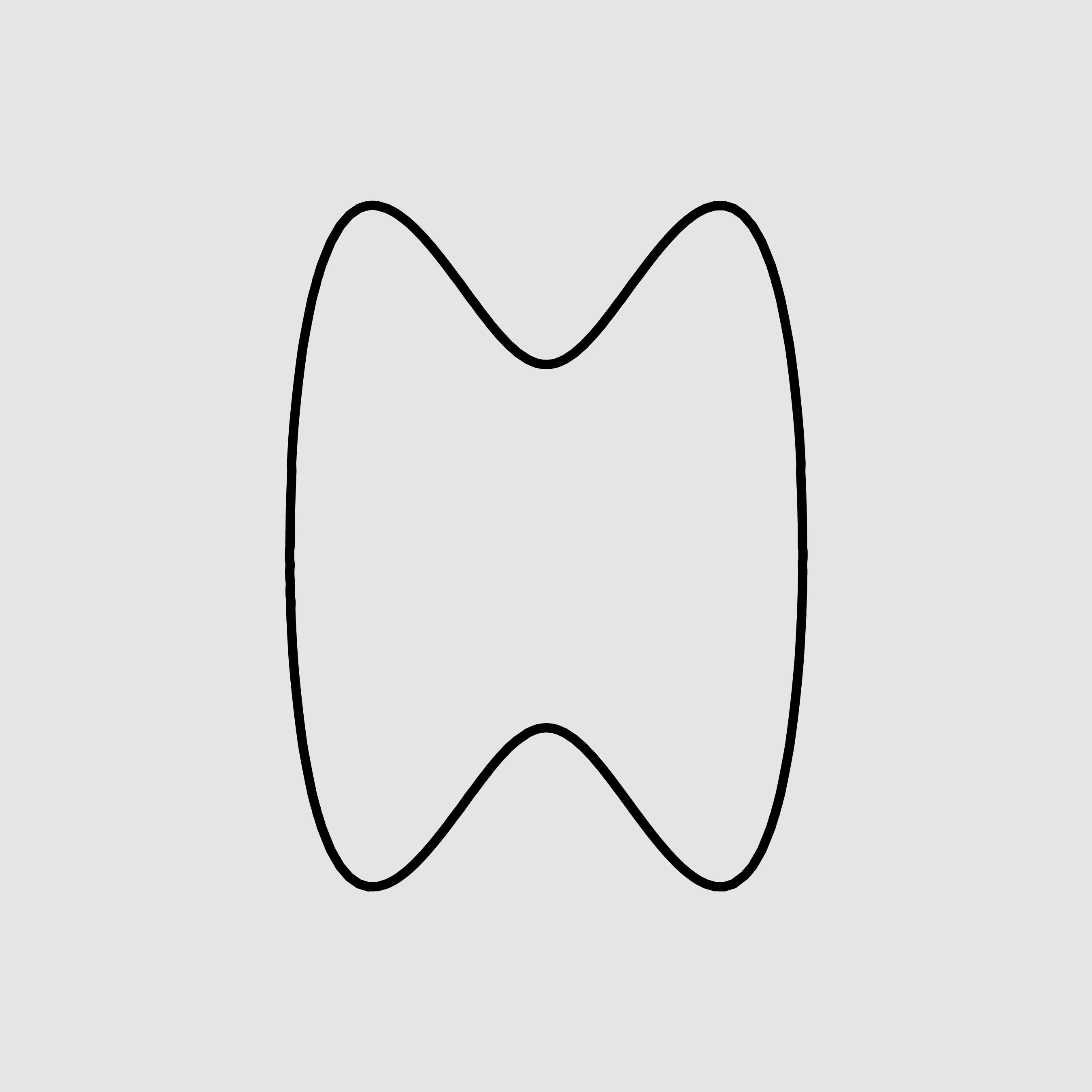}
\includegraphics[height=4.1cm,width=4.1cm]{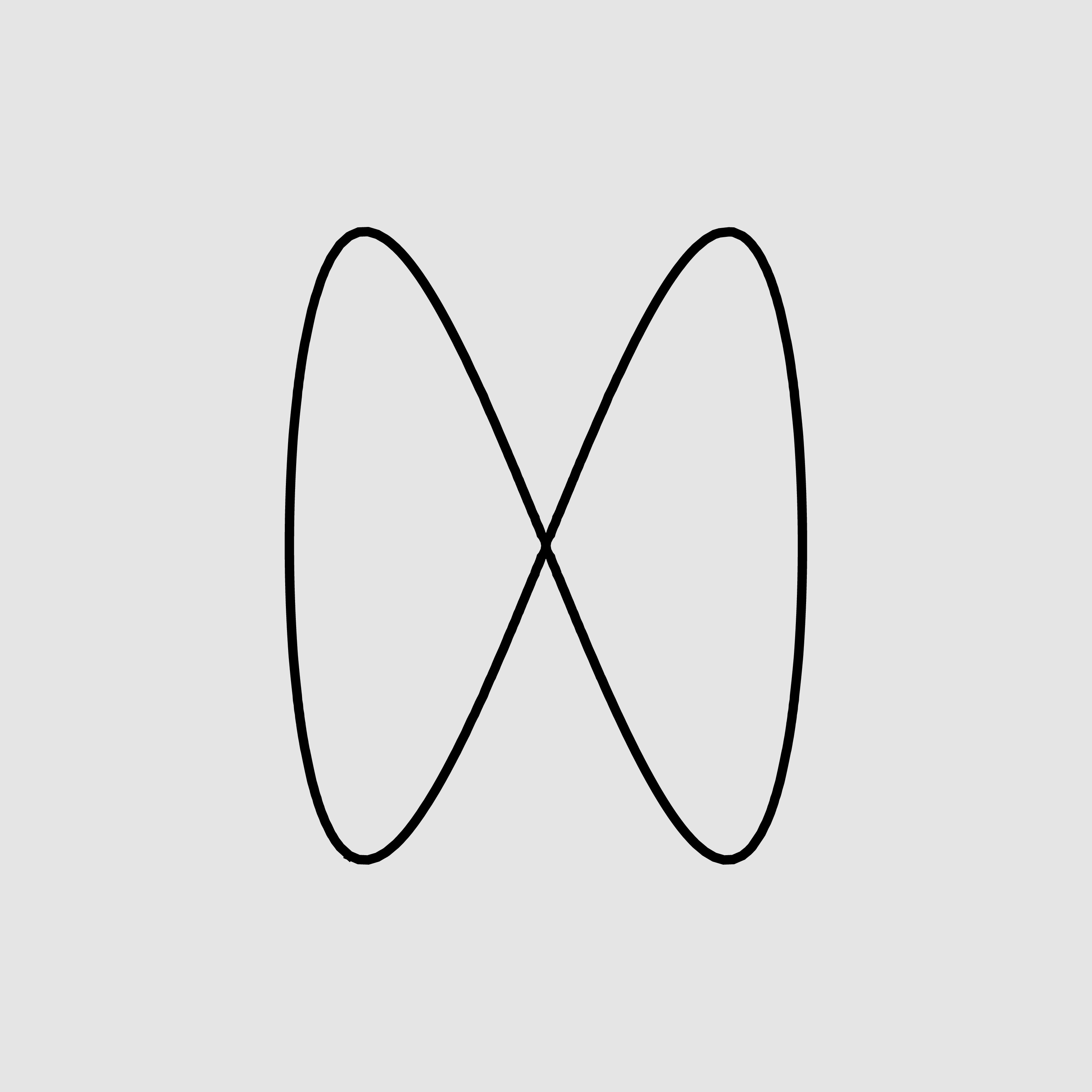}
\caption{Phase portraits of the first, second and third type.}\label{FIG5i5ii5iii}
\end{center}
\end{figure}

\begin{remark}
A generic timelike curve with constant curvatures can be a critical point of the strain functional if and only if $k^3=2kh$. Thus, assuming $k> 0$, these curves belong to the classes $C_1$, $C_4$ or $C_9$.
\end{remark}

\begin{remark}
Within the general scheme of Griffiths' formalism of the calculus of variations \cite{Gr},
using the Euler--Lagrange equations and the existence of the canonical frame, one can build a dynamical system defined on an appropriate momentum space in such a way that the projections of its trajectories on the Einstein universe are the critical curves of the variational problem. More specifically, the momentum space $Y$ is the Cartesian product $\mathcal{M}\times \R^3$ with fiber coordinates $(k,\dot{k},h)$ equipped with the projection
$\pi_Y: (\mathbf{M},k,\dot{k},h)\in Y\to [M_0]\in \mathcal{E}^{1,2}$. The exterior differential 1-forms
$\mu^0_0$, $\mu^1_0$, $\mu^2_0$, $\mu^3_0$, $\mu^2_1$, $\mu^3_2$, $\mu^3_2$, $\mu^1_4$, $\mu^2_4$,
$\mu^3_4$ and $dk$, $d\dot{k}$, $dh$
define an absolute parallelism on $Y$. Let $\partial_{\mu^0_0}$, $\partial_{\mu^1_0}$, $\partial_{\mu^2_0}$, $\partial_{\mu^3_0},\partial_{\mu^2_1}$,
$\partial_{\mu^3_1}$, $\partial_{\mu^3_2}$, $\partial_{\mu^1_4}$, $\partial_{\mu^2_4}$, $\partial_{\mu^3_4}$ and $\partial_{k}$, $\partial_{\dot{k}}$, $\partial_{h}$ be the vector fields of the parallelism and let $\xi$
be the vector field
\[
 \xi=\partial_{\mu^1_0}+\partial_{\mu^2_4}+h\partial_{\mu^1_4}+k\partial_{\mu^3_2}+\dot{k}\partial_{k}+
(k^3-2kh)\partial_{\dot{k}}+3k\dot{k}\partial_{h}.
  \]
The integral curves of $\xi$ are maps $\R \to Y$, $u\mapsto (\mathbf{M}(u),k(u),\dot{k}(u),h(u))$, where

\begin{itemize}
\item $\gamma : u\in \R\to [M_0(u)]\in \mathcal{E}^{1,2}$ is a critical curve of the conformal strain functional, parameterized by the natural parameter;

\item $k,h$ are its conformal curvatures and $\dot{k}$ is the derivative of $k$;

\item $\mathbf{M}:\R\to \mathcal{M}$ is the canonical conformal frame field along $\gamma$.
\end{itemize}

In principle then, the problem is reduced to the integration of the vector field $\xi$. It
is important to note that $\xi$ is the characteristic vector field of the
$\mathrm{A}^{\uparrow}_+(2,3)$-invariant contact form
\[
  \frac{1}{2}(\mu^1_0+\mu^2_4+(k^2-h)\mu^2_0-\dot{k}\mu^3_0+k\mu^3_1).
    \]
\end{remark}

\subsection{The conformal curvatures of the extrema}
From now on we consider critical curves with non-constant curvatures. Then (\ref{ELE1}) implies that $\gamma$ is a critical curve if and only if there exists $(e_1,e_2)\in \mathcal{D}=\{(a,b) \,:\, a<b,\,\, b>0\}$ such that
\begin{equation}
\label{PP1}
\dot{k}^2=-(k^2-e_1)(k^2-e_2),\quad h=\frac{3}{2}k^2-\frac{1}{2}(e_1+e_2).
\end{equation}
We say that $e_1,e_2$ are the {\em parameters} of the critical curve.

\begin{defn}
The {\it phase portrait} of a critical curve with parameters $e_1$ and $e_2$ is the
real algebraic curve $\mathcal{F}_{e_1,e_2}\subset \R^2$ defined by the equation
$y^2 + (x^2-e_1)(x^2-e_2)=0$.
\end{defn}

Three possibilities may occur depending on whether $0<e_1<e_2$, $e_1<0<e_2$, or $e_1=0<e_2$.
Correspondingly, we have a partition of $\mathcal{D}$ into three regions:
$\mathcal{D}_1=\{(e_1,e_2)\in \mathcal{D}\,:\, 0<e_1<e_2\}$,
$\mathcal{D}_2=\{(e_1,e_2)\in \mathcal{D}\,:\, e_1<0<e_2\}$ and
$\mathcal{D}_3=\{(e_1,e_2)\in \mathcal{D}\,:\, e_1=0<e_2\}$.
\vskip0.1cm

\noindent {\bf Phase-Type 1.} If $(e_1,e_2)\in \mathcal{D}_1$, the phase portrait $\mathcal{F}_{e_1,e_2}$
is the real part of a smooth elliptic curve. It consists of two ovaloids, one of them belongs to the
half plane $x>0$ while the other is the mirror image of the first by the reflection about the $y$-axis
(see Figure \ref{FIG5i5ii5iii}). Thus, the first conformal curvature is either strictly positive or else
strictly negative. By possibly acting with a time-preserving and orientation-reversing conformal
transformation, we may assume $k>0$. We put
\begin{equation}\label{PP2}
 m_1=\frac{e_2-e_1}{e_2},\quad p_1=e_2.
  \end{equation}
From (\ref{PP1}) it follows that the conformal curvatures of a critical curve with parameters $(e_1,e_2)\in \mathcal{D}_1$ are
\begin{equation}\label{PP3}
\begin{cases}
 k(u)=\sqrt{p_1(1-m_1)}\mathrm{nd}(\sqrt{p_1}u+c,m_1),\\
  h(u)=\frac{1}{2}(3(1-m_1)\mathrm{nd}(\sqrt{p}u+c,m_1)^2+(m_1-2),
  \end{cases}
   \end{equation}
where $\mathrm{nd}(\cdot,m)$ is the Jacobi $\mathrm{nd}$-function with parameter $m$.
By possibly shifting the independent variable, we can assume $c=0$.
\vskip0.1cm

\noindent {\bf Phase-Type 2.} If $(e_1,e_2)\in \mathcal{D}_2$, then $\mathcal{F}_{e_1,e_2}$ is the real part of a smooth elliptic curve. Unlike the previous case, $\mathcal{F}_{e_1,e_2}$ is connected and symmetric with respect
to the reflections about the coordinate axes (see Figure \ref{FIG5i5ii5iii}).
Given $(e_1,e_2)\in \mathcal{D}_2$, let
\begin{equation}\label{PP6}
  m_2=\frac{e_2}{e_2-e_1}\in (0,1),\quad p_2=e_2-e_1>0.
    \end{equation}
Consequently, the conformal curvatures can be written as
\begin{equation}\label{PP10}
\begin{cases}
 k(u)=\sqrt{p_2(1-m_2)m_2}\,\mathrm{sd}(\sqrt{p_2}u+c,m_2),\\ h(u)=\frac{1}{2}p_2(3(1-m_2)\,\mathrm{sd}(\sqrt{p_2}u+c,m_2)^2+1-2m_2),
  \end{cases}
   \end{equation}
where $\mathrm{sd}(\cdot,m)$ is the Jacobi $\mathrm{sd}$-function with parameter $m$ and $c$ is a suitable constant.
Again, by possibly shifting the independent variable, we take $c=0$.
\vskip0.1cm

\noindent {\bf Phase-Type 3.} If $e_1=0$ and $e_2>0$, the phase portrait is a rational Lemniscate, symmetric with respect to the reflections about the coordinate axes, with the double point located at the origin (see Figure \ref{FIG5i5ii5iii}). By possibly acting by a time-preserving and orientation-reversing conformal transformation,
we may assume that $k>0$. We then have
\begin{equation}\label{PP12}
\begin{cases}
k(u)=\frac{2+e_2\cdot \mathrm{exp}(\sqrt{e_2}(u+c))}{e_2+\mathrm{exp}(2\sqrt{e_2}(u+c))},\\
h(u)=\frac{6e_2^2\cdot \mathrm{exp}(2\sqrt{e_2}(u+c))}{(\mathrm{exp}^(2\sqrt{e_2}(u+c))+e_2)^2}-\frac{1}{2}e_2,
\end{cases}
\end{equation}
where $c$ is a constant that can be put equal to zero.
\vskip0.1cm

\begin{remark}
Let $\gamma$ be a critical curve with conformal curvatures $k,h$ as in (\ref{PP3}), (\ref{PP10}), or (\ref{PP12}).
Let $\mathcal{K}(h,k)= M^0_2-M^4_1- {k}M^2_3 - {h}M^0_1$, as in Section \ref{ss:conf-fr}, and define
$\mathcal{H}:=-M^0_1 - M^4_2-kM^1_3+k'M^0_3+(h-k^2)M^0_2$.
A direct computation shows that $\mathcal H$ satisfies the Lax equation $\mathcal H' =[\mathcal H, \mathcal K]$.
Let $\mathbf{M}$ be
the canonical frame field along $\gamma$. Since $\mathbf M' = \mathbf M \mathcal K$,
 it follows from the Lax equation that $(\mathbf{M} \mathcal{H}  \mathbf{M}^{-1})'=0$,
which implies the existence of a constant element $\mathfrak{m}_{\gamma}\in \mathfrak{m}(2,3)$,
the {\it momentum} of $\gamma$,
such that $\mathfrak{m}_{\gamma}=\mathbf{M} \mathcal{H}  \mathbf{M}^{-1}$,
along $\gamma$.
%
%
This is a consequence of the Noether conservation theorem and of the invariance of the strain
functional under the action of the conformal group.
\end{remark}

\begin{figure}[ht]
\begin{center}
\includegraphics[height=6.2cm,width=6.2cm]{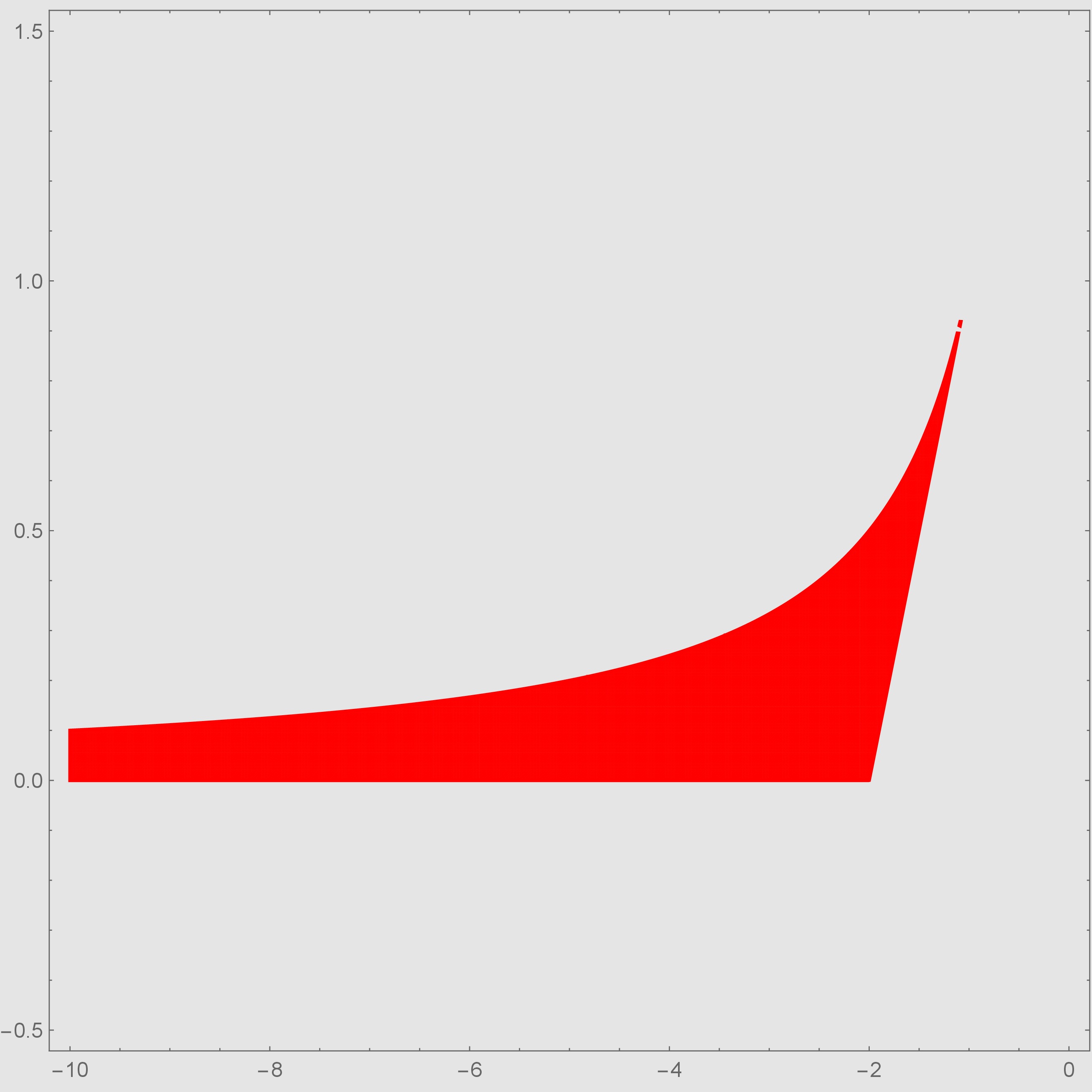}
\includegraphics[height=6.2cm,width=6.2cm]{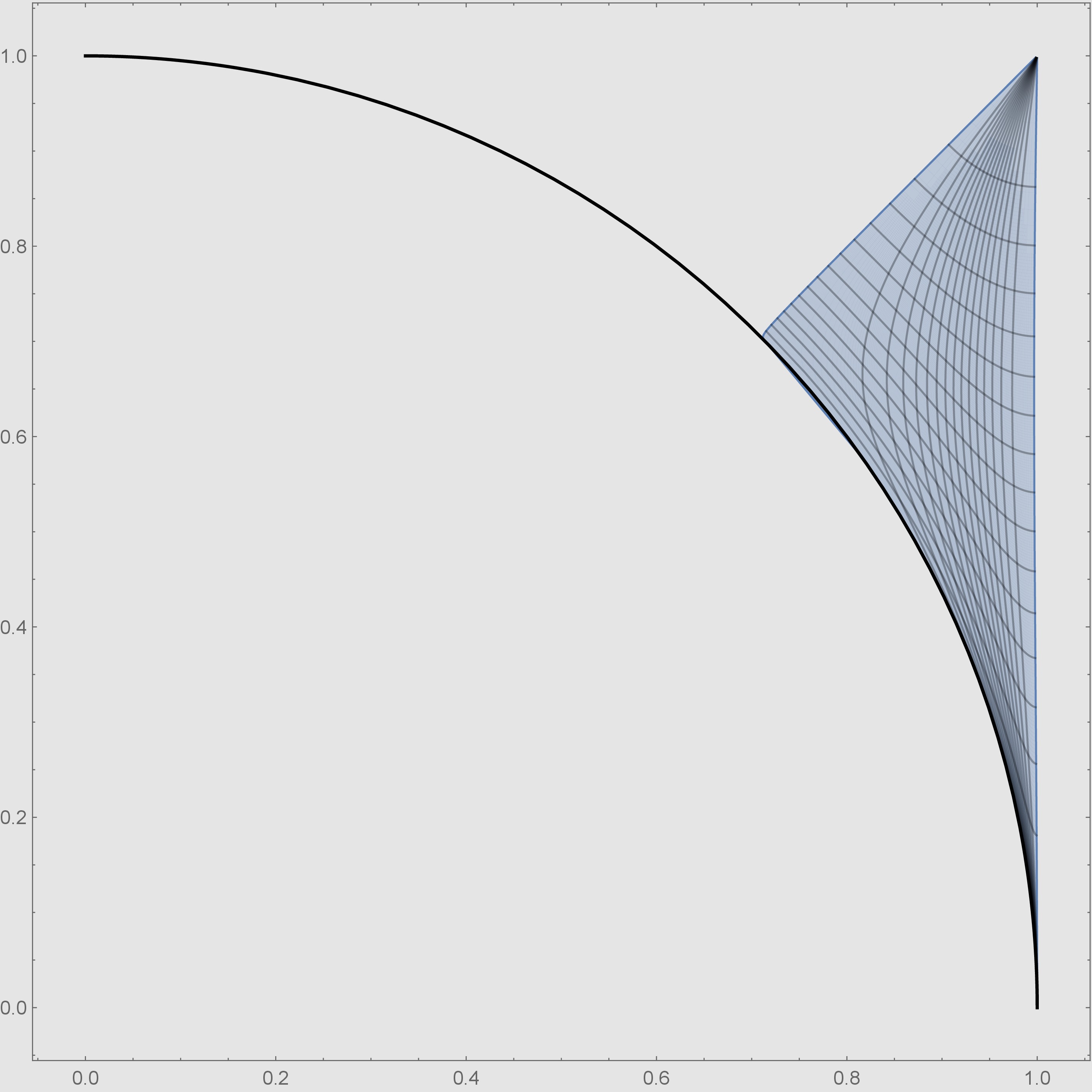}
\caption{The domain $\mathcal{D}^*$ (on the left) and the plot of the map $\Psi$.}\label{FIG5iv}
\end{center}
\end{figure}

Once the curvatures of the extrema are known, the determination of their trajectories amounts to the
integration of the $5\times 5$ linear system $X'=\mathcal{K}(h,k)X$, where $h,k$ are as in (\ref{PP3}),
(\ref{PP10}), or (\ref{PP12}). Such systems can be integrated in terms of elliptic functions and elliptic
integrals.
The integration of timelike trajectories is based on the study of their phase types and of the orbit
types of their momenta.
The detailed analysis produces fifteen cases, to be treated one by one.

\begin{thm}
There exist countably many distinct conformal equivalence classes of closed critical curve of
the the conformal strain functional.
\end{thm}

\begin{proof}
Suppose $(e_1,e_2)$ belongs to $\mathcal{D}^*=\{(a,b)\in \mathcal{D}\,:\, b-a>2,\, a+b<-\sqrt{(a-b)^2-4}\}$. In this case, the phase portrait is of the second type, so the curvatures are as in (\ref{PP10}) and the momentum belongs to the Lie algebra of a maximal compact Abelian subgroup. Let
\[
\xi_1 =\sqrt{\frac{1}{2}\mid (e_1+e_2)+ \sqrt{(e_2-e_1)^2-4}\mid},\quad
\xi_2 =\sqrt{\frac{1}{2}\mid (e_1+e_2)- \sqrt{(e_2-e_1)^2-4}\mid}
\]
and, for $j\in\{1,2\}$, consider the function
\[
\Phi_j(u)=\frac{\xi_j}{p(m-1)+\xi_j}\left(u+\frac{\sqrt{p}(m-1)}{\xi_j^2}
\Pi\left(\frac{m\xi_j^2+mp(m-1)}{\xi_j^2},\mathrm{am}(\sqrt{p}u,m),m\right)\right),
\]
where $\mathrm{am}(\cdot,m)$ is the Jacobi amplitude and
\[
\Pi(x,\phi,m)=\int_0^{\phi}\frac{d\theta}{\sqrt{1-m\sin^2(\theta)}(1-x\sin^2(\theta)}
\]
is the complete elliptic integral of the third kind. We now consider the mappings
\[
\begin{split}
 R &=\cos(\Phi_1)(E^0_0+E^1_1)-\sin(\Phi_1)(E^1_2-E^1_2)+E^2_2\\
 &\quad +\cos(\Phi_2)(E^3_3+E^4_4)-\sin(\Phi_2)(E^3_4-E^4_3)
 \end{split}
 \]
and
\[
   {}^tW = \left(\frac{\sqrt{k^2+\xi_1^2}}{\xi_1\sqrt[4]{(e_2-e_1)^2-4}},
 0,-\frac{k}{\sqrt{1+e_1e_2}},
  \frac{-\sqrt{k^2+\xi_2^2}}{\xi_2\sqrt[4]{(e_2-e_1)^2-4}},0 \right).
    \]
Then, a timelike critical curve with parameters $(e_1,e_2)\in \mathcal{D}^*$ is equivalent to
\[
   \gamma : \R\ni u \mapsto [T_{mp}\cdot R(u)\cdot W(u)]\in \mathcal{M}^{1,2}.
    \]
Let
\[
  \Psi_j=\frac{1}{2\pi}\left(\Phi_j(\frac{4K(m)}{\sqrt{p}})-\Phi_j(0))\right), \quad (j=1,2),
  \]
where $K(m)$ is the complete elliptic integral of the first kind with parameter $m$, and we consider the map
$\Psi : \mathcal{D^*} \to \R^2$, $(e_1,e_2)\mapsto (\Psi_1(e_1,e_2),\Psi_2(e_1,e_2))$. Since $W$ is periodic with minimal period $\omega=4K(m)/\sqrt{p}$, it follows that a critical curve with parameters $(e_1,e_2)\in \mathcal{D}^*$
is closed if and only if
$\Psi(e_1,e_2)\in \mathbb{Q}^2$. In fact, $\Psi$ is a bijection onto the open domain $\widetilde{\mathcal{D}}=\{(x,y)\in \R^2 \,:\, 0<x<y,\, 1-y^2<x<1\}$, so that the conformal equivalence classes of closed trajectories of this type are in one-to-one correspondence with the countable set $\widetilde{\mathcal{D}}\cap \mathbb{Q}^2$ (see Figure \ref{FIG5iv}).
\end{proof}

\begin{remark}
The inverse of the period map $\Psi$ can be computed by numerical methods. In addition,
also the closed trajectories with parameters belonging to $\mathcal{D}^*$ can effectively
be computed (see Figure \ref{FIG5v}).
The above theorem is the Lorentzian counterpart of an analogous result
for the conformal arclength functional for curves in $S^3$ \cite{M2,MN}.
\end{remark}

\begin{figure}[ht]
\begin{center}
\includegraphics[height=6.2cm,width=6.2cm]{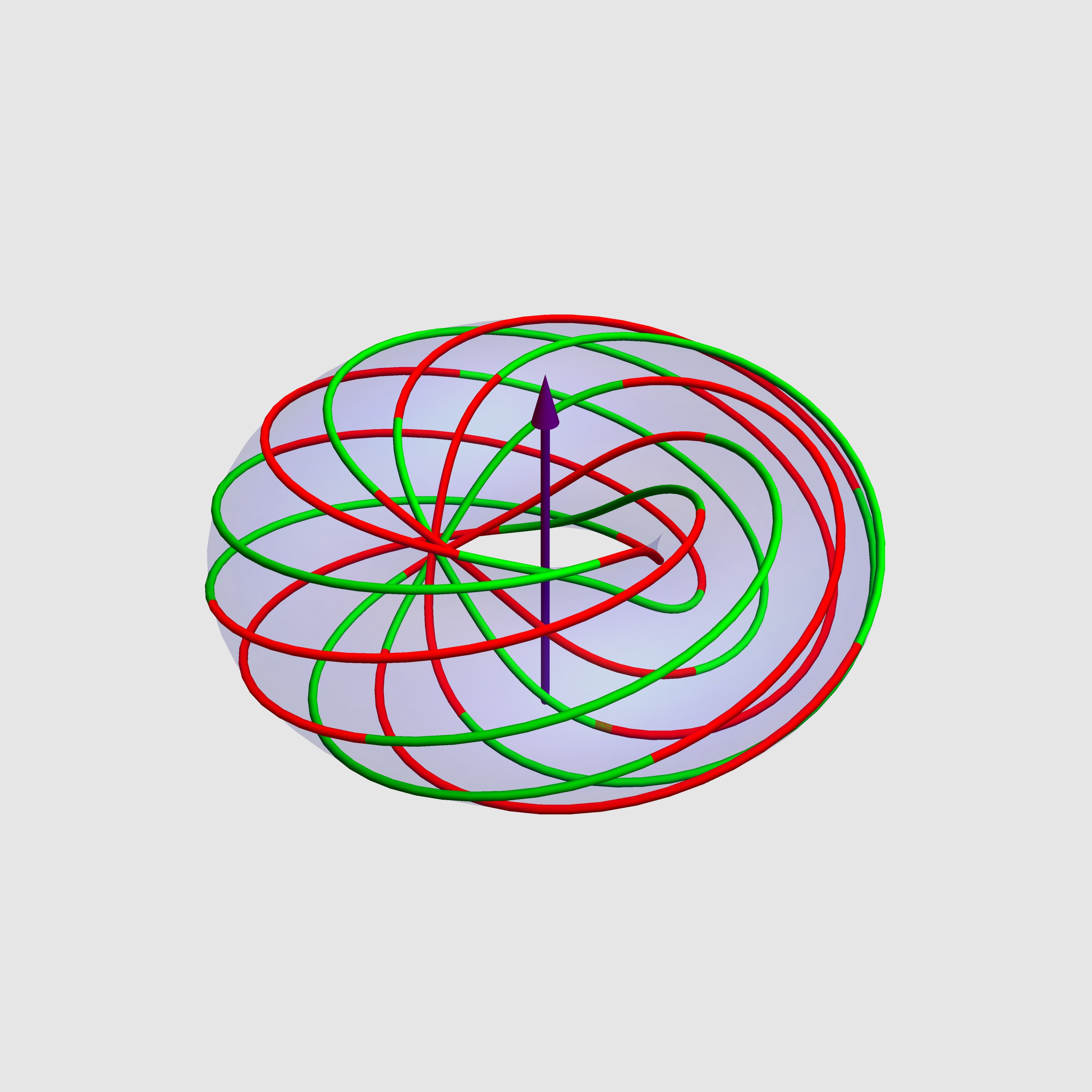}
\includegraphics[height=6.2cm,width=6.2cm]{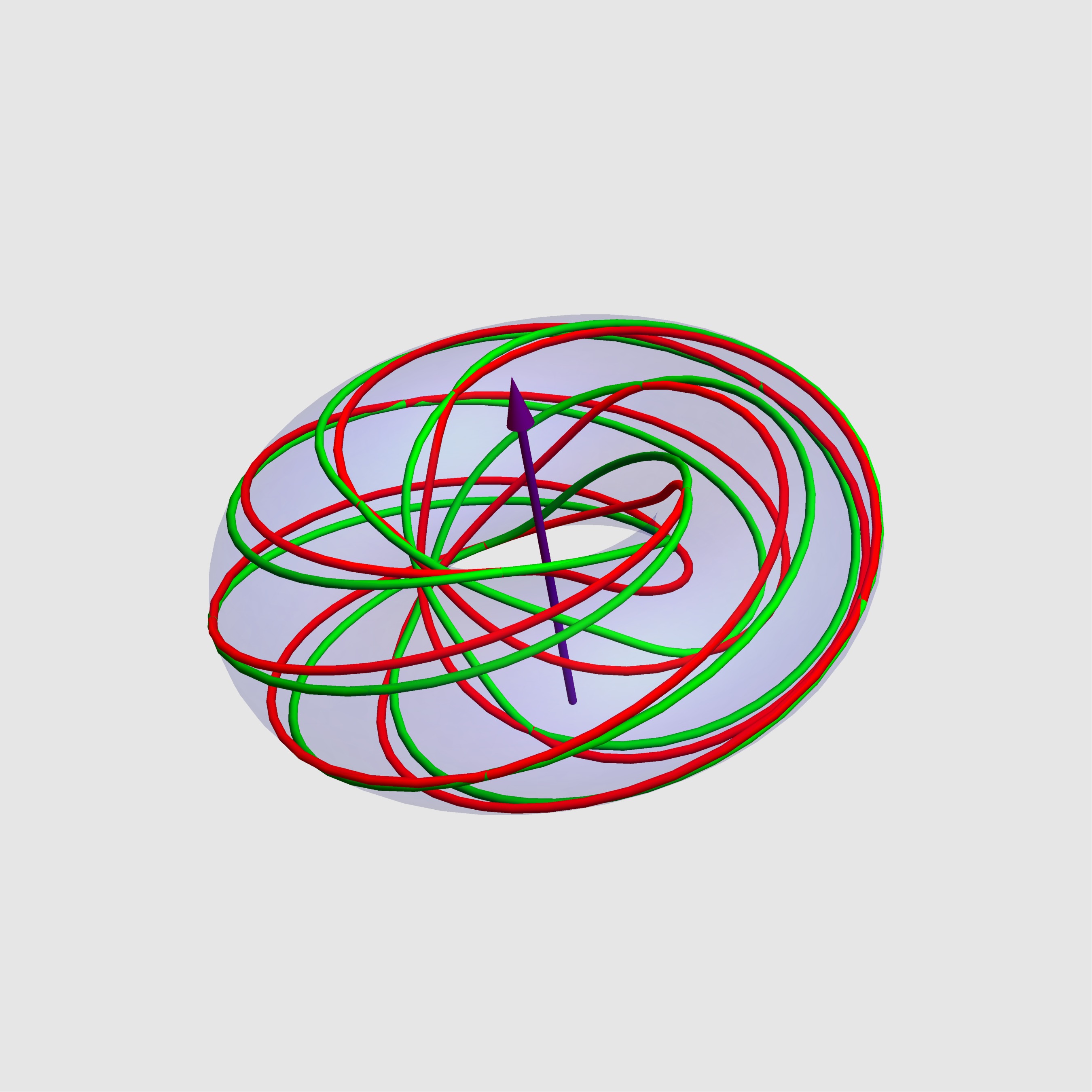}
\caption{The ``toroidal projections'' of closed critical curves with $e_1\approx -1.98638$, $e_2\approx 0.0275109$, $\Psi(e_1,e_2)=(3/4,2/3)$ (on the left) and $e_1\approx -1.74929$, $e_2\approx 0.283545$, $\Psi(e_1,e_2)=(4/5,3/4)$ (on the right). The arcs pictured in red are trapped in the positive adS chamber and the arcs pictured in green are trapped in the negative adS chamber.}\label{FIG5v}
\end{center}
\end{figure}

\section{Timelike curves and transversal knots in $S^3$}\label{s:curves-knots}

In this section we establish and discuss a
connection between the conformal Lorentzian geometry of timelike curves and the geometry of transversal knots
in the unit 3-sphere. We begin by recalling some background material about knots and links,
the symplectic group, and the standard contact structure of the 3-sphere.

\subsection{Preliminaries and notation}

Let $M$ be the unit sphere $S^3$ or the Euclidean space $\R^3$. On
$S^3\subset \R^4$,
we fix the orientation defined by the volume form $\mathrm{Vol}|_p=-i_p^\ast(dy^1\wedge dy^2\wedge dy^3\wedge dy^4)$,
where $(y_1,\dots,y_4)$ are the standard coordinates of $\R^4$ and $i : S^3 \hookrightarrow \R^4$
is the standard inclusion. In $\R^3$ we choose the natural orientation defined by the volume form $\mathrm{Vol}=dx^1\wedge dx^2\wedge dx^3$.

A {\it knot} $\mathrm{K}\subset M$ is a simple, closed oriented curve. Let $\mathrm{N}(\mathrm{K})$ be the normal bundle of the knot. For each $\epsilon>0$, we put
$\mathrm{N}_{\epsilon}(\mathrm{K})=\{(p,\overrightarrow{X})\in \mathrm{N}_p(\mathrm{K}): \|\overrightarrow{X}\|<\epsilon\}$. If $\epsilon$ is sufficiently small, there exists a tubular neighborhood $\mathcal{U}\subset M$ of $\mathrm{K}$ and a diffeomorphism $\Phi$ of $\mathrm{N}_{\epsilon}(\mathrm{K})$ onto $\mathcal{U}$. The orientations of $M$ and $\mathrm{K}$ determine an orientation on $\mathrm{N}_p(\mathrm{K})$, for every $p\in \mathrm{K}$. Fix $p_0\in \mathrm{K}$, a positive-oriented orthogonal basis $(\overrightarrow{V}_1,\overrightarrow{V}_2)$ of $\mathrm{N}_{p_0}(\mathrm{K})$ and a positive number $\eta\in (0,\epsilon)$. Let $\Gamma_{p_0}(\eta)\subset \mathrm{N}_{\epsilon}(\mathrm{K})$ be the oriented circle
$\{(p_0,\eta(\cos(\theta)\overrightarrow{V}_1+\sin(\theta)\overrightarrow{V}_2):\theta\in \R\}$.
Then, the homology class of $\Phi(\Gamma_{p_0}(\eta))$ is a generator of the first homology group $\mathrm{H}_1(M\setminus\mathrm{K},\mathbb{Z})$ that we use to identify $\mathrm{H}_1(M\setminus\mathrm{K},\Z)$ and $\Z$.

\begin{defn}
A {\it link} of $M$ is an ordered pair $(\mathrm{K},\widehat{\mathrm{K}})$ of disjoint knots. Its {\it linking number}, denoted by $\mathrm{Lk}(\mathrm{K},\widehat{\mathrm{K}})$, is the integer that corresponds to the homology class of $\widehat{\mathrm{K}}$ in $\mathrm{H}_1(M\setminus\mathrm{K})\cong \Z$.
\end{defn}

\begin{remark}\label{linking1}
Let $(\mathrm{K},\widehat{\mathrm{K}})$ be a link of $S^3$ and $\verb"P"\in S^3$ be a point such that $\verb"P"\notin \mathrm{K}\cup \widehat{\mathrm{K}}$. Denote by $\mathrm{st}:S^3\setminus\{\verb"P"\}\to \R^3$ the composition of the stereographic projection from $\verb"P"$ with the reflection of $\R^3$ with respect to the coordinate $xy$-plane. This is an orientation-preserving conformal mapping such that
$\mathrm{Lk}(\mathrm{K},\widehat{\mathrm{K}})=
\mathrm{Lk}(\mathrm{st}(\mathrm{K}),\mathrm{st}(\widehat{\mathrm{K}}))$. Thus, in principle, the computation of the linking number of a spherical link can be reduced to that of the linking number of a link in $\R^3$. The latter can be evaluated by means of its link diagram. More precisely, given a link $(\mathrm{K},\widehat{\mathrm{K}})$ in $\R^3$, we fix a plane and consider the orthogonal projections $\mathrm{K}_o$ and $\widehat{\mathrm{K}}_o$ of the knots onto this plane. For a generic choice of the plane, the curves $\mathrm{K}_o$ and $\widehat{\mathrm{K}}_o$ are immersed and intersect each other transversely in a finite number of crossing points $p_1,\dots,p_r$.
%
%
Let $\gamma,\widehat{\gamma}:\R\to \R^3$ be smooth parametrizations of $\mathrm{K}$ and $\widehat{\mathrm{K}}$, respectively, with the same minimal period $\omega>0$. Correspondingly, we have smooth parametrizations $\gamma_o$ and $\widehat{\gamma}_o$ of $\mathrm{K}_o$ and $\widehat{\mathrm{K}}_o$. If $t_j,\widetilde{t}_j\in [0,\omega)$ are chosen so that $\gamma_o(t_j)=\widehat{\gamma}_o(\widetilde{t}_j)=p_j$, the index of $p_j$ is defined by
$$
  \epsilon(p_j)=\mathrm{sgn}((\gamma(t_j)-\widehat{\gamma}(\widetilde{t}_j))\cdot
  (\gamma'_o|_{t_j}\times \widehat{\gamma}'_o|_{\widetilde{t}_j}))
  $$
and the linking number is given by
$$
\mathrm{Lk}(\mathrm{K},\widehat{\mathrm{K}})=\frac{1}{2}\sum_{j=1}^r \epsilon(p_j).
$$
\end{remark}

\begin{ex}\label{LSTK}
Let $p,q$ be two positive integers, such that $q>p>0$ and $\mathrm{gcd}(p,q)=1$. Denote by $\mathrm{K}_{p,q}, \mathrm{K}^*_{p,q}\subset \R^3$ the torus knots of type $(p,q)$ parametrized by
\begin{equation}\label{STK}
\begin{cases}
\Gamma_{p,q} : \R \ni u\mapsto \left((1+\frac{\cos(pu)}{2})\cos(qu),(1+\frac{\cos(pu)}{2})\sin(qu),-\frac{\sin(pu)}{2},\right),\\
\Gamma^*_{p,q}:\R\ni u\mapsto \left((1+\frac{\cos(pu)}{4})\cos(qu),(1+\frac{\cos(pu)}{4})\sin(qu),-\frac{\sin(pu)}{4},\right).
\end{cases}
\end{equation}
Their projections onto the $xy$-plane intersect transversely in $2pq$ distinct points,
each of which has index 1 (see Figure \ref{FIG7}). From this we infer that $\mathrm{Lk}(\mathrm{K}_{p,q},\mathrm{K}_{p,q}^*)=pq$.
\end{ex}

\begin{figure}[ht]
\begin{center}
\includegraphics[height=6.2cm,width=6.2cm]{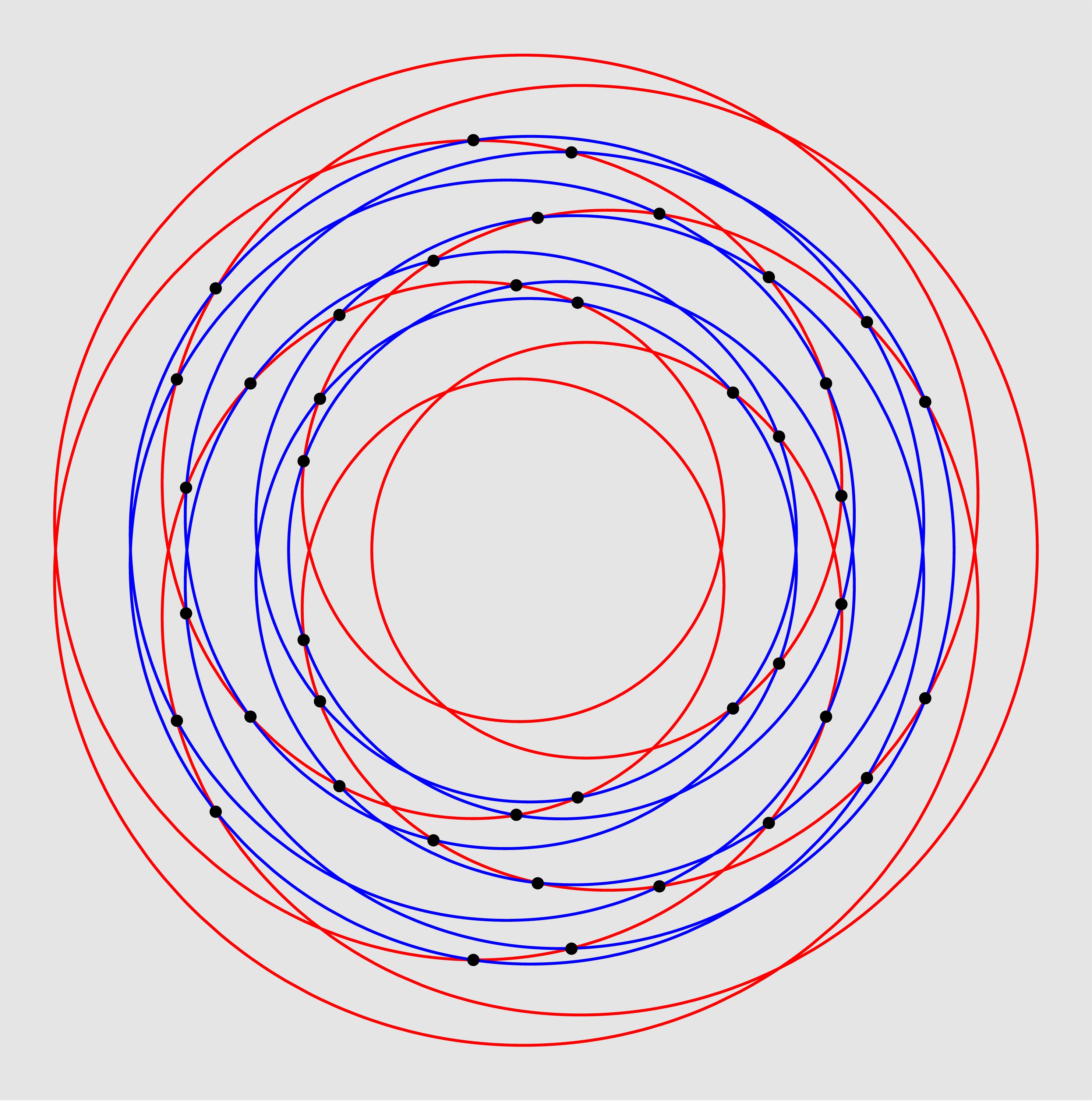}
\includegraphics[height=6.2cm,width=6.2cm]{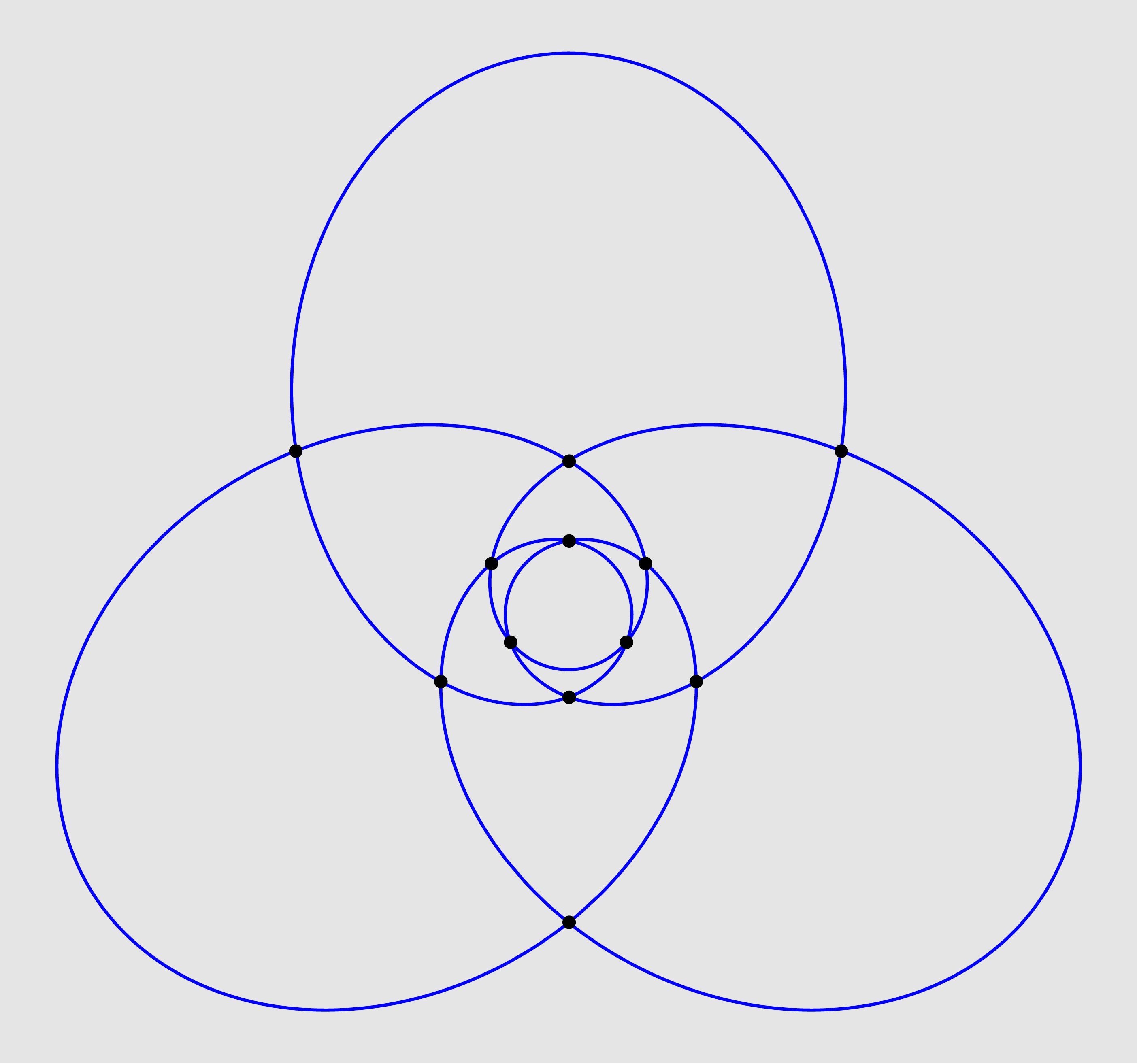}
\caption{The link diagram of $(\mathrm{K}_{3,7},\mathrm{K}_{3,7}^*)$ (on the left) and the writhe diagram of $\check{\mathrm{K}}_{3,5}$ (on the right).}\label{FIG7}
\end{center}
\end{figure}

 The notion of linking number can be adapted for defining the linking
number of a vector filed along a knot.

\begin{defn}
Let $\mathrm{K}\subset M$ be a knot and $\overrightarrow{\mathbf{X}}:\mathrm{K}\to \mathrm{T}(M)$ be a vector field along $\mathrm{K}$. Suppose that $\overrightarrow{\mathbf{X}}$ is transverse to $\mathrm{K}$ and denote by $\overrightarrow{\mathbf{X}}^{\upsilon}$ the orthogonal projection of $\overrightarrow{\mathbf{X}}$ onto the normal bundle of $\mathrm{K}$. Let $\eta>0$ be a positive real number such that $\eta \|\overrightarrow{\mathbf{X}}^{\upsilon}|_p\|<\epsilon$, for every $p\in \mathrm{K}$, and let $\mathrm{K}+\eta \overrightarrow{\mathbf{X}}$ be the knot $\mathrm{K}\ni p\mapsto \Phi[(p,\eta\overrightarrow{\mathbf{X}}^{\upsilon}|_p)]$. Then $(\mathrm{K},\mathrm{K}+\eta \overrightarrow{\mathbf{X}})$ is a link and its linking number, denoted by $\mathrm{Lk}_{\mathrm{K}}(\overrightarrow{\mathbf{X}})$, is said to be the {\it linking number of the vector field} $\overrightarrow{\mathbf X}$ {\it along} $\mathrm{K}$. If $M=S^3$ and  $\verb"P"\in S^3\setminus\mathrm{K}$, then $\mathrm{Lk}_{\mathrm{K}}(\overrightarrow{\mathbf{X}})=
\mathrm{Lk}_{\mathrm{st}(\mathrm{K})}(\mathrm{st}_*(\overrightarrow{\mathbf{X}}))$.
\end{defn}

The linking number of a vector field can be used to define the Bennequin number of a transverse knot in $S^3$ and the self-linking number of a knot in $\R^3$.
\vskip0.1cm

 On the 3-sphere, consider the contact form $\zeta = x_1dx_3+x_2dx_4-x_3dx_1-x_4dx_2$
and the corresponding contact distribution
$\mathcal{D}=\{(p,\overrightarrow{\mathbf X})\in \mathrm{T}_p(S^3) \,:\, \zeta|_p(\overrightarrow{\mathbf X})=0\}$.
Note that $\zeta\wedge d\zeta$ is the volume form used to define the orientation $S^3$. The contact distribution is parallelized by the unit vector fields
\begin{equation}\label{TRCNT}
\begin{cases}\overrightarrow{\mathbf{E}}_1:\verb"X"(x_1,x_2,x_3,x_4)\in S^3\mapsto (-x_4,x_3,-x_2,x_1)\in T_{\verb"P"}(S^3),\\
\overrightarrow{\mathbf{E}}_2:\verb"X"(x_1,x_2,x_3,x_4)\in S^3\mapsto (-x_2,x_1,x_4,-x_3)\in T_{\verb"P"}(S^3).
\end{cases}
\end{equation}
In particular, it admits nowhere vanishing global sections.

\begin{defn}\label{Bennequin}
Let $\mathrm{K}\subset S^3$ be a knot everywhere transverse to the contact distribution $\mathcal{D}$ and let $\overrightarrow{\mathbf X}$ be a nowhere vanishing cross section of $\mathcal{D}$.
The {\it Bennequin number of $\mathrm{K}$}, denoted by $\mathbf{b}(\mathrm{K})$, is
the integer $\mathrm{Lk}_{\mathrm{K}}(\overrightarrow{\mathbf X})$ \cite{Fu,Et2}.
%
%
It is easily seen that the definition is independent
of
the section, so that we can choose indifferently $\overrightarrow{\mathbf X}=\overrightarrow{\mathbf{E}}_1$ or $\overrightarrow{\mathbf X}=\overrightarrow{\mathbf{E}}_2$.
It is important to note that the Bennequin number is invariant under transversal isotopies.
\end{defn}

Let $\mathrm{K}$ be a knot without inflection points and let $(\overrightarrow{\mathbf T},\overrightarrow{\mathbf N},\overrightarrow{\mathbf B})$ be its Frenet frame field.

\begin{defn}\label{selflinking}
The linking number $\mathrm{Lk}_{\mathrm{K}}(\overrightarrow{\mathbf{N}})$ of the unit normal along $\mathrm{K}$ is called the {\it self-linking number} of $\mathrm{K}$ and will be denoted by $\mathrm{SL}(\mathrm{K})$.
\end{defn}

In the literature, the Bennequin number of a transverse knot of $S^3$ is
sometime referred to as the self-linking number of the knot. To avoid confusion, we will
not follow this practice. The self-linking number is invariant only under regular isotopies, i.e.,
isotopies through simple closed curves without inflection points. In this respect, it is a
geometric invariant, but not a topological one.
%
%
Examples of torus knots in Euclidean space of the same type but with different self-linking
numbers are given in \cite{Ba}.

\begin{remark}
In some cases, the self-linking number can be computed in terms of the {\it integral writhe}
of the knot relative to a fixed nonzero vector $\overrightarrow{v}\in \R^3$ in general position with respect to $\mathrm{K}$ (i.e., the projection $\mathrm{K}_o$ of $\mathrm{K}$ onto a plane orthogonal to
$\overrightarrow{v}$ is an immersed curve with ordinary double points). We recall that the integral
writhe of $\mathrm{K}$ with respect to $\overrightarrow{v}$, denoted by
$\mathrm{Wr}_{\mathrm{K}}(\overrightarrow{v})$, is the sum of the indices of the double
points of $\mathrm{K}_o$. The index of a double point $p\in \mathrm{K}_o$ can be defined as
in Remark \ref{linking1}. Namely, we take a parametrization
$\gamma_o:\R\to [\overrightarrow{v}]^{\perp}\subset \R^3$ of $\mathrm{K}_o$, induced by a
periodic parametrization $\gamma:\R\to \R^3$ of $\mathrm{K}$. By possibly replacing $\overrightarrow{v}$ with $-\overrightarrow{v}$, we may assume that $\gamma_o$ is compatible with the counterclockwise orientation of $\mathrm{K}_o$ with respect to normal unit vector $\overrightarrow{v}$.
If $p$ is a double point of $\mathrm{K}_o$, let $t,t'\in [0,\omega)$ such that
$\gamma_o(t)=\gamma_o(t')=p$ and $\gamma_3(t)>\gamma_3(t')$ (i.e., $\gamma(t)-\gamma(t')$
is a positive multiple of the unit vector $\overrightarrow{v}$). Then, the index of $p$ is $\epsilon(p)=\mathrm{sgn}(\overrightarrow{v}\cdot (\gamma'|_{t}\times \gamma'|_{t'})$ and
the integral writhe of $\mathrm{K}$ with respect to $\overrightarrow{v}$ is defined by
$$
 \mathrm{Wr}_{\mathrm{K}}(\overrightarrow{v})=\sum_{p\in \mathrm{K}_o^*}\epsilon(p),
  $$
where $\mathrm{K}_o^*$ is the set of double points of $\mathrm{K}_o$. Note that the integral writhe coincides with $\mathrm{Lk}_{\mathrm{K}}(\overrightarrow{v})$, the linking number of the constant vector field $\overrightarrow{v}$ along the knot $\mathrm{K}$. If the projection $K_o$ is a locally convex curve, then
$$
 \mathrm{SL}(\mathrm{K})=\mathrm{Wr}_{\mathrm{K}}(\overrightarrow{v}).
  $$
The proof of this assertion can be found in \cite{Po}.
\end{remark}

\begin{ex}\label{SLSTK}
 Consider the torus knot $\check{\mathrm{K}}_{p,q}\subset \R^3$ of type $(q,p)$ parametrized by
\begin{equation}\label{TK3}
\check{\Gamma}_{p,q}(t)=\frac{1}{4\cos(pt)-5}\left(3\sin(qt),4\sin(pt),-3\cos(qt)\right).
\end{equation}
It is a computational matter to check that $\check{\mathrm{K}}_{p,q}$ has no inflection points and that its projection onto the $xz$-plane is locally convex and possesses $pq-p$ double points, all of them with index 1 (see Figure \ref{FIG7}).
From this we can conclude that $\mathrm{SL}(\check{\mathrm{K}}_{p,q})=pq-p$.
\end{ex}

The self-linking number can be used for computing the linking number of a nowhere vanishing vector field $\overrightarrow{\mathbf X}$ normal to a knot $\mathrm{K}\subset \R^3$.
In fact, by the C\u{a}lug\u{a}reanu--Fuller--Pohl formula \cite{Ca,Fu,MR,Po,RN}, we have
\begin{equation}\label{CFP}
 \mathrm{Lk}_{\mathrm{K}}(\overrightarrow{\mathbf{X}})=\mathrm{SL}(\mathrm{K})+
  \Theta_{\mathrm{K}}({\overrightarrow{\mathbf{X}}}),
   \end{equation}
where $\Theta_{\mathrm{K}}({\overrightarrow{\mathbf{X}}})$ is the {\it rotation number}
of $\overrightarrow{\mathbf X}$ along $\mathrm{K}$, that is, the degree of the map
\[
 \psi_{\mathrm{K},\overrightarrow{\mathbf{X}}}:\mathrm{K}\ni p\mapsto (\overrightarrow{\mathbf{X}}|_p\cdot \overrightarrow{\mathbf{N}}|_p,
  \overrightarrow{\mathbf{X}}|_p\cdot \overrightarrow{\mathbf{B}}|_p)\in \R^2\setminus\{(0,0)\}.
   \]
Actually, to compute the rotation number one can replace
$\overrightarrow{\mathbf B}$ and $\overrightarrow{\mathbf N}$ by $\overrightarrow{S}_b=\gamma'\times \gamma''$ and $\overrightarrow{S}_n=-\gamma'\times \overrightarrow{S}_{b}$,
where $\gamma$ is any parametrization of the knot. Note that the rotation number and the linking number of $\overrightarrow{\mathbf X}$ depend only on the direction of the vector field and not on its norm.
For this reason, in the literature it is often assumed that $\overrightarrow{\mathbf X}$ is a unit normal
vector field along the knot.

\subsection{The symplectic group and the standard contact structure of $S^3$}

Let $\omega(\verb"X",\verb"Y")=\verb"X"^1\verb"Y"^3+
\verb"X"^2\verb"Y"^4-\verb"X"^3\verb"Y"^1-\verb"X"^4\verb"Y"^2$ be the standard symplectic form of $\R^4$, let $\mathrm{Sp}(4,\R)$ be the linear symplectic group of $\omega$, and let $\mathfrak{sp}(\R^4,\omega)$ denote
its Lie algebra. The elements of $\mathfrak{sp}(\R^4,\omega)$ are the $4\times 4$ matrices of the form
$\verb"A"(\mathbf{a},\mathbf{b},\mathbf{c})=
 \left(\begin{smallmatrix}
 \mathbf a & \mathbf b\\
 \mathbf c & -{}^t\!\mathbf a
 \end{smallmatrix}
 \right)$,
%
where $\mathbf{a}=(a^i_j)$, $\mathbf{b}=(b^i_j)$, $\mathbf{c}=(c^i_j)\in \R(2,2)$, and
$\mathbf{c}= {}^t\!\mathbf{c}$, $\mathbf{b}= {}^t\!\mathbf{b}$.
There are two transitive actions of the symplectic group playing an essential role in our discussion.
The first is the action on the unit sphere $S^3\subset \R^4$ defined by
\[
 \mathrm{Sp}(4,\R)\times S^3 \ni (\verb"X",\verb"x")\mapsto \verb"X"\star \verb"x"
 =|\verb"X"\cdot \verb"x"|^{-1} \verb"X"\cdot \verb"x"\in S^3.
  \]
Such an action preserves the standard oriented contact structure on $S^3$ defined by the 1-form $\zeta$.
Next, consider the 6-dimensional vector space $\Lambda^2(\R^4)$ of skew-symmetric 2-forms of $\R^4$.
The linear symplectic group acts on the left on $\Lambda^2(\R^4)$ by
$\verb"X"\star \alpha(\verb"x",\verb"y")=\alpha(\verb"X"^{-1}\cdot\verb"x",\verb"X"^{-1}\cdot\verb"y")$,
for every $\verb"X"\in \mathrm{Sp}(\R^4,\omega)$, $\alpha\in \Lambda^2(\R^4)$ and $\verb"x",\verb"y"\in \R^4$.
The vector space $\Lambda^2(\R^4)$ can be equipped with the $\mathrm{Sp}(\R^4,\omega)$-invariant
neutral scalar product $(\alpha, \beta)$, defined by
$(\alpha,\beta)\omega^2=\alpha\wedge \beta$, $\forall\,\, \alpha,\beta\in \Lambda^2(\R^4)$.
By construction, $\omega$ is a spacelike vector so that its polar space $[\omega]^{\perp}$ is a
5-dimensional invariant subspace of signature $(2,3)$. Let $(\verb"e"_1,\verb"e"_2,\verb"e"_3,\verb"e"_4)$
be the canonical basis of $\R^4$ and $(\verb"e"^1,\verb"e"^2,\verb"e"^3,\verb"e"^4)$ its dual basis.
Then, the elements
\[
\begin{array}{ll}
\verb"E"_0=\verb"e"^1\wedge \verb"e"^2,&
\verb"E"_1=\frac{1}{\sqrt{2}}(\verb"e"^1\wedge \verb"e"^4-\verb"e"^2\wedge \verb"e"^3),\\
\verb"E"_2=\frac{1}{\sqrt{2}}(\verb"e"^1\wedge \verb"e"^4+\verb"e"^2\wedge \verb"e"^3),&
\verb"E"_3=\frac{1}{\sqrt{2}}(\verb"e"^1\wedge \verb"e"^3-\verb"e"^2\wedge \verb"e"^4),\\
\verb"E"_4=-\verb"e"^3\wedge \verb"e"^4 & \\
\end{array}
\]
constitute a basis of $[\omega]^{\perp}$ such that $(\verb"E"_i,\verb"E"_j)=
\mathtt{m}_{ij}$.
Consequently, $\R^{2,3}$ can be identified with $[\omega]^{\perp}$ by the linear isometry
\[
 \widehat{\varrho}: \R^{2,3}\ni \sum_{j=0}^{4} x^jM^{o}_j
  \mapsto \sum_{j=0}^{4} x^j\verb"E"_j\in [\omega]^{\perp}\subset \Lambda^2(\R^4).
   \]
For $\verb"X"\in \mathrm{Sp}(\R^4,\omega)$, let $\varrho(\verb"X")$ be the automorphism
of $\R^{2,3}$ defined by $V\mapsto \widehat{\varrho}^{-1}\left(\verb"X" \cdot \widehat{\varrho}(V) \right)$.
The map
$\varrho: \mathrm{Sp}(\R^4,\omega)\to \mathrm{A}^{\uparrow}_+(2,3)$, $\verb"X"\mapsto \varrho(\verb"X")$,
is a 2:1 group covering homomorphism whose kernel is the center $\mathcal{Z}=\{\pm \mathrm{Id}_{4\times 4}\}\cong \mathbb{Z}_2$ of $\mathrm{Sp}(4,\R)$.
The corresponding induced Lie algebra isomorphism is given by
$\varrho_*:\mathfrak{sp}(\R^4,\omega)\to \mathfrak{a}(2,3)$,
$\verb"A"(\mathbf{a},\mathbf{b},\mathbf{c})\mapsto \sum_{i,j}\widehat{\varrho}_*(\mathbf{a},\mathbf{b},\mathbf{c})_j^i M_i^o\otimes M_o^j$,
%
%
where $\widehat{\varrho}_*(\mathbf{a},\mathbf{b},\mathbf{c})\in \mathfrak{m}(2,3)$ is defined by
\[
\begin{split}
 \widehat{\varrho}_*(\mathbf{a},\mathbf{b},\mathbf{c})=&-(a^1_1+a^2_2)M^0_0+ a^2_2M^1_2-(a^1_2+a^2_1)M^1_3
 +(a^1_2-a^2_1)M^2_3 \\
 &  -\frac{1}{\sqrt{2}}(c^1_1+c^2_2)M^0_1+\frac{1}{\sqrt{2}}(c^1_1-c^2_2)M^0_2+\sqrt{2}c^1_2M^0_3\\
 &  +\frac{1}{\sqrt{2}}(b^1_1+b^2_2)M^4_1+\frac{1}{\sqrt{2}}(b^1_1-b^2_2)M^4_2+\sqrt{2}b^1_2M^4_3.
  \end{split}
  \]
%
%

\begin{remark}[A Lagrangian model for the Einstein universe]
The symplectic form $\omega$ induces an
invariant pairing between $\R_4$, the dual of $\R^4$, and $\R^4$ given by
\[
  \R_4 \ni \verb"x"= x_1\verb"e"^1+x_2\verb"e"^2+x_3\verb"e"^3+x_4\verb"e"^4 \mapsto \verb"x"^{\natural}
     = x_3\verb"e"_1+x_4\verb"e"_2-x_1\verb"e"_3-x_2\verb"e"_4\in \R^4.
      \]
The linear isometry $\widehat{\varrho}:\R^{2,3}\to [\omega]^{\perp}\subset \Lambda^2(\R^4)$
takes the null vectors of $\R^{2,3}$ to the decomposable 2-forms orthogonal to $\omega$.
Any $[X]\in \mathcal{M}^{1,2}$ determines a unique oriented line
$[\widehat{\varrho}(X)]\subset [\omega]^{\perp}$ spanned by a decomposable 2-form
$\widehat{\varrho}(X)=\alpha(X)\wedge \beta(X)$, such that $\widehat{\varrho}(X)\wedge \omega=0$.
The latter identity means that the 2-plane
$\widehat{\varrho}(X)^{\perp}=\{\verb"w"\in \R^4: i_{\verb"w"}\widehat{\varrho}(X) =0\}$
is Lagrangian. Such a plane is oriented by the basis $(\alpha(X)^{\natural},\beta(X)^{\natural})$, so that $\mathcal{M}^{2,1}$ can be identified with $\mathcal{L}^2_+(\R^4,\omega)$, the manifold of oriented Lagrangian
planes of $(\R^4,\omega)$, by the map $\mathcal{M}^{1,2}\ni [X]\mapsto [\alpha(X)^{\natural}\wedge \beta(X)^{\natural}]\in \mathcal{L}_+^2(\R^4,\omega)$.
\end{remark}

\begin{figure}[ht]
\begin{center}
\includegraphics[height=6.2cm,width=6.2cm]{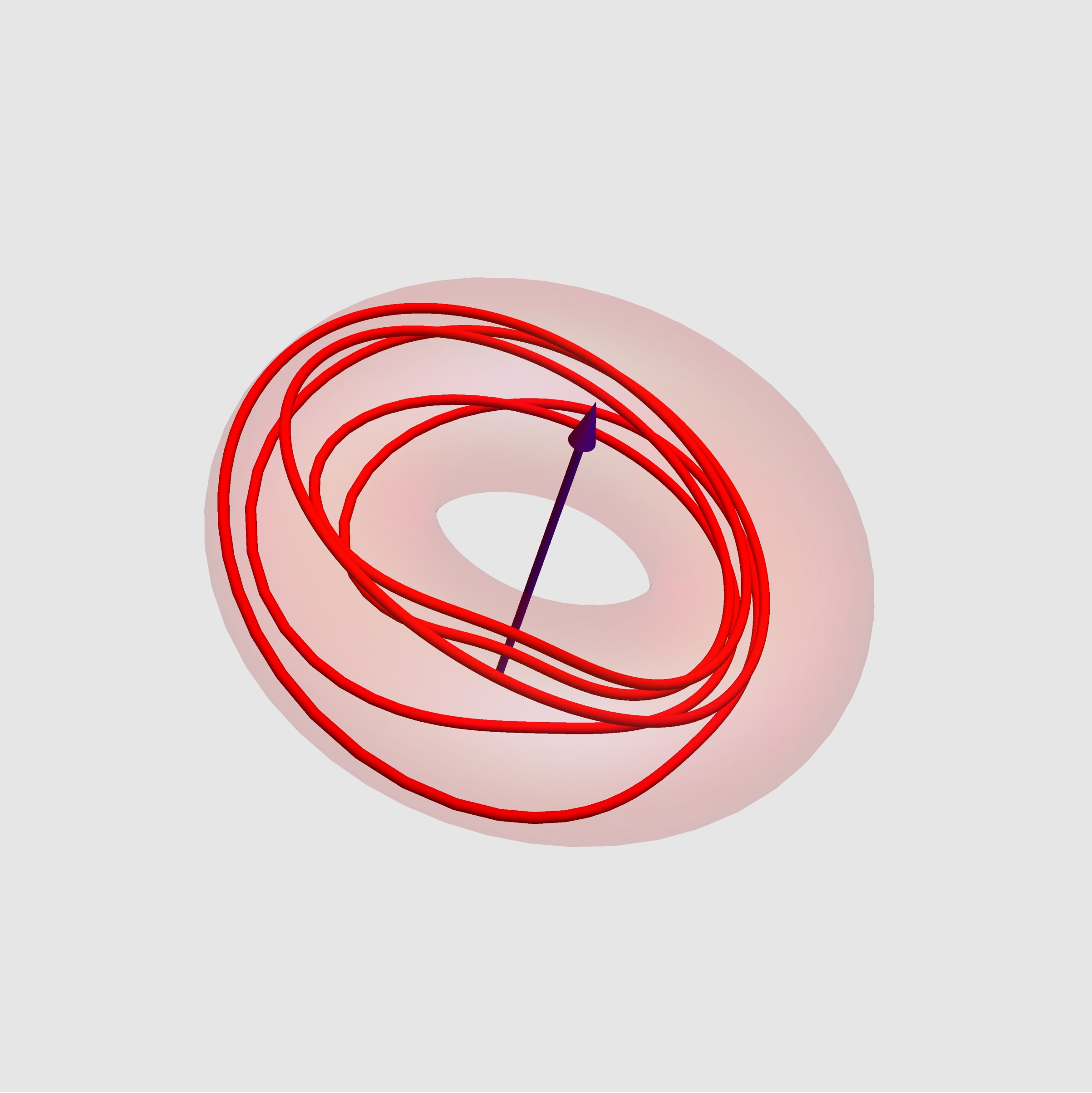}
\includegraphics[height=6.2cm,width=6.2cm]{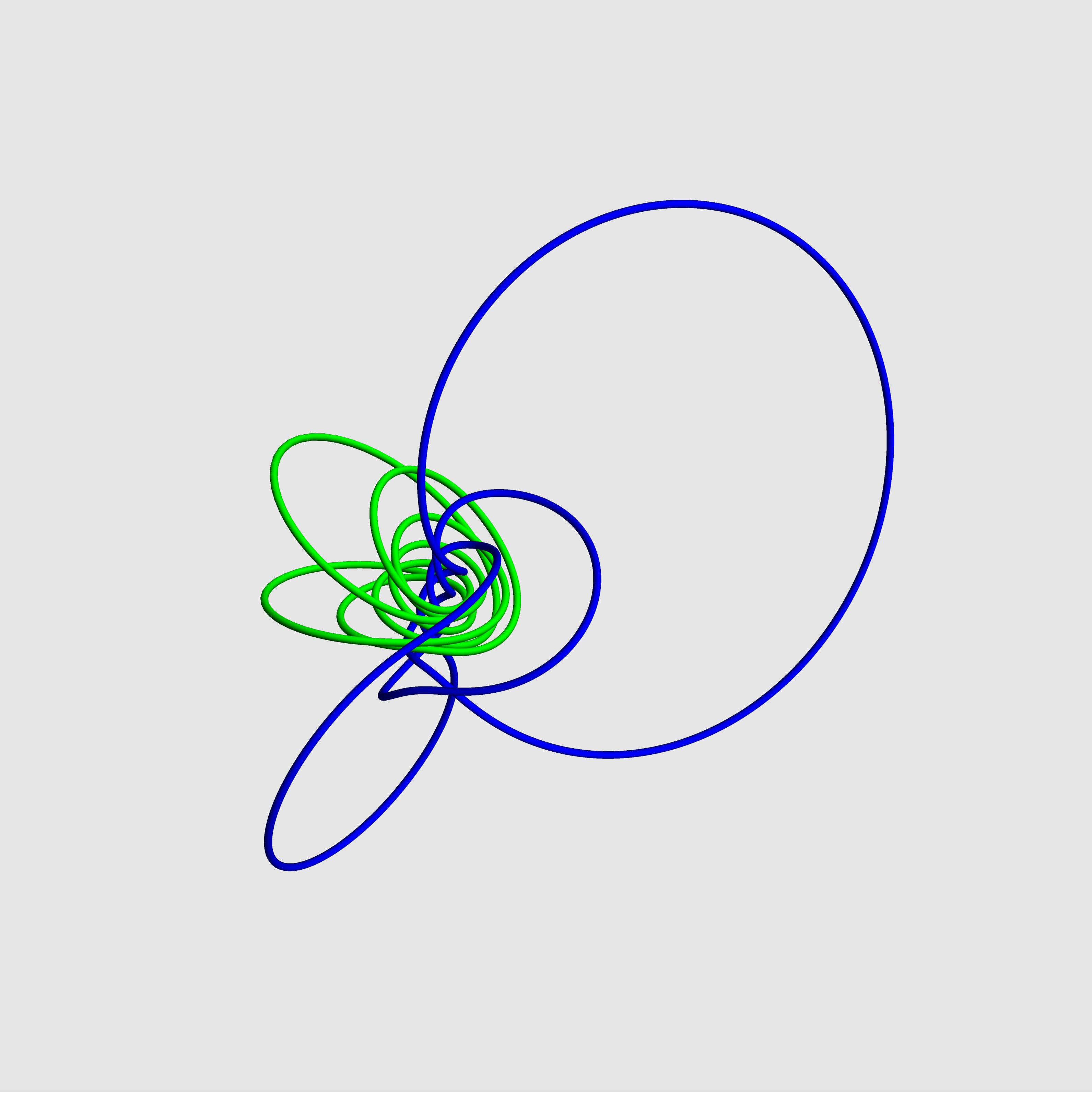}
\caption{A closed homogenoeus curve with $b=2/5$ trapped in the adS chamber (on the left) and the stereographic projections of its directrices from the point $(1,0,0,0)$.}\label{FIG6}
\end{center}
\end{figure}

\subsection{Directrices}
Let $\gamma:I\to \mathcal{E}^{1,2}$ be a generic timelike curve parametrized by the conformal
parameter and let $\mathbf{M}:I\to \mathcal{M}$ be its canonical frame.
Using the identification of $\mathcal{M}$ with $\mathrm{A}^{\uparrow}_+(2,3)$, let
$\widetilde{\mathbf{M}}:\R\to \mathrm{Sp}(4,\R)$ be a lift of $\mathbf{M}$ to $\mathrm{Sp}(4,\R)$.
Then, it follows that $\Gamma= |\widetilde{\mathbf{M}}_3|^{-1}\widetilde{\mathbf{M}}_3$ and $\Gamma^*= |\widetilde{\mathbf{M}}_4|^{-1}\widetilde{\mathbf{M}}_4$ are smooth maps into the 3-sphere, which
are referred to as the {\it directrices} of $\gamma$. Using the isomorphism $\varrho_*$, it
is easily seen that $\widetilde{\mathbf{M}}$ satisfies $\widetilde{\mathbf{M}}'=\widetilde{\mathbf{M}}\widetilde{\mathcal{K}}(h,k)$, where $k$ and $h$
are the conformal curvatures of $\gamma$ and $\widetilde{\mathcal{K}}(h,k)$ is
the $\mathfrak{sp}(4,\R)$-valued map
\begin{equation}\label{s6F1}
\widetilde{\mathcal{K}}(h,k)=
\left(
\begin{array}{cccc}
0 & -k/2 & -1/\sqrt{2} & 0 \\
k/2 & 0 & 0 & -1/\sqrt{2} \\
(h-1)/\sqrt{2} & 0 & 0 & -k/2 \\
0 & (h+1)/\sqrt{2} & k/2 & 0 \\
\end{array}
\right).
\end{equation}
This implies that the directrices are immersed curves of $S^3$, everywhere transverse to the
contact distribution of $S^3$. If $\gamma$ is a simple closed curve, with minimal period $\ell$, then
$\Gamma$ and $\Gamma^*$ are periodic, with minimal period $\widetilde{\ell}=\ell/\sigma^*$,
where $\sigma^*\in \{1,1/2\}$ is the {\it symplectic spin} of $\gamma$. Note that $\sigma^*$ is invariant
by regular timelike homotopies (i.e., through generic timelike curves). Consequently,
$\Gamma$ and $\Gamma^*$ originate smooth immersions of the circle $\R/\widetilde{\ell} \mathbb{Z}$
into the 3-sphere.  If the trajectories of $\Gamma$ and $\Gamma^*$ are simple and disjoint,
one can compute the linking number $\mathrm{Lk}(\Gamma,\Gamma^*)$ and the Bennequin numbers
$\mathbf{b}(\Gamma)$ and $\mathbf{b}(\Gamma^*)$.
These three integers provide global conformal invariants of $\gamma$, different in general
from the Maslov index of $\gamma$. The next Proposition explains how to compute these invariants for a
closed homogeneous timelike curve of the class $C_{2.i}$ considered in Section \ref{s:constant curvatures}
(see Figure \ref{FIG6}).

\begin{prop}
Let $\gamma$ be a closed homogeneous curve of the second class and of the first type,
with parameters $a\in (0,1)$ and $0<b=m/n<1$. Then, the following hold true:

\begin{enumerate}

\item $n$ is its Maslov index;

\item the curve has symplectic spin $1/2$;

\item if $p$ and $q$ are the relatively prime positive integers defined by $p/q=(n-m)/(n+m)$, then its directrices $\Gamma$ and $\Gamma^*$ are transversal positive torus knots of type $(q,p)$ such that $\mathrm{Lk}(\Gamma,\Gamma^*)=pq$ and $\mathbf{b}(\Gamma)=\mathbf{b}(\Gamma^*)=pq - (p+q)$.
\end{enumerate}
\end{prop}

\begin{proof}
(1) The assertion follows immediately from formula (\ref{C2ii}) of Theorem \ref{thm:reg-hom-cur}.

(2) Consider the parametrization of $\gamma$ given by the natural parameter so that, up to the action of the conformal group,
\[
 \gamma(u)=[\mathbf{M}^o\cdot \mathrm{Exp}(u \mathcal{K}(h,k))\cdot {}^t\!(1,0,0,0)],
  \]
where $h$ and $k$ are the constants defined as in (\ref{CurvC2}) and $\mathcal{K}(h,k)$ is the element of $\mathfrak{m}(2,3)$ defined by
$\mathcal{K}(h,k):= M^0_2-M^4_1- {k}M^2_3 - {h}M^0_1$.
This is a simple element of $\mathfrak{m}(2,3)$ with eigenvalues $0$, $\pm i\lambda_1$, $\pm i\lambda_2$, where
\[
 \lambda_1=\frac{\sqrt{b(1+b^2(a^2-1))}}{(1-a^2)^{1/4}\sqrt{1-b^2}},\quad
 \lambda_2=\frac{\sqrt{1+b^2(a^2-1)}}{(1-a^2)^{1/4}\sqrt{b(1-b^2)}}.
 \]
Since $\lambda_1/\lambda_2=b=m/n$, the minimal period of $\gamma$ (i.e., the total strain of the curve) is
\begin{equation}\label{per1}
\omega(a,m/n)=
 \frac{2\pi n}{\lambda_2}=2\pi\frac{(1-a^2)^{1/4}\sqrt{m(n^2-m^2)}}{\sqrt{n(n^2-(1-a^2)m^2)}}.
  \end{equation}
Let $\widetilde{\mathcal{K}}(h,k))\in \mathfrak{sp}(4,\R)$ be given as in (\ref{s6F1}). Then, putting
\[
 \mathrm{C}=\mathrm{Exp}(u \widetilde{\mathcal{K}}(h,k))\cdot \verb"e"_3,\quad \mathrm{C}^*=\mathrm{Exp}(u \widetilde{\mathcal{K}}(h,k))\cdot \verb"e"_4,
  \]
the two directrices of $\gamma$ are parameterized by
\[
 \Gamma : u\to \|\mathrm{C}(u)\|^{-1}\mathrm{C}(u),\quad
  \Gamma^* : u\to \|\mathrm{C}^*(u)\|^{-1}\mathrm{C}^*(u),
  \]
respectively. The matrix $\widetilde{\mathcal{K}}(h,k)$ is a simple element and has four purely imaginary distinct eigenvalues, $\pm i\widetilde{\lambda}_1$, $\pm i\widetilde{\lambda}_2$, where
\[
 \widetilde{\lambda}_1=\frac{\sqrt{(1+b)(1+(a^2-1)b^2)}}{2(1-a^2)^{1/4}\sqrt{b(1-b)}},
  \quad \widetilde{\lambda}_2=\frac{\sqrt{(1-b)(1+(a^2-1)b^2)}}{2(1-a^2)^{1/4}\sqrt{b(1+b)}}.
   \]
Thus, $\widetilde{\lambda}_1/\widetilde{\lambda}_2=(n+m)/(n-m)$. This implies that the two directrices
have minimal period $\omega_*(a,m/n)=2\pi(n+m)/\widetilde{\lambda}_2(n-m)=2\omega(a,m/n)$.
This proves that the symplectic spin of the curve is $1/2$.

(3) We first write the directrices of $\gamma$ in a canonical form. To this end, we consider the eigenvectors
$\verb"U"_1$ and $\verb"U"_2$ of the eigenvalues $-i\widetilde{\lambda}_1$ and $-i\widetilde{\lambda}_2$,
respectively, and denote by $\verb"V"_j$ and $\verb"W"_j$, $j=1,2$, their real and imaginary parts.
Then, the matrix $\verb"B"$ with column vectors
\[
\begin{array}{ll}
\verb"B"_1=\verb"V"_1/\sqrt{-\omega(\verb"V"_1,\verb"W"_1)},& \verb"B"_2=\verb"V"_2/\sqrt{-\omega(\verb"V"_2,\verb"W"_2)}, \\ \verb"B"_3=\verb"W"_1/\sqrt{-\omega(\verb"V"_1,\verb"W"_1)}, & \verb"B"_4=\verb"W"_2/\sqrt{-\omega(\verb"V"_2,\verb"W"_2)}\\
\end{array}
\]
belongs to $\mathrm{Sp}(4,\R)$ and a direct computation shows that
\begin{equation}\label{STK1}
\begin{cases}
\verb"B"\cdot \Gamma(u) =\left(\frac{-2A\sin(qu)}{1+A^2},\frac{(1-A^2)sin(pu)}{1+A^2},\frac{2A\cos(qu)}{1+A^2},
\frac{(A^2-1)\cos(pu)}{1+A^2}\right),\\
\verb"B"\cdot \Gamma^*(u)=\left(\frac{(1-A^2)\cos(qu)}{1+A^2},\frac{2A\cos(pu)}{1+A^2},\frac{(1-a^2)\sin(qu}{1+A^2},
\frac{2A\sin(pu)}{1+A^2}\right),
\end{cases}
\end{equation}
where
$$
  A=\frac{\sqrt{1-\sqrt{1+a^2}}}{\sqrt{2}-\sqrt{1+\sqrt{1+a^2}}}\in (1+\sqrt{2},+\infty).
  $$
Since all our considerations are invariant under the action of the symplectic group,
we may assume that the directrices are the transversal knots $\widehat{\mathrm{K}}_{A,p,q}$ and $\widehat{\mathrm{K}}^*_{A,p,q}$ of $S^3$ parameterized by the maps $\Gamma_{A,p,q}$ and $\Gamma_{A,p,q}$
in the left hand side of (\ref{STK1}). Since the ``point at infinity'' $\verb"P"={}^t\!(0,0,0,1)$
is not contained in by $\widehat{\mathrm{K}}_{A,p,q}\cup \widehat{\mathrm{K}}^*_{A,p,q}$,
we may consider their images $\mathrm{K}_{A,p,q}$ and $\mathrm{K}^*_{A,p,q}\subset \R^3$ under
the conformal map $\mathrm{st}:S^3\setminus\{\verb"P"\}\to \R^3$.\footnote{The map $\mathrm{st}$ is the
composition of the stereographic projection form the point at infinity and the reflection about the $xy$-plane
of $\R^3$.} It is now an easy task to check that the link $(\mathrm{K}_{A,p,q},\mathrm{K}^*_{A,p,q})$ is isotopic
to $(\mathrm{K}_{p,q}, \mathrm{K}^*_{p,q})$, where
$\mathrm{K}_{p,q}$ and $\mathrm{K}^*_{p,q}$ are the torus knots defined as in (\ref{STK}). This implies
\[
 \mathrm{Lk}(\widehat{\mathrm{K}}_{A,p,q},\widehat{\mathrm{K}}^*_{A,p,q})=
  \mathrm{Lk}(\mathrm{K}_{A,p,q},\mathrm{K}^*_{A,p,q})=\mathrm{Lk}(\mathrm{K}_{p,q}, \mathrm{K}^*_{p,q})=pq.
   \]
For fixed integers $p,q$, the knots $\widehat{\mathrm{K}}_{A,p,q}$, $A>1$, are transversal to the contact distribution and transversally isotopic to each other. Thus, the Bennequin numbers do no depend on the parameter $A$.
Let $(\overrightarrow{\mathbf{E}}_1,\overrightarrow{\mathbf{E}}_2)$ be the parallelization of the contact distribution of $S^3$ defined as in (\ref{TRCNT}) and let $\overrightarrow{\mathbf{Y}}_{p,q}$ be the vector field
\[
 \overrightarrow{\mathbf{Y}}_{p,q}:\verb"P" (x_1,x_2,x_3,x_4)\mapsto (-qx_3,-px_4 , qx_1,px_2 )\in \mathrm{T}_{\verb"P"}(S^3).
  \]
Then, $(\overrightarrow{\mathbf{Y}}_{p,q},\overrightarrow{\mathbf{E}}_1,\overrightarrow{\mathbf{E}}_2)$ is a parallelization of the 3-sphere and $\widehat{\mathrm{K}}_{A,p,q}$ is the trajectory of
$\overrightarrow{\mathbf{Y}}_{p,q}$ passing through the point ${}^t\!(0,0,2A/(1+A^2),(A^2-1)/(1+A^2))$. Since $\overrightarrow{\mathbf{Y}}_{p,q}$ is tangent to $\widehat{\mathrm{K}}_{A,p,q}$ and transversal to the contact distribution, the vector field
$$
 \overrightarrow{\mathbf{Z}}_{p,q}= \overrightarrow{\mathbf{E}}_1-
 \frac{\overrightarrow{\mathbf{Y}}_{p,q}\cdot \overrightarrow{\mathbf{E}}_1}{\|\overrightarrow{\mathbf{Y}}_{p,q}\|^2}\overrightarrow{\mathbf{Y}}_{p,q}
  $$
is normal to $\widehat{\mathrm{K}}_{A,p,q}$ and, in addition, $\overrightarrow{\mathbf{E}}_2, \overrightarrow{\mathbf{Z}}_{p,q}$ span a 2-dimensional distribution everywhere transverse to $\widehat{\mathrm{K}}_{A,p,q}$. Thus, the vector fields
$\overrightarrow{\mathbf{X}}_{p,q}=\mathrm{st}_*(\overrightarrow{\mathbf{Z}}_{p,q})$ and $\overrightarrow{\mathbf{V}}_{p,q}=\mathrm{st}_*(\overrightarrow{\mathbf{E}}_2)$ generate a plane field distribution transverse to $\mathrm{K}_{A,p,q}$. Taking into account that $\mathrm{st}$ is a conformal map, the vector field $\overrightarrow{\mathbf{X}}_{p,q}$ is normal to $\mathrm{K}_{A,p,q}$.
Choosing $\epsilon>0$ sufficiently small, the image of the map
$$
 (p,\theta)\in \mathrm{K}_{A,p,q}\times \R\mapsto p + \epsilon(\cos(\theta)\overrightarrow{\mathbf{X}}_{p,q}+\sin(\theta)\overrightarrow{\mathbf{V}}_{p,q})\in \R^3
   $$
is an embedded torus whose inner region contains $\mathrm{K}_{A,p,q}$. Since the map
$$
 (p,r)\in \mathrm{K}_{A,p,q}\times [0,1]\mapsto p + \epsilon(\cos(r\pi/2)\overrightarrow{\mathbf{X}}_{p,q}+\sin(r\pi/2)\overrightarrow{\mathbf{V}}_{p,q})
   $$
is a smooth homotopy (in $\R^3\setminus\mathrm{K}_{A,p,q}$) between $\mathrm{K}_{A,p,q}+\epsilon \overrightarrow{\mathbf{V}}_{p,q}$ and $\mathrm{K}_{A,p,q}+\epsilon \overrightarrow{\mathbf{X}}_{p,q}$,
we obtain
\[
 \mathbf{b}(\widehat{\mathrm{K}}_{A,p,q})=
\mathrm{Lk}_{\widehat{\mathrm{K}}_{A,p,q}}(\overrightarrow{\mathbf{E}}_2)=
\mathrm{Lk}_{\mathrm{st}(\widehat{\mathrm{K}}_{A,p,q})}(\mathrm{st}_*(\overrightarrow{\mathbf{E}}_2))=
\mathrm{Lk}_{\mathrm{K}_{A,p,q}}(\overrightarrow{\mathbf{V}}_{p,q})=
\mathrm{Lk}_{\mathrm{K}_{A,p,q}}(\overrightarrow{\mathbf{X}}_{p,q}).
 \]
By construction, the vector field $\overrightarrow{\mathbf{X}}_{p,q}$ is a positive multiple of the normal projection $\overrightarrow{\mathbf W}_{p,q}$ of $\mathrm{st}_*(\overrightarrow{\mathbf E}_1)$ along $\mathrm{K}_{A,p,q}$ so that
\[
 \mathrm{Lk}_{\mathrm{K}_{A,p,q}}(\overrightarrow{\mathbf W}_{p,q})=
  \mathrm{Lk}_{\mathrm{K}_{A,p,q}}(\overrightarrow{\mathbf{X}}_{p,q}).
   \]
This implies that
$\mathbf{b}(\widehat{\mathrm{K}}_{A,p,q})$ coincides
with $\mathrm{Lk}_{\mathrm{K}_{A,p,q}}(\overrightarrow{\mathbf W}_{p,q})$.
The latter is equal to
$$
  \mathrm{SL}(\mathrm{K}_{A,p,q})+\Theta_{\mathrm{K}_{A,p,q}}(\overrightarrow{\mathbf W}_{p,q}).
    $$
Since everything is independent of the parameter $A$, we assume $A=3$.
Then, $\widehat{\mathrm{K}}_{A,p,q}$ is the torus knot parameterized by the map $\check{\Gamma}_{p,q}$ defined as in \eqref{TK3}. From this we get $\mathrm{SL}(\mathrm{K}_{A,p,q})=pq-p$. On the other hand, $\Theta_{\mathrm{K}_{A,p,q}}(\overrightarrow{\mathbf W}_{p,q})$ is the degree of the map
\[
 \phi :  \R/2\pi \mathbb{Z}\ni [u]\mapsto (\overrightarrow{\mathbf W}_{p,q}|_u\cdot \overrightarrow{\mathbf P}_n|_u,
  \overrightarrow{\mathbf W}_{p,q}|_u\cdot \overrightarrow{\mathbf P}_b|_u)\in \R^2\setminus\{(0,0)\},
  \]
where $\overrightarrow{\mathbf P}_b|_u = \check{\Gamma}'_{p,q}\times \check{\Gamma}''_{p,q}|_u$ and
$\overrightarrow{\mathbf P}_n|_u = -\check{\Gamma}'_{p,q}\times \overrightarrow{\mathbf P}_b|_u$.
The components of $\phi$ can be computed by any software of symbolic computation.
As a result $\phi$ can be written as $\phi([u])=(\mathrm{S}_1(u)+\mathrm{S}_2(u))\, {}^t\!(\cos(qu),\sin(qu))$, where $S_1$ is a off-diagonal $2\times 2$ matrix with positive entries and $\mathrm{S}_2$ is a periodic map into $\mathbb{R}(2,2)$ of period $2\pi$ such that
\[
  \mathrm{det}(S_1(u)+rS_2(u))<0,\quad \forall u\in \R,\,\,r\in [0,1].
    \]
Consequently, $\phi$ is homotopic to $\R/2\pi\mathbb{Z} \ni [u]\mapsto {}^t\!(\cos(qu),-\sin(qu))$, and hence
\[
\Theta_{\mathrm{K}_{A,p,q}}(\overrightarrow{\mathbf W}_{p,q})=-q.
   \]
This proves that $\mathbf{b}(\widehat{\mathrm{K}}_{A,p,q})=pq-p-q$, as claimed.

It is now easy to check that $\Gamma_{\sqrt{3},p,q}$ and $\Gamma^*_{\sqrt{3},p,q}$ are congruent to each other
with respect to the symplectic group. In particular, $\widehat{\mathrm{K}}_{\sqrt{3},p,q}$ and $\widehat{\mathrm{K}}^*_{\sqrt{3},p,q}$ have the same Bennequin number. Taking into account that the Bennequin number is independent of the parameter $A$, provided that $A>1$, we obtain
\[
  \mathbf{b}(\Gamma^*_{A,p,q})=\mathbf{b}(\Gamma^*_{\sqrt{3},p,q})
    =\mathbf{b}(\Gamma_{\sqrt{3},p,q})=pq-p-q,
     \]
which proves the required result.
\end{proof}

\begin{remark}
The proposition above implies
as a consequence that the conformal equivalence class of a closed homogeneous timelike curve of type $C_{2i}$ is determined by a geometrical invariant, the total strain, and two topological invariants of its directrices, the Bennequin number and the linking number. In addition, the proposition shows that the directrices of the homogeneous knots of the class $C_{2i}$ are representatives of the contact isotopy classes of transversal torus knots of type $(p,q)$ with maximal Bennequin number \cite{Et,EtHo}.
\end{remark}

\bibliographystyle{amsalpha}

\begin{thebibliography}{AA}


\bibitem{AD}
{C. Alvarez Paiva and C. E. Dur\'an}, Geometric invariants of fanning curves,
{\em Adv. in Appl. Math.}, {\bf 42} (13) (2009), 290-312.

\bibitem{Ba}
{T. F. Banchoff},
Osculating tubes and self-linking for curves on the three-sphere,
{\em Cont. Math.}, (2001).

\bibitem{BCDG}
{T. Barbot, V. Charette, T. Drumm, W. M. Goldman, and K. Melnick},
{A primier on the (2+1)-Einstein universe}, in
{\em Recent developments in pseudo-Riemannian geometry}, 179--229,
ESI Lect. Math. Phys., Eur. Math. Soc., Z\"urich, 2008; {arXiv:0706.3055[math.DG]}

\bibitem{Be}
{V. Bergmann},
Irreducible unitary representations of the Lorentz group,
{\em Ann. Math.} {\bf 48} (1947), 568--640.


\bibitem{Blaschke}
{W. Blaschke},
{\em Vorlesungen \"uber Differentialgeometrie. III:
Differentialgeometrie der {K}reise und Kugeln},
Grundlehren der mathematischen Wissenschaften, 29, Springer, Berlin, 1929.



\bibitem{Ca}
{G. C{\u{a}}lug{\u{a}}reanu},
L'int\'egral de Gauss et l'analyse des noeuds tridimensionnels,
{\em  Rev. Math. Pures Appl.} \textbf{4} (1959), 5--20.


\bibitem{Ce}
T.E. Cecil, {\em Lie sphere geometry{:} with applications to
submanifolds}, Springer-Verlag, New York, 1992.



\bibitem{CW}
S.-S. Chern and H.-C. Wang,
Differential geometry in symplectic space I, {\em Sci. Rep. Nat. Tsing Hua Univ.}
\textbf{4} (1947), 453--477.

\bibitem{CGP-BAMS}
P. T. Chru\'sciel, G. J. Galloway, D. Pollack,
Mathematical general relativity: a sampler,
{\em Bull. Amer. Math. Soc. (N.S.)} \textbf{47}
(2010), 567--638.


\bibitem{DNF}
{B. A. Dubrovin, A. T. Fomenko, and S. P. Novikov},
\textit{Modern geometry--methods and applications. Part I},
2nd edn., GTM 93. Springer-Verlag, New York, 1992.


\bibitem{Einstein}
{A. Einstein},
Kosmologische Betrachtungen zur allgemeinen Relativit\"atstheorie,
{\em Sitzungsberichte der K\"oniglich Preussischen Akademie der Wissenschaften},
142--152, Berlin, 1917.

\bibitem{Et}
{J. B. Etnyre},
Transversal torus knots, {\em Geom. Topol.} {\bf 3} (1999), 253--268.

\bibitem{Et2}
{J. B. Etnyre},
{\em Legendrian and transveral knots},
Hanbook of knot theory, 105--185, Elsevier B. V., Amsterdam, 2005.

\bibitem{EtHo}
{J. B. Etnyre and K. Honda},
Knots and contact geometry I: torus knots and the figure eight knot,
{\em  J. Symplectic Geom.} {\bf 1} (2001), 63--120.

\bibitem{FGL}
A. Ferr\'andez, A. Gim\'enez, P. Lucas, Geometrical particle models
on 3D null curves, \textit{Phys. Lett. B} \textbf{543} (2002),
311--317; hep-th/0205284.


\bibitem{Fr}
{C. Frances}, \textit{G\'eometrie et dynamique lorentziennes conformes}, Th\'ese, E.N.S. Lyon
(2002).


\bibitem{FT}
{D. Fuchs, S. Tabachnikov},
Invariants of Legendrian and transverse knots in the standard contact space,
{\em Topology} {\bf 36} (1997), no. 5, 1025--1053.

\bibitem{Fu}
{F. B. Fuller},
The writhing number of a space curve,
{\em Proc. Nat. Acad. Sci. U.S.A.} {\bf 68} (1971), 815--819.

\bibitem{GM}
{J. D. Grant and E. Musso},
Coisotropic variational problems,
{\em J. Geom. Phys.} {\bf 50} (2004), 303--338.

\bibitem{Gr}
{P. A. Griffiths},
\textit{Exterior differential systems and the calculus of variations},
Progress in Mathematics, 25, Birkh\"auser, Boston, 1982.


\bibitem{GS3}
{V. Guillemin and S. Sternberg},
\textit{Variations on a theme by Kepler},
AMS Colloquium Publications, \textbf{285}, Providence, RI, 1990.


\bibitem{HE}
{S. W. Hawking and G. F. R. Ellis},
\textit{The large scale structure of space-time},
Cambridge Monographs on Mathematical Physics, no. 1,
Cambridge University Press, London-New York, 1973.

\bibitem{KP}
Y. A. Kuznetsov, M. S. Plyushchay,
(2+1)-dimensional models of relativistic particles with curvature and torsion,
\textit{J. Math. Phys.} \textbf{35} (1994), no. 6, 2772-2778.


\bibitem{LO2010}
{R. Langevin and J. O'Hara},
Conformal arc-length as $\frac{1}{2}$-dimensional length of the set of osculating circles,
\textit{Comment. Math. Helv.} \textbf{85} (2010), no. 2, 273--312.



\bibitem{MMR}
{M. Magliaro, L. Mari, and M. Rigoli},
 On the geometry of curves and conformal geodesics in the M\"obius space,
 \textit{Ann. Global Anal. Geom.} \textbf{40} (2011), 133-165.





\bibitem{Mar}
{C. M. Marle},
A property of conformally Hamiltonian vector fields; application to the Kepler problem,
arXiv:1011.5731v2[math.SG].

\bibitem{Mi}
{J. Milnor},
On the geometry of the Kepler problem,
\textit{Amer. Math. Monthly} \textbf{90} (1983), 353--365.

\bibitem{Mon}
{A. Montesinos Amilibia, M. C. Romero Fuster, and E. Sanabria},
Conformal curvatures of curves in $\R^{n+1}$,
\textit{Indag. Math. (N.S.)} \textbf{12} (2001), 369--382.

\bibitem{Mos}
{J. Moser},
Regularization of Kepler's problem and the averaging method on a manifold,
\textit{Comm. Pure Appl. Math.} \textbf{23} (1970), 609--636.

\bibitem{M1}
{E. Musso},
The conformal arclength functional,
\textit{Math. Nachr.} \textbf{165} (1994), 107--131.

\bibitem{M2}
{E. Musso},
Closed trajectories of the conformal arclength functional,
{\em Journal of Physics: Conference Series} \textbf{410} (2013), 012031.

\bibitem{MN1}
{E. Musso and L. Nicolodi},
Closed trajectories of a particle model on null curves in anti-de Sitter 3-space,
\textit{Classical Quantum Gravity} \textbf{24} (2007), no. 22, 5401--5411.

\bibitem{MN2}
{E. Musso and L. Nicolodi},
Reduction for constrained variational problems on 3-dimensional null curves,
\textit{SIAM J. Control Optim.} \textbf{47} (2008), no. 3, 1399-1414.

\bibitem{MN}
{E. Musso and L. Nicolodi},
Quantization of the conformal arclength functional on space curves,
\textit{Comm. Anal. Geom.} (to appear);
arXiv:1501.04101[math.DG].

\bibitem{NFS2}
{V. V. Nesterenko, A. Feoli, and G. Scarpetta},
Complete integrability for Lagrangians dependent on accelaration
in a spacetime of constant curvature,
\textit{Classical Quantum Gravity} \textbf{13} (1996), 1201--1211.

\bibitem{NMMK}
{A. Nersessian, R. Manvelyan, and H. J. W. M\"uller-Kirsten},
Particle with torsion on 3d null-curves,
\textit{Nuclear Phys. B} \textbf{88} (2000), 381--384; arXiv:hep-th/9912061

\bibitem{MR}
{H. K. Moffatt and R. L. Ricca},
Helicity and the C{\u{a}}lug{\u{a}}reanu invariant,
{\em Proc. Roy. Soc. London Ser. A} \textbf{439} (1992), no. 1906, 411--429.

\bibitem{NR}
{A. Nersessian and E. Ramos},
Massive spinning particles and the geometry of null curves,
\textit{Phys. Lett. B} \textbf{445} (1998), 123-–128.

\bibitem{Oh}
{J. O'Hara},
\textit{Energy of knots and conformal geometry},
Series on Knots and Everything, 33, World Scientific Publishing Co., Inc., River Edge, NJ, 2003.


\bibitem{Pn1}
{R. Penrose},
\textit{Cycles of time: An extraordinary new view of the universe},
Alfred A. Knopf, Inc., New York, 2010.

\bibitem{Pn3}
{R. Penrose},
On the gravitization of quantum mechanics 2: Conformal cyclic cosmology,
\textit{Found. Phys.} \textbf{44} (2014), 873-890.

\bibitem{PR}
{R. Penrose and W. Rindler},
\textit{Spinors and space-time},
Camridge monographs on Mathematical Physics, Vol. I and II, Cambridge University Press, 1986.

\bibitem{P}
{R. D. Pisarski},
Field theory of paths with a curvature-dependent term,
\textit{Phys. Rev. D} \textbf{34} (1986), no. 2, 670--673.

\bibitem{Po}
{W. F. Pohl},
The self-linking number of a closed space curve,
\textit{J. Math. Mech.} \textbf{17} (1968), 975--985.

\bibitem{Raw}
{J. Rawnsley},
On the universal covering group of the real symplectic group,
\textit{Journal Geom, and Phys.} \textbf{62} (2012), 2044-2058.

\bibitem{RN}
{R. L. Ricca and B. Nipoti},
Gauss' linking number revisited,
{\em J. Knot Theory Ramifications} \textbf{20} (2011), no. 10, 1325--1343.

\bibitem{SS}
{C. Schiemangk and R. Sulanke},
Submanifolds of the M\"obius space,
\textit{Math. Nachr.} \textbf{96} (1980), 165--183.

\bibitem{Scho}
{M. Schottenloher},
A mathematical introduction to conformal field theory, 2nd edn., Lecture Notes in Physics, 759,
Springer-Verlag, Berlin, Heidelberg, 2008.

\bibitem{S}
{R. Sulanke},
Submanifolds of the M\"obius space II, Frenet formula and curves of constant curvatures,
\textit{Math. Nachr.} \textbf{100} (1981), 235--257.



\bibitem{Tod}
{P. Tod},
Penrose's Weyl curvature hypothesis and conformally-cyclic cosmology,
{\em Jornal of Physics: Conference Series} \textbf{229} (2010), 1--5.

\bibitem{Ur}
{H. Urbantke},
Local differential geometry of null curves in conformally flat space-time,
{\em J. Math. Phys.} \textbf{30} (10), (1989), 2238--2245.


\bibitem{W}
{H. Weyl},
\textit{Time, Space, Matter}
Dover Publications, Inc., Mineola, N.Y., 1952.










\end{thebibliography}

\end{document}